%% file: main.tex
\documentclass[12pt,reqno]{amsart}

\usepackage[text={420pt,660pt},centering]{geometry}
\usepackage{amssymb,amsfonts,amsthm}
\usepackage{comment}
\usepackage{enumitem}
\usepackage{bm}
\usepackage{bbm}

\usepackage{xfrac}

\usepackage[dvipsnames]{color}

\usepackage[colorlinks=true, pdfstartview=FitV, linkcolor=black, citecolor=blue, urlcolor=blue]{hyperref}

\usepackage{caption}

\def\Xint#1{\mathchoice
{\XXint\displaystyle\textstyle{#1}}%
{\XXint\textstyle\scriptstyle{#1}}%
{\XXint\scriptstyle\scriptscriptstyle{#1}}%
{\XXint\scriptscriptstyle%
\scriptscriptstyle{#1}}%
\!\int}
\def\XXint#1#2#3{{\setbox0=\hbox{$#1{#2#3}{%
\int}$ }
\vcenter{\hbox{$#2#3$ }}\kern-.6\wd0}}
\def\barint{\, \Xint -} 
\def\bariint{\barint_{} \kern-.4em \barint}
\def\bariiint{\bariint_{} \kern-.4em \barint}
\renewcommand{\iint}{\int_{}\kern-.34em \int} 
\renewcommand{\iiint}{\iint_{}\kern-.34em \int} 

\DeclareMathAlphabet{\mathcal}{OMS}{cmsy}{m}{n}

\theoremstyle{plain}

\newtheorem{theorem}{Theorem}[section]

\newtheorem{lemma}[theorem]{Lemma}

\newtheorem{corollary}[theorem]{Corollary}
\newtheorem{proposition}[theorem]{Proposition}

\newtheorem{assumption}[theorem]{Assumption}

\theoremstyle{definition}
\newtheorem{remark}[theorem]{Remark}

\newcommand{\R}{\mathbb{R}}
\newcommand{\C}{\mathbb{C}}
\newcommand{\N}{\mathbb{N}}
\newcommand{\Z}{\mathbb{Z}}
\newcommand{\T}{\mathbb{T}}

\newcommand{\p}{\partial}

\newcommand{\les}{\lesssim}

\newcommand{\norm}[1]{\lVert #1 \rVert}

\newcommand{\D}{D_\lambda}
\newcommand{\A}{A_\epsilon}
\newcommand{\ai}{A_\varepsilon^{-1}}

\newcommand{\lb}{\langle}
\newcommand{\rb}{\rangle}

\let\div\relax
\DeclareMathOperator{\div}{div}

\DeclareMathOperator{\sgn}{sgn}
\let\tilde\relas
\newcommand{\tilde}[1]{\widetilde{#1}}

\let\Re\relax
\DeclareMathOperator{\Re}{Re}
\let\Im\relax
\DeclareMathOperator{\Im}{Im}

\numberwithin{equation}{section}
\numberwithin{figure}{section}

\title[]{Sharp uniform-in-diffusivity mixing rates for passive scalars in parallel shear flows}

\author[Albritton]{Dallas Albritton}
\address[Dallas Albritton]{University of Wisconsin-Madison, Department of Mathematics,  480 Lincoln Dr, Madison, WI 53706, USA}
\email{dalbritton@wisc.edu}

\author[Beekie]{Rajendra Beekie}
\address[Rajendra Beekie]{Imperial College London, Department of Mathematics,  180 Queen's Gate, London, SW7 2AZ}
\email{r.beekie@imperial.ac.uk}

\date{\today}

\begin{document}

\begin{abstract}
We consider the advection-diffusion equation describing the evolution of a passive scalar in a background shear flow. We prove the optimal uniform-in-diffusivity mixing rate $\| f \|_{H^{-1}} \lesssim \langle t \rangle^{-1/(N+1)}$, $t \geq 0$, where $N$ is the maximal order of vanishing of the derivative $b'(y)$ of the shear profile, e.g., $N=1$ for plane Pouseille flow. Our proof is based on the description of the solution in terms of resolvents and involves pointwise estimates on the resolvent kernel. In the non-degenerate case, we further give a rigorous asymptotic description of generic solutions in terms of shear layers localized 
around the critical points. This verifies formal asymptotics in~\cite{Camassa2010}.
\end{abstract}





\maketitle

\tableofcontents

\parskip   2pt plus 0.5pt minus 0.5pt

\input{introduction.tex}
\input{mixingviaresolvent.tex}

\input{efunctiongluing.tex}

\input{Conclusion.tex}

\appendix
\numberwithin{theorem}{section}

\input{funsolbds.tex}

\input{monotoneresolventestimate.tex}


\subsubsection*{Acknowledgments}

DA was supported by NSF Grant No. 2406947 and the Office of the Vice Chancellor for Research and Graduate Education at the University of Wisconsin–Madison with funding from the Wisconsin Alumni Research Foundation. RB was partially supported by NSF Grant No. 2202974 and ERC-EPSRC Horizon Europe Guarantee EP/X020886/1. We thank Tarek Elgindi, Gautum Iyer, Hao Jia, and Kyle Liss for valuable discussions.

\bibliographystyle{plain}
\bibliography{Bibliodep}

\end{document}

%% file: introduction.tex


\section{Introduction}





Mixing is the process in which incompressible flows homogenize a passive scalar through the rearrangement of particles. While mixing is primarily driven by transport, in reality the scalar is also subject to molecular diffusion, however small.

In this paper, we quantify mixing in the presence of small diffusivity. We focus on simple examples, namely, parallel shear flows, whose study is foundational to the area. Our underlying model is
\begin{equation}
\label{adv:diff}
    \p_t f + b(y)\p_x f = \kappa \Delta f \, ,
\end{equation}
where $b(y)$ is the shear profile and $\kappa > 0$ is the molecular diffusivity. 

The interaction between the effects of diffusion and shearing in \eqref{adv:diff} has been extensively studied in the mathematical fluids literature. A central focus has been understanding \emph{enhanced dissipation}, or the decay on 
time scales shorter than the diffusive time scale $\kappa^{-1}$. 
Despite the significant progress made in understanding~\eqref{adv:diff}, there remain fundamental questions, particularly in regard to understanding the dynamics of~\eqref{adv:diff} \emph{uniformly-in-$\kappa$}.

To illustrate what can be expected, we consider the zero diffusivity case $\kappa = 0$, which we sometimes call ``inviscid." The passive scalar $f(t,x ,y): [0, \infty) \times \T \times D \to \R $ ($D = \R $ or $\T := \R / 2\pi \Z$) evolves according to
\begin{equation}
\label{int1}
    \p_t f + b(y) \p_x f = 0 \, .
\end{equation}
We analyze the evolution mode-by-mode in $x$:
\begin{equation}
\label{int2}
    \p_t f_k + ib(y) k f_ k = 0 \, ,
\end{equation}
from which we deduce the solution formula 
\begin{equation}
\label{int3}
    f_k(t,y) = e^{-ib(y) k t} f_k^{\rm in} \, .
\end{equation}
We measure mixing by integrating against observables $\phi$. 
The method of (non-)stationary phase then yields a precise asymptotic description of the oscillatory integral
\begin{equation}
\label{int4}
    \int f_k(t,y) \phi(y) \, dy = \int e^{-ib(y) k t} f_k^{\rm in} \phi(y) \, dy \, .
\end{equation}
Let $k \neq 0$. If $b$ has no critical points and $\phi$ is smooth, then, under mild background assumptions, the above integral is $O(t^{-\infty})$: It vanishes to any polynomial order. On the other hand, if $b$ has a non-degenerate critical point, then no matter how smooth $\phi$ is, the integral decays like $t^{-1/2}$ (unless $f^{\rm in} \phi$ vanishes at the critical point). For shear profiles possessing critical points $\gamma$ of order $N$, i.e., $b^{(j)}(\gamma) = 0$, $j=1,\hdots, N$, and $b^{(N+1)}(\gamma) \neq 0$, it is known that the inviscid decay rates
\begin{equation}
    \label{eq:optimalinvisid}
    \tag{$\ast$}
    \| f_k \|_{H^{-1}} \lesssim \langle kt \rangle^{-\frac{1}{N+1}} \| f_k^{\rm in} \|_{H^1}
\end{equation}
are optimal.

This begs a simple question:
\begin{quote}
    \textbf{(Q) What mixing estimates hold with non-zero diffusivity?}
\end{quote}
In particular, does the estimate~\eqref{eq:optimalinvisid} continue to hold?

Clearly, the inviscid method, based on the exact solution formula~\eqref{int3}, fails. Moreover, even at a conceptual level, the result is not obviously true: It is known that \emph{diffusion limits mixing}~\cite{Miles2018}
in that mixing is due to the creation of small scales, but it is conjectured that diffusion creates an effective minimal length scale (sometimes known as the \emph{Batchelor scale}), cf. discussion in~\cite[Remark~1.4]{CotiZelati2023}. For example, in the sinusoidal shear flow $b(y) = \sin y$, the minimal length scale is observed to be $\sim \kappa^{1/4}$. 
By contrast, a long-time length scale is absent specifically in monotone shear flows, as in the explicit Fourier calculation for the Couette flow $b(y) = y$ on $\R$, where the typical length scale is $\sim t^{-1}$ as $t \to +\infty$~\cite{kelvin1887stability}. The $\langle t \rangle^{-1}$ rate for monotone shears was demonstrated in~\cite{CZ19}, leaving the non-monotone case open.

In this paper, we clarify the picture in two ways.
\begin{enumerate}[leftmargin=*,label=(\Roman*)]
    \item We answer $\textbf{(Q)}$ by proving~\eqref{eq:optimalinvisid} \emph{uniformly-in-$\kappa$} for arbitrary $N \geq 1$.
    \item We present a precise characterization of the asymptotic behavior of the passive scalar in the non-degenerate case $N=1$: The solutions exhibit decaying traveling wave structures localized to length scale $\kappa^{1/4}$ around the critical points. Such structures were predicted to exist in the applied literature~\cite{Camassa2010}, and their existence was posed as Open Problem~3 in~\cite[Section~4.4]{Elgindi2022Lectures}.
\end{enumerate}

Thus, the estimates~\eqref{eq:optimalinvisid} hold \emph{and} the effective minimal length scale exists. 



\subsection{Optimal mixing rates} 

First, we state the assumptions on the shear profile $b : \T \to \R$.

\begin{assumption}[Assumptions on $b$]
\label{assump:b}
There exists $\sigma_{\sharp} \in (0,1)$ and $N \in \N$ such that $b \in C^{N+2}(\T)$,
\begin{equation}
    \norm{b}_{C^{N+2}} \leq \sigma_\sharp^{-1} \, ,
\end{equation}
 and $b(y)$ has finitely many critical points $\gamma_j$, $j = 1, \hdots, J_0$, each with maximal order of vanishing at most $N$:
\begin{equation}
    \forall j \, , \exists k \leq N \text{ s.t. } b^{(k+1)}(\gamma_j) \neq 0 \, .
\end{equation}
Moreover, the distance between the critical points is bounded below as 
\begin{align}
    |\gamma_j - \gamma_l| \geq 4\sigma_\sharp,  \quad j \neq l.  
\end{align}
\end{assumption}

Under these assumptions, we have uniform-in-diffusivity mixing estimates:

\begin{theorem}
\label{thm:1}
Let $b : \T \to \R$ satisfy Assumption~\ref{assump:b} and $f^{\rm in} \in H^1(\T)$ with zero mean on streamlines:
\begin{equation}
     \langle f^{\rm in} \rangle(y) := \int_\T f^{\rm in}(x,y) \, dx = 0 \, , \quad \forall y \in \T \, .
\end{equation}
There exist $c, \kappa_0 \in (0,1]$, depending only on $b$, such that for $\kappa \in (0,\kappa_0]$, the solution $f : [0,+\infty) \times \T^2 \to \R$ to~\eqref{adv:diff} with initial datum $f^{\rm in}$ satisfies the uniform-in-diffusivity mixing and enhanced dissipation estimate
\begin{equation}
\label{main:bdd}
    \| f(t) \|_{L^2_x H^{-1}_y(\T^2)} \lesssim \langle t \rangle^{-\frac{1}{N+1}} e^{-c \kappa^{\frac{N+1}{N+3}} t} \| f^{\rm in} \|_{L^2_x H^1_y(\T^2)} \, , \quad \forall t \geq 0 \, .
\end{equation}
\end{theorem}



\begin{remark}
The exponential factor in~\eqref{main:bdd} quantifies the enhanced dissipation, which falls under the ``shear-diffuse mechanism" in the applied literature, see~\cite{RhinesYoung1983, HKMoffatt_1983, BAJER_BASSOM_GILBERT_2001}. In the language of~\cite{CZ19}, Theorem \ref{thm:1}  is a ``stable mixing estimate." 
\end{remark}




Observe that for $t \lesssim T_{\rm e} := \kappa^{-\frac{N+1}{N+3}}$, the ``enhanced dissipation time scale," the exponential factor in~\eqref{main:bdd} is negligible: The solution to~\eqref{adv:diff}, which we refer to as the ``viscous" solution, is guaranteed to decay with (at least) the optimal inviscid rate $\langle t \rangle^{-\frac{1}{N+1}}$. However, for $t \gtrsim T_{\rm e}$, the viscous solution is qualitatively different from the inviscid solution,\footnote{In fact, the viscous solution will be qualitatively different from the inviscid solution already for $t \gtrsim \kappa^{-1/3}$, due to better enhanced dissipation in the monotone regions which we do not quantify in~\eqref{main:bdd}.} as is visible by the transition from algebraic to exponential decay. To describe what is going on at this transition, we crucially extract the changing length scale of the solution near the critical points. 



It is useful to recall how the estimates~\eqref{eq:optimalinvisid} are proved in the inviscid case.\footnote{This is essentially the proof of the Van der Corput lemma.} We decompose the solution (equivalently, the initial datum) at time $t$ into the parts near the critical points $\gamma_j$ of $b(y)$ and away from those points:
\begin{equation}
    \int f_k(t,y) \phi \, dy = \sum_j \int_{|y-\gamma_j| \leq \ell_j(t)} + \int_{\rm complement} e^{-i b(y)kt} f^{\rm in}_k  \phi \, dy \, .
\end{equation}
In the ``inner" region near the critical point, we take the support of the initial data small in a $t$-dependent way and use that the initial data is in $L^\infty$ to gain the size of the support, $\ell_j(t)$. 
In the ``outer" region away from the critical points, we integrate by parts off of $e^{-ib(y)kt}$ in~\eqref{int4} and utilize the smoothness of the initial data to gain, in the worst case, $(\ell_j(t)^{N} kt)^{-1}$. Finally, we optimize the size of the cutoff region, e.g., $\ell_j(t) \sim t^{-1/2}$ when $N=1$, to obtain the desired rates~\eqref{eq:optimalinvisid}.

Now consider the viscous problem. We attempt to follow the same strategy as before and decompose the solution into an inner and outer part. However, the relevant length scale may be determined by both shearing and diffusion. To understand where these effects balance, we seek a scaling in which all terms become $O(1)$. For the sake of discussion, assume that we are considering a critical point with $N = 2$ and replace~\eqref{adv:diff} with the model equation
\begin{align}
    \p_t f + i y^2 f = \kappa \p_y^2 f \, .
\end{align}
The relevant scaling is $f(t,y ) = F(T,Y)$ with $T = \kappa^{1/2}t$ and $Y = \frac{y}{\kappa^{1/4}}$:
\begin{equation}
    \label{eq:zoominontimedep}
    \p_T F + i Y^2 F = \p_Y^2 F \, .
\end{equation}
This suggests that, at length scales $\kappa^{1/4}$ around a critical point of order $N=1$ and timescales $\kappa^{-1/2}$, diffusion and advection should balance. Therefore, the inviscid length scale $t^{-1/2}$ should be relevant until $t^{-1/2} \approx \kappa^{1/4}$, at the enhancement time $t \approx \kappa^{-1/2}$.  
Therefore, in order to prove the optimal bounds, we should \emph{saturate} the cutoff scale $\ell(t)$ used to divide between the inner and outer initial data at the length scale $\kappa^{1/4}$. The global-in-time bounds on the inner region will then follow from a combination of enhanced dissipation and the smallness of the support of the initial data. On the outer region, we will exploit the shearing where it is still sufficiently strong. However, the problem cannot be so easily localized, as the diffusion causes ``leakage" between the two regions which cannot be ignored.\footnote{Indeed, the diffusion length scale is $(\kappa t)^{1/2}$, which at the enhancement time $\kappa^{-1/2}$ is comparable to the inner length scale $\kappa^{1/4}$. One interpretation is that, by the enhancement time, particles initially in the outer region will have already sampled the inner region, and vice versa.} The approach we take to circumvent this difficulty and gain access to all of these effects is to analyze the \emph{pointwise} behavior of the resolvent of the operator $\kappa \p_y^2 - ib(y)$ -- see Section~\ref{sec:resolve}. 

\begin{remark}
\label{rem:k=1}
The viscous analogue to the mode-by-mode problem~\eqref{int2} is
\begin{equation}
    \p_t f_k + ikb(y) f_k = \kappa (-k^2 + \p_y^2 f_k) \, .
\end{equation}
It suffices to consider the behavior when $k > 0$, since solutions with $k < 0$ may be obtained by complex conjugation. The PDE for the unknown $g_k(s,y) = e^{\kappa k^2 t} f_k(y,t)$, $s = kt$, is
\begin{equation}
    \p_s g_k + ib(y) g_k = \frac{\kappa}{k} \p_y^2 g_k \, .
\end{equation}
Upon relabeling $s$ and $g_k$, we obtain the problem
\begin{equation}
    \label{eq:equationafterreduction}
    \p_t f + ib(y) f = \varepsilon \p_y^2 f
\end{equation}
for a function $f : [0,+\infty) \times \T \to \C$, keeping in mind the time-rescaling and $\varepsilon := \kappa/k$. Below, $\kappa$ appears when we analyze the full problem~\eqref{adv:diff}, whereas $\varepsilon$ appears when we analyze the $k=1$ mode.
\end{remark}


\subsection{Asymptotic description and limiting length scale}

We now discuss the long-time asymptotics of~\eqref{adv:diff}. Given Remark~\ref{rem:k=1}, our main focus will be on the equation~\eqref{eq:equationafterreduction}. Generically, the asymptotic shape of $f$ should be represented by the ``slowest eigenfunctions", i.e., those corresponding to the eigenvalues $\lambda$ of
\begin{equation}
    \label{eq:Lvarepsilondef}
    L_\varepsilon := \varepsilon \p_y^2 - i b(y) : H^2(\T) \subset L^2(\T) \to L^2(\T)
\end{equation}
with maximal real part. We will characterize such eigenfunctions $\phi$ when $b$ has finitely many critical points $\gamma_j$, each of which is non-degenerate: $N=1$.

We expect $\phi$ to be concentrated around a critical point $\gamma_j$, where the shear is the weakest. We stretch variables in the eigenvalue equation
\begin{equation}
    [\lambda + ib(\gamma_j)] \phi + i [b(y) - b(\gamma_j)] \phi = \varepsilon \p_y^2 \phi
\end{equation}
via the transformation
\begin{equation}
    \label{eq:innervariables1}
    Y = \frac{y-\gamma_j}{\tilde{\ell_j} \varepsilon^{1/4}} \, , \quad \tilde{\ell_j} := |b''(\gamma_j)/2|^{1/4}
\end{equation}
\begin{equation}
    \label{eq:lambdascaling}
    \Lambda = \tilde{\tau_j} \varepsilon^{-1/2} \lambda + ib(\gamma_j) \, , \quad \tilde{\tau_j} := |b''(\gamma_j)/2|^{-1/2}
\end{equation}
\begin{equation}
    \label{eq:innervariables3}
     \Phi(Y) = \phi(y) \, , \quad B(Y) = \tilde{\ell}_j^{-2} \varepsilon^{-1/2} [b(y) - b(\gamma_j)] \, ,
\end{equation}
the spectral equivalent of the transformation in~\eqref{eq:zoominontimedep}, to obtain
\begin{equation}
    \label{eq:innerequationfromthebeginning}
    \Lambda \Phi + i \sgn b''(\gamma_j) \; Y^2 \Phi = \p_Y^2 \Phi + O(\varepsilon^{1/4} Y^3) \, .
\end{equation}
The equation~\eqref{eq:innerequationfromthebeginning}, without the big Oh term, is the eigenvalue problem for a \emph{complex harmonic oscillator} (see~\cite{DaviesComplexHarmonic}),  whose spectrum may be obtained from that of the quantum harmonic oscillator $\p_Y^2 - Y^2$ by a rotation in the complex plane. Specifically, the eigenvalues and eigenfunctions of $\p_Y^2 \mp iY^2$ are
\begin{equation}
    \label{eq:lambdascaling2}
\Lambda_\alpha^{\pm} = - (2\alpha + 1) e^{\pm i\pi/4} \, , \quad \Phi_\alpha^{\pm}(Y) := e^{\pm i\pi/16} G_\alpha(e^{\pm i\pi/8}  Y) \, , \quad \alpha \in \N_0 \, ,
\end{equation}
and $G_\alpha$ is the $\alpha$th Hermite function (see Section~\ref{sec:efngluing}). Consequently, $\lambda \approx \varepsilon^{1/2}$, which is consistent with the enhanced dissipation timescale. The concentrated eigenfunctions correspond to the shear layers observed numerically.



The above asymptotics motivate us to provide the following \emph{full asymptotic characterization} of the spectrum of $L_\varepsilon$ in the spectral windows $\{ \Re \lambda \geq - q \varepsilon^{1/2} \}$ for arbitrary $q \geq 1$. 

\begin{theorem}
    \label{thm:eigenfunctions}
Let $b(y) : \T \to \R$ be a smooth function with finitely many critical points $\gamma_j$, $j = 1,\hdots, J_0$, each of which is non-degenerate: $b''(\gamma_j) \neq 0$. Suppose that the wave speeds $b(\gamma_j)$, $j = 1,\hdots,J_0$, are distinct. Let $L_\varepsilon$ be the operator in~\eqref{eq:Lvarepsilondef}. 

Let $q \geq 1$. There exist $\varepsilon_0 > 0$ (depending only on $b$ and $q$) and $A_j \in \N_0$, $j = 1, \hdots, J_0$, such that, for all $\varepsilon \in (0,\varepsilon_0]$, we have
\begin{equation}
\sigma(L_\varepsilon) \cap \{ \Re \lambda \geq - q \varepsilon^{1/2} \} = \sqcup_{j=1}^{J_0} \sqcup_{\alpha=0}^{A_j} \{ \lambda^{(j)}_\alpha(\varepsilon) \} \, ,
\end{equation}
where each spectral value $\lambda^{(j)}_\alpha(\varepsilon)$ is an algebraically simple eigenvalue satisfying the asymptotics
\begin{equation}
    \label{eq:evalremainderass}
\lambda^{(j)}_\alpha(\varepsilon) + ib(\gamma_j) = \tilde{\tau_j} \varepsilon^{1/2} \Lambda^{\pm}_\alpha + O_{\varepsilon \to 0^+}(\varepsilon^{3/4}) \, ,
\end{equation}
$\pm$ chosen to match $\sgn b''(\gamma_j)$,
with corresponding eigenfunction $\phi^{(j)}_\alpha(y;\varepsilon)$, which may be normalized to satisfy the asymptotics
\begin{equation}
\label{evalexp}
\phi^{(j)}_\alpha(y;\varepsilon) = \Phi_\alpha^{\pm} \left( \frac{y-\gamma_j}{\tilde{\ell}_j \varepsilon^{1/4}} \right) + O_{\varepsilon \to 0^+}(\varepsilon^{3/8}) \text{ in } L^2(\T) \, .
\end{equation}
\end{theorem}

Theorem~\ref{thm:eigenfunctions} is illustrated in Figures~\ref{fig:evals} and~\ref{fig:efns}. The parameter $\alpha \in \N_0$ indexes a family of eigenfunctions \emph{for each critical point}~$\gamma_j$. The asymptotics of $\lambda^{(j)}_\alpha$ and $\phi^{(j)}_\alpha$ are also available to arbitrary order, as we describe in Section~\ref{sec:higherorderasymptotics}.

\begin{remark}
    The condition that the wave speeds are distinct is merely technical and can probably be removed. Versions of the above theorem should hold with Dirichlet and Neumann boundary conditions in $y$ provided that the critical points $\gamma_j$ are away from the boundary.
\end{remark}


\begin{figure}
\centering
\begin{minipage}{.5\textwidth}
  \centering
  \includegraphics[width=\linewidth]{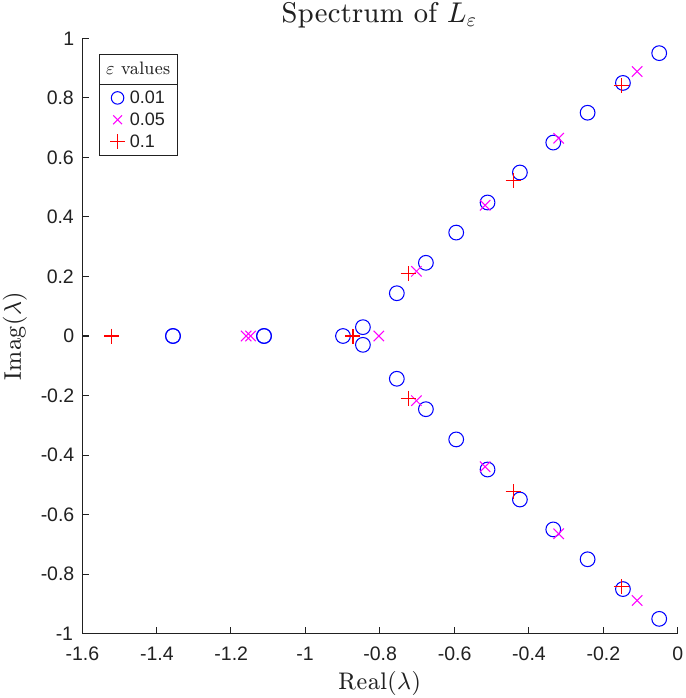}
\end{minipage}%
\begin{minipage}{.5\textwidth}
  \centering
  \includegraphics[width=\linewidth]{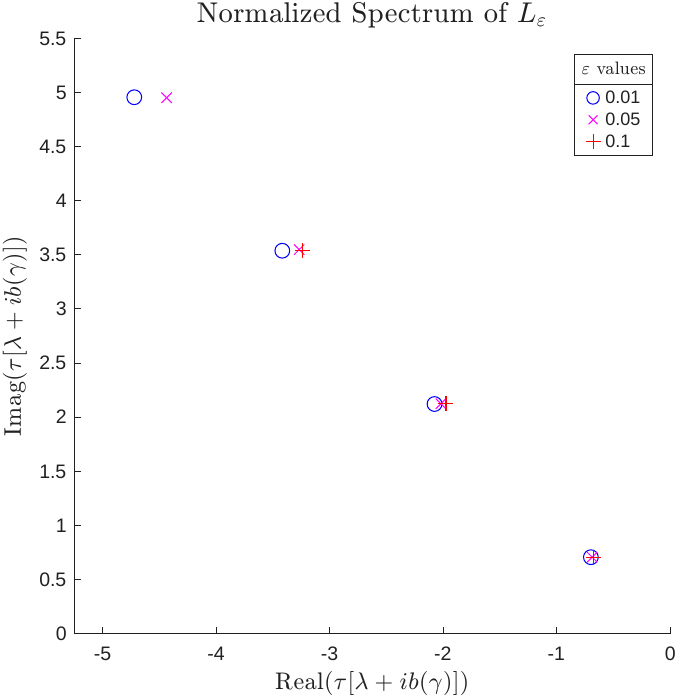}
\end{minipage}
\caption{Left: First 25, 12, and 8 eigenvalues of the operator $L_\varepsilon = \varepsilon \p_y^2 - i \sin y$ with sinusoidal shear flow and $\varepsilon = 0.01, 0.05, 0.1$, respectively.  Compare~\cite[Figure 22.7, p. 223]{TrefethenEmbree} for the complex Airy operator $\varepsilon \p_y^2 + iy$ with Dirichlet boundary conditions. Right: First 4 normalized eigenvalues $\Lambda_\alpha$, $\alpha = 0, 1, 2, 3$, in the notation of~\eqref{eq:lambdascaling} and~\eqref{eq:lambdascaling2}, from the ``lower branch": $\gamma = \pi/2$ and $b(\gamma) = 1$. Here, $\tau = (2/\varepsilon)^{1/2}$ is the scaling factor in~\eqref{eq:lambdascaling}.}
\label{fig:evals}
\end{figure}


\begin{figure}
\centering
\begin{minipage}{.5\textwidth}
  \centering
  \includegraphics[width=\linewidth]{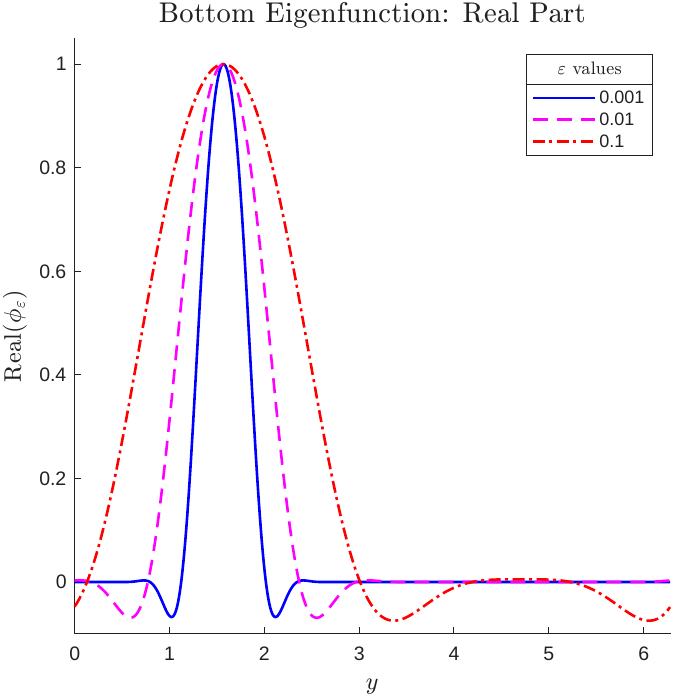}
\end{minipage}%
\begin{minipage}{.5\textwidth}
  \centering
  \includegraphics[width=\linewidth]{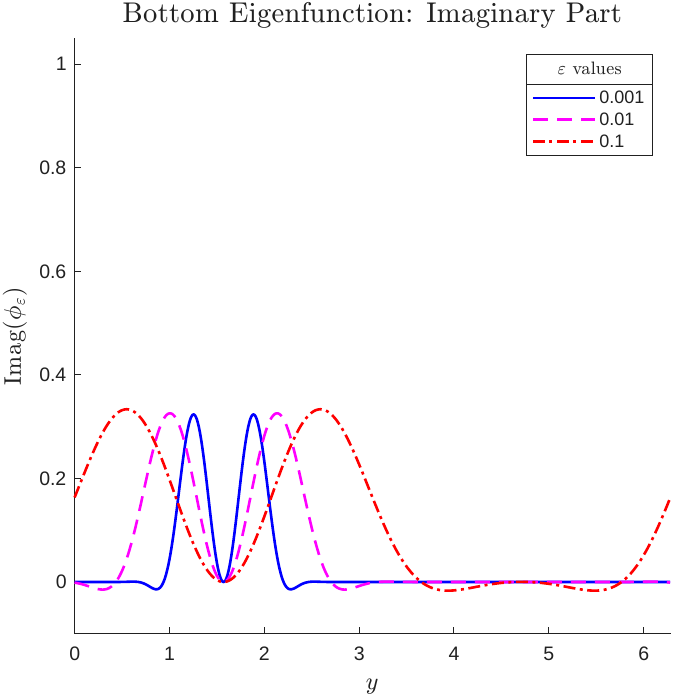}
\end{minipage}
\caption{Bottom eigenfunction, concentrated around $y = \pi/2$, of the operator $L_\varepsilon = \varepsilon \p_y^2 - i \sin y$ with sinusoidal shear flow and $\varepsilon = 10^{-3}, 10^{-2}, 10^{-1}$.}
\label{fig:efns}
\end{figure}

As a corollary, we prove that, generically (outside of a finite-dimensional subspace $E_\kappa$), the full solution to~\eqref{adv:diff} asymptotically decomposes as a superposition of decaying traveling waves localized to ``shear layers" around the critical points. 
Such traveling waveforms are illustrated in Figure~\ref{fig:traveling}. This answers Open Problem \#3 in Section~4.4 of~\cite{Elgindi2022Lectures} and rigorously verifies the numerical simulations in~\cite{Camassa2010}. 

\begin{corollary}
    \label{cor:spectralcorollary}
    Let $b$ as in Theorem~\ref{thm:eigenfunctions}. There exist $K_0 \leq J_0$ and $\kappa_0 > 0$ such that for all $\kappa \in (0,\kappa_0]$, there exists a subspace $E_\kappa \subset L^2$ with $\dim E_\kappa \leq 2 K_0$ such that each solution $f$ to~\eqref{adv:diff} with $L^2$ initial datum $f^{\rm in} \not\in E_\kappa$ and $\langle f^{\rm in} \rangle = 0$ asymptotically decomposes into a superposition of at most $K_0$ decaying, traveling waves which decay with rate $e^{t [-\kappa+\max \sigma(L_\kappa)]}$. For these solutions, there exist $0 < \underbar{c} < \bar{c}$, depending only on $b$, such that the quantity
    \begin{equation}
        \label{eq:elltdef}
\ell(t) := \frac{\| f(t) \|_{L^2}}{\| f(t) \|_{\dot H^1}}
\end{equation}
satisfies the asymptotics
    \begin{equation}
        \label{eq:conclusionofelltdef}
\limsup_{t \to +\infty} \ell(t) \leq \bar{c} \kappa^{1/4} \, , \quad
    \liminf_{t \to +\infty} \ell(t) \geq \underbar{c} \kappa^{1/4} \, .
\end{equation}
\end{corollary}

Specifically, $E_\kappa$ is the subspace of functions whose projection onto the slowest eigenmodes is zero; its dimension is bounded above by $K_0$, the number of critical points $\gamma_j$ with minimal $|b''(\gamma_j)|$. The wave speeds are $- \Im \lambda_0^{(j)}(\varepsilon)$ for certain $j$, and the waveforms are $\Re e^{ix} \phi_1^{(j)}$. We can further describe the spectral projections onto these waveforms. All of this is proven in Section~\ref{sec:asymptoticdescription}.

\begin{figure}
\centering
\begin{minipage}{.5\textwidth}
  \centering
  \includegraphics[width=\linewidth]{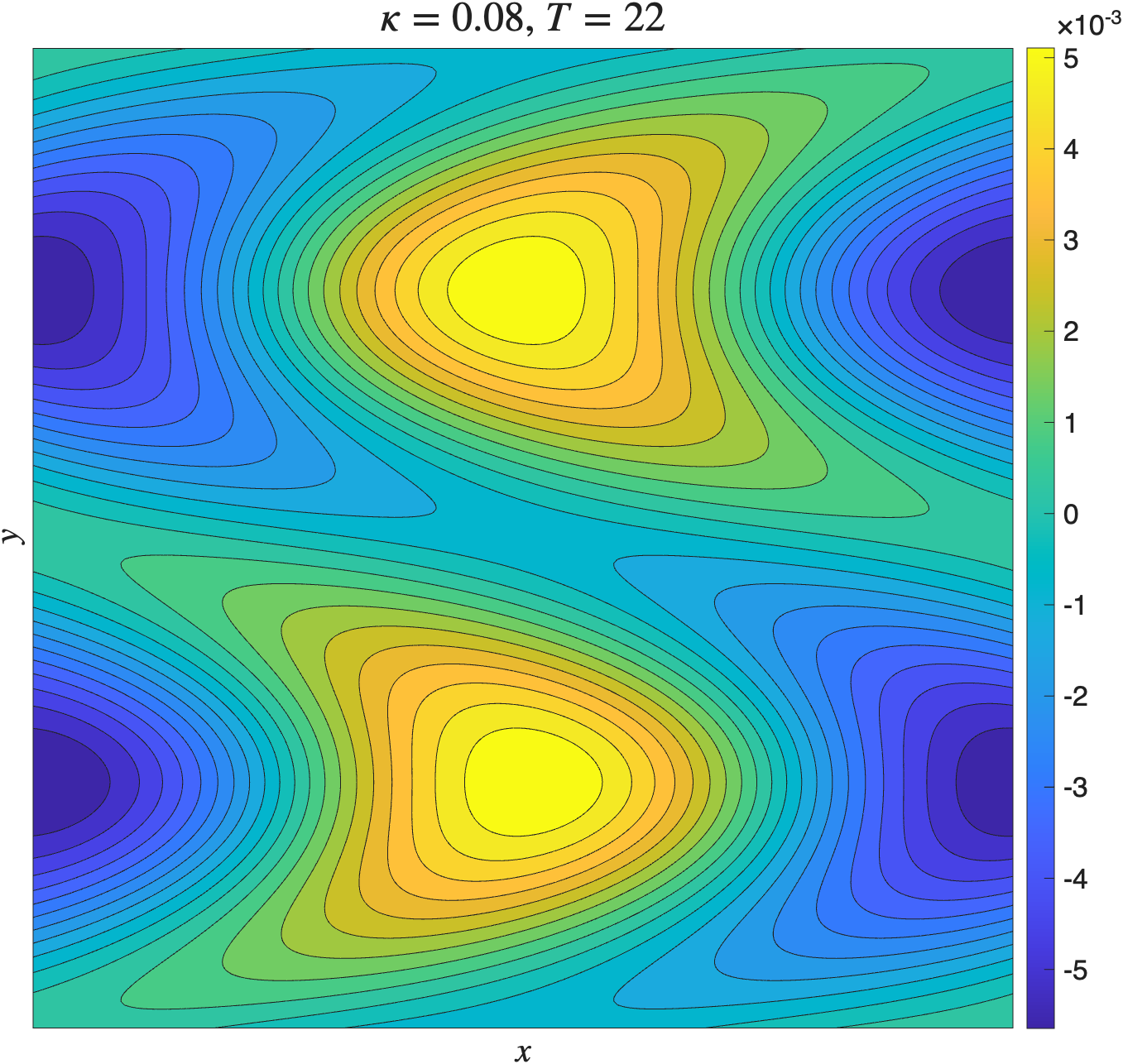}
\end{minipage}%
\begin{minipage}{.5\textwidth}
  \centering
  \includegraphics[width=\linewidth]{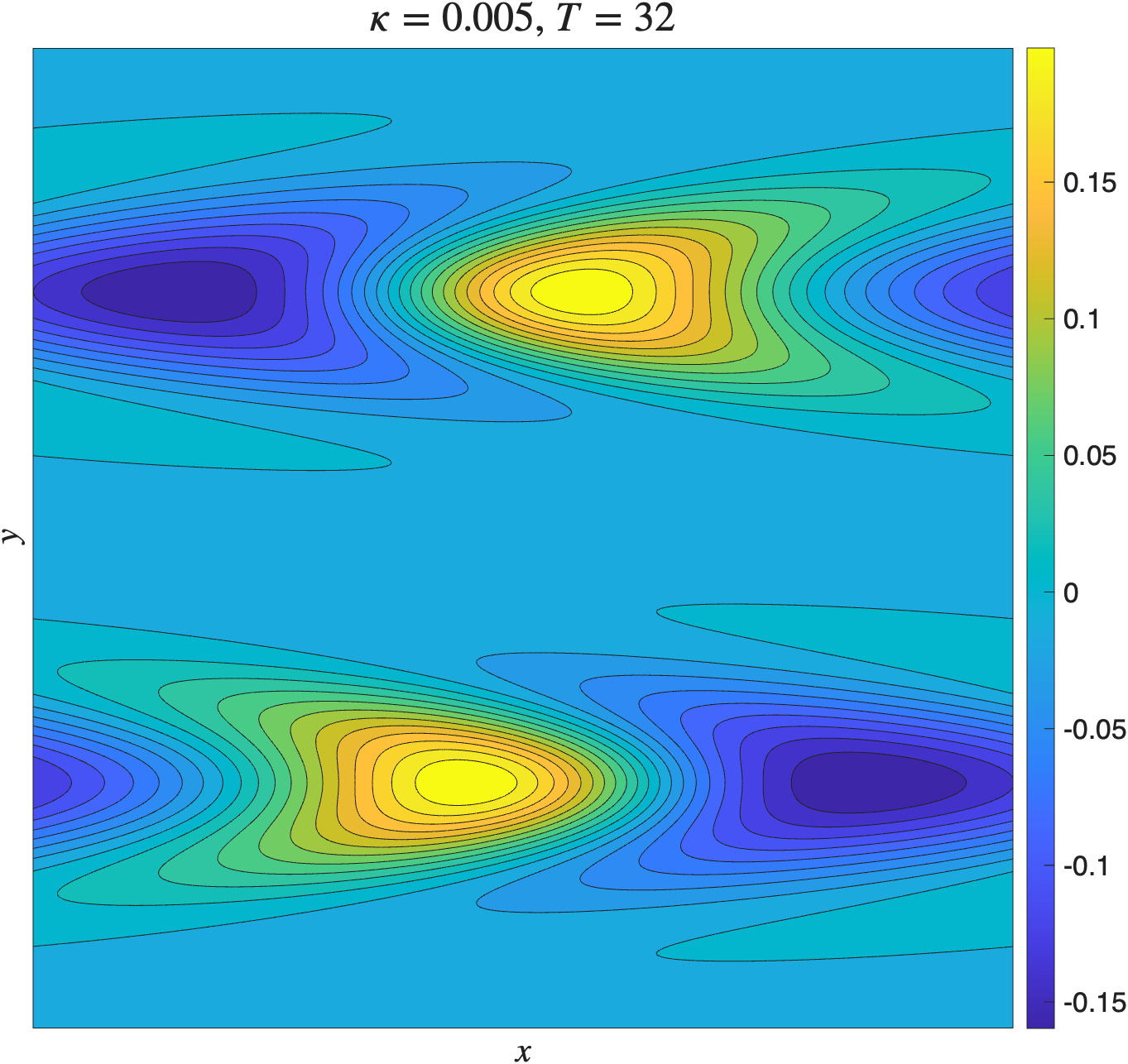}
\end{minipage}
\caption{Heat plots at time $T$ of a passive scalar with initial datum $f^{\rm in} = \exp(-|x-\pi|^2) - \langle \exp(-|x-\pi|^2) \rangle$ evolving in a sinusoidal shear flow $b(y) = \sin y$. Left: Diffusivity $\kappa=0.08$, $T=22$. 
Right: Diffusivity $\kappa=0.005$, $T=32$.
Compare~\cite{Camassa2007,Camassa2010}.}
\label{fig:traveling}
\end{figure}

\begin{remark}
The quantity $\ell(t)$ in~\eqref{eq:elltdef} is a length scale associated with the solution, and~\eqref{eq:conclusionofelltdef} says that this length scale is asymptotically $\approx \kappa^{1/4}$. 
It is also common~\cite{Miles2018} to measure the quantity
\begin{equation}
    \bar{\ell}(t) := \frac{\| f(t) \|_{\dot H^{-1}}}{\| f(t) \|_{L^2}} \, .
\end{equation}
We demonstrate that $\limsup_{t \to +\infty} \bar{\ell}(t) \approx \kappa^{1/8}$, in the sense of upper and lower bounds. However,~\eqref{eq:conclusionofelltdef} remains true for quantities of the type
\begin{equation}
    \bar{\ell}_s(t) := (\| f(t) \|_{\dot H^{-s}}/\| f(t) \|_{L^2})^{1/s} \, , \quad s \in (0,1/2) \, .
\end{equation}
 It also remains true when the full $H^1$ norm appears in the denominator of~\eqref{eq:elltdef}. We highlight the distinction between these quantities more carefully in Remark~\ref{rmk:projectivized} and in the proof. Notably, the quantity $\ell(t)$ in~\eqref{eq:elltdef} also appears in the recent work~\cite{hairer2024lower} on the Batchelor scale conjecture. Finally, in the shear setting, the true ``length scale" of the solution is the box size if you do not remove the means on streamlines $\langle f \rangle(y)$, which satisfy the heat equation $\p_t \langle f \rangle = \kappa \p_y^2 \langle f \rangle$.
\end{remark}

We discuss future questions relating to Theorem~\ref{thm:eigenfunctions} in Section~\ref{sec:futurequestions}.

\subsection{Comparison with existing literature}
In the monotone case (formally $N = 0$), \cite{CZ19} proved uniform-in-$\kappa$ mixing analogous to Theorem~\ref{thm:1} via a combination of the vector field method and hypocoercivity. We originally learned of question \textbf{(Q)} from~\cite{CZ19}, which partly served as the inspiration for the present work.

It is worthwhile to review the vector field method, as it gives the previous best results. The vector field method involves commuting $X := t b'(y)\p_x + \p_y $ through the equation\footnote{The ``good derivative" $D_\lambda$ in~\eqref{thed} may be regarded as a (truncated) resolvent analogue of this vector field.} and estimating $\| f \|_{H^{-1}}$ by $f$, $Xf$, and $\p_x f$. For the Couette flow, $X$ further commutes with the diffusion $\p_y^2$, and in the monotone case, the commutator terms can still be estimated. In~\cite{CotiZelati2023, DCZGV2024}, the authors designed a vector field capable of handling non-degenerate critical points\footnote{The authors' specific application was to a kinetic model of microswimmers in which the ``free streaming" term is $\p_t + p \cdot \nabla_x$, $p \in S^2$ -- see also~\cite{AlbrittonOhmStabilizing}. The phase function $p_3$ has critical points at the north and south poles.} and which has good commutativity with respect to the Laplacian. With this method, they proved mixing rates up to the enhanced dissipation time, but not globally-in-time. Therein, a sticking point in how the vector field is used to control the $H^{-1}$ norm led to a logarithmic loss, see~\cite[Section~4.3]{CotiZelati2023}. It would be interesting to further refine this method.



Closely related to the evolution of a passive scalar~\eqref{adv:diff} is that of the two dimensional Navier-Stokes equation written in terms of the vorticity linearized around a shear flow. In the case of Couette flow, the two equations are, in fact, identical. For a general shear flow $(b(y),0)$, the linearized equations for the vorticity $\omega(t, x ,y )$ take the form\footnote{$(b(y),0)$ is a solution of the Navier-Stokes equations with force $(-\kappa b''(y), 0)$.} 
\begin{subequations}
\label{lin:nse}
    \begin{align}
   \p_t \omega + b(y)\p_x \omega - b''(y)\p_x \psi &= \kappa \Delta \omega, \\
    \Delta \psi &= \omega.
\end{align}
\end{subequations}
The nonlocal term $b''(y)\p_x \psi $ is a significant complicating factor for the analysis of \eqref{lin:nse}. For the domain $\T \times [0,1 ]$ with $\kappa = 0$ and spectrally stable and monotone $b$, optimal mixing rates were first proved in \cite{WZZ18}.\footnote{In the context of the linearized equations, mixing is usually referred to as \emph{inviscid damping}. } See also \cite{J20, J20A} for alternative proofs. For the domains $\T \times [0,1]$ and $\T \times \R$, uniform-in-diffusivity mixing for spectrally stable monotone shear flows  was proved  in \cite{CWZ23} and \cite{J23} respectively. See also \cite{GNRS20} where uniform estimates are proved for monotone spectrally stable flows up to time $\sim \kappa^{-1}$.

For $y \in \T$, smooth shears necessarily contain critical points. While the proofs are very different, the behavior one can prove in the passive monotone case and the linearized spectrally stable monotone case are similar. However, for critical points, there is a marked difference between the behavior of the passive evolution and the linearized evolution. The key new effect in the linearized equations is known as \emph{vorticity depletion} and leads to vanishing of the solution at the critical points which in turn leads to improved mixing rates for the linearized equation relative to the passive equation. In the inviscid case, this was first observed numerically in \cite{BM10} and later proved in \cite{WZZ19, WZZ20}. A more refined pointwise description of the behavior near the critical points was later given in \cite{IIJ24} and \cite{BG24}. In the case with viscosity, vorticity depletion was proved in \cite{BCJ24}. Many aspects of the proof of Theorem \ref{thm:1} are inspired by the techniques developed in \cite{J23, BCJ24}. We highlight some key technical distinctions in Section \ref{sec:compare}.

Besides mixing, enhanced dissipation also plays a key role in the proof of Theorem \ref{thm:1}. The relationship between enhanced dissipation and mixing was explored in the pioneering work \cite{CKRZ08}. A more quantitative version of the argument from \cite{CKRZ08} was later developed in \cite{CZDE20}. See also \cite{Wei21} which provides a simpler proof of \cite{CKRZ08} via a Gearhardt-Pr\"uss type theorem. For the case of general shears, the number of enhanced dissipation results is now vast. Making no attempt to be exhaustive, see \cite{BCZ16, ALBRITTON2022, CZG23, GLM24, V24} and references therein. 

Finally, we mention a selection of the many important contributions from the applied literature. Shear-diffusion in the context of the Couette flow goes back to Kelvin~\cite{kelvin1887stability}. In the context of passive scalars, much attention has been devoted to Taylor dispersion~\cite{taylor1954dispersion}. Notably, Taylor dispersion is a \emph{low frequency effect}, valid on $O(\kappa^{-1})$ timescales, which should not be conflated with the intermediate frequency behavior associated to shear-diffusion which we investigate here. The papers~\cite{Camassa2007,Camassa2010} investigated behavior on intermediate timescales, akin to what is done in this paper. Notably, our Theorem~\ref{thm:eigenfunctions} was predicted by formal WKB analysis therein. A theory for passive scalars advected by vortices is contained in~\cite{BAJER_BASSOM_GILBERT_2001}. Other relevant papers include~\cite{GIONA2004,VANNESTE2007}.




\subsection{Dimensional analysis}
As a guiding principle, we describe certain important length and time scales which, in particular, inform the change of variables~\eqref{eq:innervariables1}-\eqref{eq:innervariables3} and the analysis in Section~\ref{sec:resolve}. Consider
\begin{equation}
\p_t f + ik b(y) f = \kappa \p_y^2 f \, .
\end{equation}
It will be useful to distinguish between $x$- and $y$-lengths.
Then $[t] = T$, $[k] = 1/L_X$, $[y] = L_Y$, $[b] = L_X/T$, and $[\kappa] = L^2_Y/T$. The derivative $b'$, with dimensions $[b'] = L_X/L_Y T$, will play a more important role than $b$ itself. Where $b'(y) \neq 0$, the \emph{local enhanced dissipation time} is
\begin{equation}
    T_{\rm loc}(y) = (k |b'(y)| \kappa^{1/2})^{-2/3} = (k|b'(y)|)^{-2/3} \kappa^{-1/3} \, ,
\end{equation}
and the associated length scale is
\begin{equation}
    L_{\rm loc}(y) = ( \kappa T_{\rm loc}(y) )^{1/2} = (k|b'(y)|)^{-1/3} \kappa^{1/3}, \, 
\end{equation}
see Remark \ref{phylength}. At a critical point $\gamma$ of order $N$, nearby which the local behavior is $b(y) \approx b(\gamma) + \frac{1}{(N+1)!} b^{(N+1)}(\gamma) (y-\gamma)^{N+1}$, the local enhanced dissipation time is instead
\begin{equation}
    T_{\rm loc}(\gamma) = ( k b^{(N+1)}(\gamma) \kappa^{\frac{N+1}{2}} )^{-\frac{2}{N+3}} =(k b^{(N+1)}(\gamma) )^{-\frac{2}{N+3}} \kappa^{-\frac{N+1}{N+3}} \, .
\end{equation}
The associated length scale is
\begin{equation}
    L_{\rm loc}(\gamma) = (k b^{(N+1)}(\gamma) )^{-\frac{1}{N+3}} \kappa^{\frac{1}{N+3}} \, .
\end{equation}
Notably, when $|y - \gamma| \lesssim L_{\rm loc}(\gamma)$, we have $T_{\rm loc}(y) \gtrsim T_{\rm loc}(\gamma)$ and $L_{\rm loc}(y) \gtrsim L_{\rm loc}(\gamma)$; in that whole region, the length and time scales associated with the critical point $\gamma$ are the most relevant. This is visible in the definition of $|B'(y)|$ in~\eqref{B} below.

%% file: mixingviaresolvent.tex
\section{Mixing via resolvents}
\label{sec:resolve}

We consider $f^{\rm in} \in H^2(\T)$ to justify the representation formula~\eqref{rep}. This will only be used as a qualitative assumption, and none of the bounds we prove will depend on the $H^2(\T)$ norm. For general $H^1(\T)$ initial data, Theorem \ref{thm:1} will follow by a standard approximation argument.

Taking the Fourier transform in $x$ of \eqref{adv:diff} and recalling Remark~\ref{rem:k=1}  leads us to consider
\begin{equation}
\label{eq:passive:k}
    \p_t f_1 + ib(y)f_1 - \varepsilon (\partial_y^2 - 1) f_1 = 0, \quad 
       f_1(0,y ) =  f_1^{\rm in}(y),
\end{equation}
We can rewrite \eqref{eq:passive:k} as 
\begin{align}
\label{reduced}
    \p_t f_* =  L_{\varepsilon} f_*,
\end{align}
where  $L_\varepsilon$ is defined as in \eqref{eq:Lvarepsilondef},
and we have set
\begin{align}
    f_* (t,y) := f_1(t,y)e^{\varepsilon t}, \quad f_{*}^{\rm in}(y) := f_1^{\rm in}(y).
\end{align}
The following representation formula holds for $f_*$:\footnote{This is guaranteed to hold in the principal value sense by the Laplace inversion formula in, e.g.,~\cite[Chapter III, Corollary 5.15]{engel2000one} whenever $f^{\rm in}_*$ belongs to the domain of $L_\varepsilon$, namely, $H^2(\T)$, and $s(L_\varepsilon) := \sup \Re {\rm spec}(L_\varepsilon) < - \sigma_0 \varepsilon^{\frac{N+1}{N+3}}$, which holds for sufficiently small $\sigma_0$ and which we also prove. Recall that the growth bound of the semigroup $e^{tL_\varepsilon}$ and the spectral bound $s(L_\varepsilon)$ are equal in this setting, see~\cite[Chapter IV, Corollary 3.11]{engel2000one}.} 
\begin{align}
\label{rep}
    f_*(t,y) = \frac{e^{ -\sigma_0 \varepsilon^{\frac{N+1}{N+3}}  t}}{2 \pi}\int_{\R} e^{-i \lambda t  } \left[(-i \lambda  -\sigma_0 \varepsilon^{\frac{N+1}{N+3}}   - L_{\varepsilon})^{-1} f_*^{\rm in} \right](y) d \lambda,
 \end{align}
for $\sigma_0 \in (0,1)$  which will be fixed as part of Proposition~\ref{k:bdds} (see Appendix~\ref{sec:airy}).
Slightly abusing notation, we define
\begin{align}
\label{sdf}
    f(y, \lambda) := \left[(-i \lambda  -\sigma_0 \varepsilon^{\frac{N+1}{N+3}}- L_{\varepsilon})^{-1} f_*^{\rm in} \right](y).
\end{align}
It follows from \eqref{sdf} that 
\begin{align}
\label{reso}
    -\varepsilon\p_y^2 f(y, \lambda)   -\sigma_0 \varepsilon^{\frac{N+1}{N+3}}f(y, \lambda) + i(b(y) - \lambda)f(y, \lambda) = f_*^{\rm in}(y).
\end{align}
We denote the Airy-type operator on the left-hand side of \eqref{reso} as
\begin{align}
    \label{eq:Airy}
    \A g(y) := -\varepsilon\p_y^2 g(y)  -\sigma_0 \varepsilon^{\frac{N+1}{N+3}}g(y ) + i(b(y) - \lambda)g(y).
\end{align}
In the inviscid case $\varepsilon = 0$, $\A$ is simply multiplication by $i(b(y) - \lambda)$ with inverse $-i(b(y) - \lambda)^{-1}$. For $\varepsilon > 0$,  $\ai$ is no longer explicit and the kernel is highly oscillatory. A key insight of \cite{J23} and \cite{BCJ24} for the case of monotone shears  and shears with non-degenerate critical points  respectively, was to clarify the sense in which 
\begin{align}
    \ai (g)(y) \approx \frac{g(y)}{i(b(y) -\lambda)  - r(\varepsilon)},  
\end{align}
where $r(\varepsilon)$ is a regularization scale determined by the interaction between the diffusion $\varepsilon \p_y^2$ and the potential $i(b(y) -\lambda)$. We will necessarily have to use and extend these ideas to the case of critical points of arbitrary finite order in the current setting.

\subsection{Airy bounds}
We first state the fundamental solution bounds for the Airy kernel which are of fundamental importance to the proof of Theorem \ref{thm:1}.

\begin{proposition}
\label{k:bdds}
   There exist $\varepsilon_0, c_0, \sigma_0 >0 $ such that if  $K_\varepsilon(y,z, \lambda): \T^2 \times \R \ \to \C$ solves 
    \begin{align}
    \label{fs0}
        -\varepsilon \p_y^2 K + \alpha K - \sigma_0 \varepsilon^{\frac{N+1}{N+3}} K + i(b(y) - \lambda) K = \delta(y-z)
    \end{align}
    with $\alpha \geq 0$, then for $\varepsilon \in (0, \varepsilon_0)$, the following bounds hold:
    \begin{align}
    \label{fs1}
    |K(z, z, \lambda)| \lesssim A(z) :=   \frac{\varepsilon^{-1/2}}{( \alpha + \varepsilon^{1/3}|B'(z)|^{2/3} + |b(z) - \lambda|  )^{1/2}},
\end{align}
\begin{align}
    \label{fs2}
     |K(y,z, \lambda)| \lesssim A(z) e^{-c_0 \frac{|y-z|}{L(y,z, \lambda)} }, 
\end{align}
where 
\begin{align}
\label{fs3}
    \frac{1}{L(y,z, \lambda)} := 
    \begin{cases}
        \frac{( \alpha + \varepsilon^{1/3}|B'(y)|^{2/3}  + \varepsilon^{1/3}|B'(z)|^{2/3}  + |b(y) - \lambda| + |b(z) - \lambda| )^{1/2}}{\varepsilon^{1/2}} & \quad |y-z| < \sigma_\sharp\\
        \frac{(\alpha  +\sigma_\sharp)^{1/2}}{\varepsilon^{1/2}} & \text{ otherwise }
    \end{cases}
\end{align}
and 
\begin{align}
\label{B}
    |B'(y)| := 
    \begin{cases}
   2\sigma_0\left(|y-\gamma_j|^m + \varepsilon^{\frac{m}{m + 3}} \right) & \quad |y - \gamma_j | \leq \sigma_\sharp,\gamma_j\text{ is a critical point of order m} \\
     2\sigma_0|b'(y)| & \quad |y - \gamma_j| \geq 2 \sigma_\sharp \text{ for any critical point } \gamma_j \\
    \text{ linear interpolation }  &\quad \text{ otherwise. } 
    \end{cases}
\end{align}
where $\sigma_\sharp$ is as in Assumption \ref{assump:b}, and
\begin{align}
    \label{fs4}
     |\p_yK(y,z, \lambda)| \lesssim \varepsilon^{-1} e^{-c_0 \frac{|y-z|}{L(y,z, \lambda)} } .
\end{align}
\end{proposition}
We also introduce the notation 
\begin{align}
    L(x):=  \frac{\varepsilon^{1/2}}{ (\alpha + \varepsilon^{1/3}|B'(x)|^{2/3} + |b(x) -\lambda| )^{1/2 }   }.
\end{align}
Using this, the fundamental solution bounds in \eqref{fs1} and \eqref{fs2} can be reformulated as 
\begin{align}
    \label{fsalt}
    |K(y,z, \lambda)| \lesssim \varepsilon^{-1} L(z)  e^{-c_0 \frac{|y-z|}{L(y,z, \lambda)} } .
\end{align}
\begin{remark}
    For all estimates in this section, we take $\alpha = 0$.  
\end{remark}
\begin{remark}
    $|B'(y)|$ can be thought of as a regularized version of $|b'(y)|$ that is always positive as opposed to merely nonnegative at the critical points. Away from critical points, $\varepsilon^{1/3} |b'(y)|^{2/3} \approx \varepsilon^{1/3}|B'(y)|^{2/3}$. However, near a critical point $\gamma_j$ of order $m$ (specifically when $|y - \gamma_j| \lesssim \varepsilon^{\frac{1}{m +3}}$), we have 
    $\varepsilon^{1/3}|B'(y)|^{2/3} \approx \varepsilon^{\frac{m+1}{m+3}}$. 
\end{remark}
\begin{remark}
\label{phylength}
    Points where $b(y) = \lambda$ are sometimes referred to as ``critical layers" in the literature (see, for example, \cite{DR04}), a convention we adopt in this paper. For $\alpha = 0$, away from the critical layers, $K(y,z, \lambda)$ varies on the short length scale $\varepsilon^{1/2}$. When $y \approx z$ and $b(y) = \lambda$, $K(y,z, \lambda)$ varies on the longer length scale $\frac{\varepsilon^{1/3}}{|B'(y)|^{1/3}}$. Note that at a critical point of order $m$, this length scale is of the order $\varepsilon^{\frac{1}{m + 3}}$.  It is sometimes helpful to keep in mind that $\frac{\varepsilon^{1/3}}{|B'(y)|^{1/3}}$ is the ``physical" length scale that is actually observed in the solution near the point $y$. 
\end{remark}
\begin{remark}
    Regarding the piecewise definition of $L(y,z, \lambda)$, this accounts for the case where there are distinct critical layers $z_*$ and $y_*$ satisfying $b(z_*) = b(y_*) = \lambda$ which are \emph{not} close to the same critical point.  The length scale at $y_*$ is not given by $\frac{\varepsilon^{1/3}}{|B'(y)|^{1/3}}$ but rather a shorter length scale comparable to $\varepsilon^{1/2}$.  
    See Lemma \ref{lemlow} for further details on this point.
\end{remark}

The proof of Proposition~\ref{k:bdds} is given in Appendix~\ref{sec:airy}. Here we give a heuristic justification of why the bounds take their particular form. As a first observation, we seek to carefully quantify the $L^2$ ``spectral gap" coming from the interaction between the potential $i(b(y) - \lambda)$ and the dissipation $\varepsilon \p_y^2$. Away from critical layers the potential is coercive on $L^2$. However, near critical layers, one needs to use the dissipation to ensure $L^2$ coercivity, see Lemma \ref{spec:gap}. The upshot of this is that there is an ``effective potential" term $\varepsilon^{1/3}|B'(y)|^{2/3}$ which can be included in \eqref{fs0} to get 
\begin{align}
\label{effective}
  -\varepsilon \p_y^2 K +  \varepsilon^{1/3}|B'(y)|^{2/3}K + i(b(y) - \lambda) K = \delta(y-z).
\end{align}
Note that, for all $y$ and $\varepsilon$, we have $\varepsilon^{1/3}|B'(y)|^{2/3} > \sigma_0 \varepsilon^{\frac{N+1}{N+3}}$. Dividing \eqref{effective} by $\varepsilon$ and viewing the coefficients as being fixed at $z$, we would expect that 
\begin{align}
    |K(y,z,\lambda)| &\lesssim L(z)\varepsilon^{-1} e^{-c_0 \frac{|y-z|}{L(z)}},\\
    L(z) &\approx \frac{\varepsilon^{1/2}}{(|b(z) - \lambda|  + \varepsilon^{1/3} |B'(z)|^{2/3}   )^{1/2}}. 
\end{align}
Because $K(y,z,\lambda)$ is a fundamental solution,\footnote{$\overline{K(z,y,-\lambda)}$ is a fundamental solution for the problem with $b$ replaced by $-b$} the bounds should be symmetric with respect to $y,z$, which leads us to consider the symmetric version of the length scale $L(y,z,\lambda)$ in \eqref{fs3}. To ensure that the estimates are truly symmetric in $y,z$, $A$ from \eqref{fs1} should also depend on $y$. This can be done, but it turns out to be more convenient to express the bounds in a slightly asymmetric form where the role of $z$ as the location of the Dirac mass plays a distinguished role. The justification for viewing the operator as constant coefficient is that the fundamental solution varies on a much shorter length scale than the potential. Therefore, from the perspective of the fundamental solution, the coefficients are ``essentially constant." A similar observation plays an important role in WKB expansions, see~\cite[Chapter 7]{Miller06}.

As a consequence of the Proposition \ref{k:bdds} we have the following bounds for solutions of $\A h = g$. 
\begin{proposition}
\label{ai:bdds}
Let $g \in H^1(\T)$ and assume that $\norm{g}_{H^1(\T)} = 1$. We have the following bounds for $\ai(g)$:
\begin{subequations}
    \begin{align}
    |\ai(g)(y)| &\lesssim  \frac{1}{|b(y) - \lambda| + \varepsilon^{1/3}|B'(y)|^{2/3}} \label{g1} \\
    |\p_y\ai(g)(y)| &\lesssim \frac{|B'(y)|}{(|b(y) - \lambda| + \varepsilon^{1/3}|B'(y)|^{2/3})^2} + \frac{\varepsilon^{-1/4}}{(|b(y) - \lambda|  + \varepsilon^{1/3}|B'(y)|^{2/3}  )^{3/4}} \label{g2} \\
    |\p_y^2\ai(g)(y)| &\lesssim \frac{|B'(y)|\varepsilon^{-1/2}}{(|b(y) - \lambda| + \varepsilon^{1/3}|B'(y)|^{2/3})^{3/2}} + \frac{\varepsilon^{-3/4}}{(|b(y) -\lambda| + \varepsilon^{1/3}|B'(y)|^{2/3})^{1/4}  } \label{g3}
\end{align}
\end{subequations}
\end{proposition}
\begin{proof}
The bound \eqref{fs2} and the embedding $H^1(\T) \to L^\infty(\T)$ immediately imply~\eqref{g1}. For \eqref{g2} we set $h(y, \lambda) = \ai(g)$ and differentiate the equation to~get 
 \begin{align}
 \label{g4}
    -\varepsilon \p_y^2 \p_y h+ i (b(y) - \lambda) \p_y h= - ib'(y) h + \p_y g.  
 \end{align}
The first term on the right-hand side of \eqref{g2} is obtained by combining \eqref{g1} with \eqref{fs2} and using that $\left(\frac{L(z)}{L(y)}\right)^\theta e^{-c_0\frac{L(z)}{L(y)}} \lesssim 1, \theta \geq 0 $. The second term comes from  \eqref{fs2} combined with $\p_yg \in L^2$. We comment briefly on the appearance of $|B'(y)|$. Here we are using that the maximal length scale determined by $K(y,z)$ when $y$ is near a critical point of order $m$ is $ \varepsilon^{\frac{1}{m + 3}}$  which converts the \emph{vanishing} of $b'(y)$ at a critical point into the \emph{smallness} of $|B'(y)|$. Finally, \eqref{g3} follows from \eqref{g4} and the derivative kernel bounds \eqref{fs4}. 
\end{proof}

We introduce some useful notation for functions which frequently occur in the bounds for quantities involving the Airy kernel. Let $c_*, c_{**} \in (0,1)$ and $\ell \in (0,1 )$  satisfy $ \varepsilon^{\frac{1}{N+3}}\lesssim \ell \lesssim \sigma_\sharp$.   We define $E_{I, \ell}(y), \tilde{E}_{I,\ell}(y), E_{O,\ell}(y),  \tilde{E}_{O,\ell}(y): [-\pi, \pi] \to \R$  as follows:

\begin{align}
        E_{I,\ell}(y) := 
        \begin{cases}
           e^{-c_* \left(\frac{|B'(\ell)|^{1/3}}{\varepsilon^{1/3}} \right)|y - \ell|} & \quad y \in [0, \ell] \\
            1 & \quad y \in [\ell, \pi] \\
            E_{I,\ell}(-y) & \quad y \in [-\pi, 0]
        \end{cases}
    \end{align}
\begin{align}
        \tilde{E}_{I,\ell}(y) := 
        \begin{cases}
           e^{-c_{**} \left(\frac{|B'(\ell)|^{1/3}}{\varepsilon^{1/3}} \right)|y - \ell|} & \quad y \in [0 , \ell] \\
            1 & \quad y \in [\ell, \pi] \\
            \tilde{E}_{I,\ell}(-y) & \quad y \in [-\pi, 0]
        \end{cases}
    \end{align}
 \begin{align}
        E_{O,\ell}(y) := 
        \begin{cases}
            1 & \quad y \in [0, \ell] \\
            e^{-c_* \left(\frac{|B'(\ell)|^{1/3}}{\varepsilon^{1/3}} \right)|y - \ell|} & \quad y \in [\ell, \pi] \\
            E_{O,\ell}(-y) & \quad y \in [-\pi, 0]
        \end{cases}
\end{align}
\begin{align}
        \tilde{E}_{O,\ell}(y) := 
        \begin{cases}
            1 & \quad y \in [0, \ell] \\
            e^{-c_{**} \left(\frac{|B'(\ell)|^{1/3}}{\varepsilon^{1/3} } \right)|y - \ell|} & \quad y \in [\ell, \pi] \\
            \tilde{E}_{O,\ell}(-y) & \quad y \in [-\pi, 0].
        \end{cases}
\end{align}


We have that $c_{**} < c_* < c_0$ where $c_0$ is as in Proposition \ref{k:bdds} and the values of $c_*$ and $c_{**}$ will be fixed in the proofs of Lemma \ref{lemE1} and \ref{lemE2} below. $E_{O,\ell}$ and $E_{I,\ell}$ appear when measuring the size of  $\ai(g)$ at a distance approximately $\ell$ outside the support of $g$. $\tilde{E}_{O,\ell}$ and $\tilde{E}_{I,\ell}$ occur when we feed $E_{O,\ell}$ and $E_{I,\ell}$, respectively, through the Airy kernel. 


\begin{lemma}
\label{lemE1}
Let $0$ be a critical point of $b$. Suppose that  $g \in L^\infty(\T)$ and $g$ satisfies either 
\begin{enumerate}
    \item $\text{supp } g \subset [-\ell, \ell]$,
    \item $g \equiv 0$ on $[-\ell , \ell]$.
\end{enumerate}
where we assume $\ell \leq \sigma_\sharp$. 
Then in case (1):
\begin{align}
\label{OL}
    |\ai(g)(y)| \lesssim \frac{E_{O,\ell}(y)  \norm{g}_{L^{\infty}}  }{|b(y)- \lambda | + \varepsilon^{1/3}|B'(y)|^{2/3}},
\end{align}
and in case (2):
\begin{align}
\label{IL}
    |\ai(g)(y)| \lesssim \frac{E_{I,\ell}(y)  \norm{g}_{L^{\infty}}  }{|b(y) - \lambda | + \varepsilon^{1/3}|B'(y)|^{2/3}}.
\end{align}
\end{lemma}
\begin{proof}
    We only prove \eqref{IL} as the proof of \eqref{OL} is simpler. First, observe that for $|y|\geq \ell$ and the definition $E_{I,\ell}$, \eqref{IL} follows from \eqref{g1}. Therefore, it suffices to consider the case $|y| \leq \ell$. Assume that $y \geq 0$.  We have 
    \begin{align}
        \ai(g)(y) = \int_{-2\sigma_\sharp}^{ 2\sigma_\sharp} K(y,z,\lambda) g(z) dz + \int_{\T \setminus [-2\sigma_\sharp,  2\sigma_\sharp] }K(y,z,\lambda) g(z) dz.
    \end{align}
We focus on the first integral. Observe that for $y \in [0, \ell] $ and $|z| \geq \ell$ we have $|y-z| \geq |y-\ell|$.  Since $|z| \in [\ell, 2\sigma_\sharp]$  we also have that 
\begin{align}
    \frac{1}{L(y,z, \lambda)} \gtrsim \frac{|B'(\ell)|^{1/3}}{\varepsilon^{1/3}}.
\end{align} 
Therefore,
\begin{align}
    \int_{-2\sigma_\sharp}^{ 2\sigma_\sharp} K(y,z,\lambda) g(z) dz \lesssim e^{-\frac{c_0}{2} \left(  \frac{|B'(\ell)|^{1/3}}{\varepsilon^{1/3}} \right) |y- \ell|}\int_{-2\sigma_\sharp}^{ 2\sigma_\sharp} \epsilon^{-1} L(z) e^{-\frac{c_0}{2} \frac{|y -z|}{L(y,z,\lambda)}  } |g(z)| dz 
\end{align}
which is consistent with the desired result by choosing $c_*$ sufficiently small in terms of $c_0$. The contribution from the integral over $\T\setminus[-2 \sigma_\sharp, 2 \sigma_\sharp]$ satisfies better bounds on account of the improved decay of the fundamental solution for $|y-z| \geq \sigma_\sharp$. For $y < 0$, the above argument works with $\ell$ replaced by $-\ell$. This concludes the proof.



\end{proof}
Finally, we have the following relation between $E_{*,\ell}$ and $\tilde{E}_{*, \ell}, * \in \{ I, O \}$.
    \begin{lemma}
    \label{lemE2}
    Let 0 be a critical point of $b$. We have the following bounds: 
\begin{subequations}
\begin{align}
   |\ai(E_{O,\ell})(y)| & \lesssim \frac{\tilde{E}_{O, \ell}(y)}{|b(y) -\lambda| + \varepsilon^{1/3}|B'(y)|^{2/3}}, \label{OL1} \\
       |\ai(E_{I,\ell})(y)| &\lesssim \frac{\tilde{E}_{I, \ell}(y)}{|b(y) -\lambda| + \varepsilon^{1/3}|B'(y)|^{2/3}}. \label{IL1} 
    \end{align} 
\end{subequations}
\end{lemma}
\begin{proof}
    We start with the  proof of \eqref{IL1} and briefly comment on the changes needed for the proof of \eqref{OL1} at the end of the proof. Recall that we can take $c_{**}$ to be sufficiently small with respect to $c_0, c_*$. For $|y|\geq \ell$, \eqref{IL1} follows from \eqref{g1}. Therefore, it suffices to consider the case $|y| \leq \ell$. Using \eqref{fs2} we have 
    \begin{align}
       |\ai(E_{I,\ell})(y)| \lesssim \int_{\T} E_{I,\ell}(z) \varepsilon^{-1} L(z) e^{-c_0\frac{|y-z|}{L(y,z, \lambda )}} dz.
    \end{align}
  We decompose the integral as 
  \begin{align}
     \int_{\T}  = 
     \int_{|z| \leq \ell }   +  \int_{|z| \geq \ell }. 
  \end{align}
We assume that $y > 0$, the case with $y <0 $ being completely analogous. Using that $y \leq \ell$, for $|z|\geq \ell$ we have 
\begin{align}
        e^{-c_0\frac{|y-z|}{L(y,z,\lambda)}}  &\geq e^{-\frac{c_0}{3}\frac{|y-\ell|}{L(y,z,\lambda)}} e^{-\frac{c_0}{3}\frac{|\ell-z|}{L(y,z,\lambda)}}. \\
\end{align}
For $|z|\geq \ell $ we have
\begin{align}
    \frac{1}{L(y,z, \lambda)} \gtrsim \frac{|B'(\ell)|^{1/3}}{\varepsilon^{1/3}}.
\end{align}
Therefore,
\begin{align}
  \int_{|z| \geq \ell }   E_{I,\ell}(z) \varepsilon^{-1} L(z) e^{-c_0\frac{|y-z|}{L(y,z, \lambda )}} dz 
  &\lesssim e^{-c_* \left(  \frac{|B'(\ell)|^{1/3}}{\varepsilon^{1/3}} \right)|y - \ell|} \int_{|z| \geq \ell} \varepsilon^{-1} L(z) e^{-c_*\frac{|z- y|}{L(y,z, \lambda)} } dz \\
      &\lesssim   L^2(y)\varepsilon^{-1}e^{-c_*\frac{|B'(\ell)|^{1/3}}{\varepsilon^{1/3}}|y - \ell|} \\
  &\lesssim  \frac{\tilde{E}_{I, \ell}(y)}{|b(y)-\lambda| + \varepsilon^{1/3}|B'(y)|^{2/3}} \, .
\end{align}
Now for $|z| \leq \ell$, we exploit the exponential decay of $E_{I,\ell}$:
\begin{align}
    \int_{|z| \leq \ell} E_{I,\ell}(z) \varepsilon^{-1} L(z) e^{-c_0\frac{|y-z|}{L(y,z, \lambda )}} dz 
    &\lesssim \int_{|z| \leq \ell} e^{-c_* \left(\frac{|B'(\ell)|^{1/3}}{\varepsilon^{1/3}} \right) |z - \ell|} \varepsilon^{-1} L(z) e^{-c_0\frac{|y-z|}{L(y,z, \lambda )}} dz  \\
    &\lesssim 
    \int_{|z| \leq \ell} e^{-c_* \left(\frac{|B'(\ell)|^{1/3}}{\varepsilon^{1/3}} \right) |y - \ell|} \varepsilon^{-1} L(z) e^{-c_*\frac{|y-z|}{L(y,z, \lambda )}} dz\\
    &\lesssim \frac{\tilde{E}_{O, \ell}(y)}{|b(y) -\lambda| + \varepsilon^{1/3}|B'(y)|^{2/3}},
    \end{align}
where we have used that for $|y|, |z| \leq \ell $ we have
\begin{align}
\frac{|y -z|}{L(y,z, \lambda)} + \frac{|B'(\ell)|^{1/3}}{\varepsilon^{1/3}}||z - \ell| \gtrsim \frac{|B'(\ell)|^{1/3}}{\varepsilon^{1/3}}|y - \ell|.
\end{align}
This can be proved by considering separately the cases $|y-\ell| \approx |z-\ell|$ and the case where $|y -\ell| \gg |z - \ell|$. For the proof of \eqref{OL1}, the proof is similar with the roles of $y$ and $z$ switched. In the case where $|y|, |z| \geq \ell$ we can instead use that  $|B'(\ell)| \lesssim |B'(y)|+ |B'(z)|$.
\end{proof}




\subsection{Overview of the proof}

In the sequel it will be convenient to decompose various functions based on the the critical points of $b(y)$ which we now describe. Let $ \{ \gamma_j\}, j\geq 1$ be an enumeration of all of  the critical points. Let $\chi\in C_c^\infty(-2,2)$ be a cutoff function satisfying $\chi(y) \equiv 1, y \in [-1, 1]$.  

Define
\begin{align}
\label{chij}
    \chi_j(y) := \chi\left( \frac{y- \gamma_j}{\sigma_{\sharp}}  \right),
 \end{align}
where $\sigma_\sharp$ is as in Assumption \ref{assump:b}. This implies that
\begin{align}
\label{chijl}
    \text{supp} \chi_j  \cap \text{supp} \chi_l = \emptyset, \quad  j \neq l. 
\end{align}
Furthermore, define $\chi_0$ so that 
\begin{align}
    \label{chi0}
    1 = \chi_0(y) + \sum_{j\geq 1} \chi_j(y).
\end{align}
We apply these cutoffs to the initial data $f_*^{\rm in}$ to get  
\begin{align}
\label{finj}
    f_j^{\rm in }(y) :=  f_*^{\rm in}(y) \chi_j(y)  , \quad j \geq 0.
\end{align}
Let $f_j(t,y)$ be the associated solution of \eqref{eq:passive:k} with initial data $f_j^{\rm in}(y)$ for $j \geq 0$. By linearity, it suffices to prove that Theorem \ref{thm:1} holds for each $j$. We will treat the cases $j \geq 1$ and $j =0$ separately, focusing mainly on $j \geq 1$ as this case is harder. Fix an index $j \geq 1$ and let the associated critical point be of order $m$. We introduce the length scale
\begin{align}
\label{length}
    \ell(t) := \frac{\sigma_\sharp}{10}\max( (1 +t)^{-\frac{1}{N + 1}}, \varepsilon^{\frac{1}{N + 3}}),
\end{align}
and define 
    \begin{align}
     \chi_{ \ell(t)} (y) &:= \sum_{j \geq 1}\chi \left( \frac{y-\gamma_j}{\ell(t)} \right). \label{crit1}
    \end{align}
Note that $\ell(t)$ satisfies
\begin{align}
    \frac{\varepsilon^{1/3}}{|B'(y)|^{1/3}} \lesssim \ell(t),
\end{align}
with the quantities being comparable when $y$ is near a critical point $\gamma_j$ of order $N$ and $|y - \gamma_j| \lesssim \varepsilon^{\frac{1}{N + 3}}$. We further decompose $f_j^{\rm in}$ as 
\begin{subequations}
    \begin{align}
        f_{I}^{\rm in }(y) &:= f_j^{\rm in} \chi_{j,  \ell(t)}(y), \label{in} \\
        f_O^{\rm in }(y) &:= f_j^{\rm in} (1-\chi_{j,  \ell(t)}(y) )\label{out}, 
     \end{align}
\end{subequations}
and decompose $f_j(t,y)$ as 
\begin{align}
    f_j(t,y) = f_I(t,y) + f_O(t,y),
\end{align}
where
\begin{subequations}
    \begin{align}
        f_{I}(t,y) &= e^{L_\varepsilon t}f_I^{\rm in}, \label{ievo} \\
        f_{O}(t,y) &= e^{L_\varepsilon t}f_O^{\rm in} \label{oevo}.
    \end{align}
\end{subequations}

With the notation set, we now give an outline of the proof. Using the representation formula \eqref{rep} we will employ a similar idea as used in the inviscid case for proving the desired mixing estimates. However, one key issue in the viscous setting is the need to \emph{saturate the cutoffs at an appropriate viscous scale}, determined by the (maximal) order of vanishing of the derivative at the critical points. This is precisely the reason for the definition of \eqref{length}.
Essentially, the viscous case can be treated like the inviscid case until the viscous effects become dominant. This at first appears as an obstacle to proving global-in-time decay results as in the inviscid case. However, precisely at this time, \emph{enhanced dissipation combined with smallness of the viscous region} allow the mixing estimates to be proven for all time.

On a more technical level, to study the inner problem we make use of duality and employ the $L^\infty$ enhanced dissipation estimates of \cite{GLM24}, see remarks 2.2 and 2.3 therein. For the outer problem, we seek to use the oscillations coming from $e^{-i\lambda t}$ in \eqref{rep}. However, unlike the inviscid case \eqref{int3}, we require regularity of the spectral density function $f(y, \lambda)$ with respect to the spectral parameter $\lambda$ instead of the physical variable $y$. A key difficulty in obtaining regularity with respect to $\lambda$ is the singularity in the resolvent equation \eqref{reso} when $b(y) = \lambda$. Instead of directly differentiating \eqref{reso} with $\p_\lambda$, it is better to track $\p_\lambda$ through a combination of $\p_y$ and the so called good derivative 
$D_\lambda \approx \frac{1}{b'(y)}\p_y + \p_\lambda$. Formally, $\frac{1}{b'(y)}\p_y + \p_\lambda$  commutes with $b(y) - \lambda$ but the critical points of $b(y)$ necessitate modifying the definition of $D_\lambda$ so that it is no longer singular at the critical points. This introduces the commutator $[D_\lambda, b(y) -\lambda ]$. There is also the ``viscous" commutator term $[\varepsilon \p_y^2,  D_\lambda]$. By appropriately designing $D_\lambda$, it is possible to leverage the rapid decay of the Airy kernel away from the support of its data to ensure that $[D_\lambda, b(y) -\lambda ]$ is consistent with the desired bounds. For the terms arising from $[\varepsilon \p_y^2, D_\lambda]$, the essential gain comes from the correct regularization scale and the smallness from the  $\varepsilon$ prefactor. 

Before beginning the proof of the main theorem, we state an $L^\infty$ enhanced dissipation bound which is slightly modified from \cite{GLM24}.
\begin{theorem}[$L^\infty$ enhanced dissipation]
\label{edlinf}
    Let $f(t,y): [0, \infty) \times \T \to \C$ satisfy \eqref{reduced} with initial data $f^{\rm in }$ and assume that $b(y)$ satisfies Assumption~\ref{assump:b}. Then there exist $\varepsilon_0, c_1 > 0$ such that for all $\varepsilon \in (0, \varepsilon_0)$ and for all $t > 0$ we have that
    \begin{align}
       \norm{f(t)}_{L^\infty} \lesssim e^{-c_1 \varepsilon^{\frac{N+1}{N+3}}t  } \norm{f^{\rm in}}_{L^\infty}. 
    \end{align}
\end{theorem}
This differs from the statement of Corollary 1 in \cite{GLM24} which is stated for solutions of \eqref{adv:diff}. We can handle this by applying the corollary to $F(t,x,y) := f(t,y)e^{ix} + \bar{f}(t,y) e^{-ix} $ where $\bar{f}$ denotes complex conjugate. Taking $x = 0, \pi/2$ yields the real and imaginary part of $f(t,y)$ respectively in terms of $F$.

\subsection{Proof of Theorem~\ref{thm:1}}
\label{mainproof}
Recall that it suffices to only consider a fixed index $j$. We assume that $j \geq 1$ and include the relevant modifications in Section~\ref{j=0} for $j = 0$.  We start with the estimates for the inner solution $f_I(t,y)$. Recalling the dual characterization of $H^{-1}$, for $\phi \in H^1(\T)$
it follows from Theorem~\ref{edlinf} that 
\begin{align}
\label{inn2}
    \lb e^{t L_\varepsilon }f_{I}^{\rm in }, \phi \rb = 
   \int_{\T} f_I^{\rm in}(y) e^{tL_\varepsilon} \bar{\phi} \, dy  
     &\lesssim |\text{supp } f_I^{\rm in } | \norm{f_I^{\rm in } }_{L^\infty} e^{-c_1 \varepsilon^{\frac{N+1}{N+3} } t} \norm{\phi}_{L^\infty }  \\
         &\lesssim \ell(t) e^{-c_1 \varepsilon^{\frac{N+1}{N+3} } t}  \norm{ f^{\rm in }}_{L^\infty}, \notag
 \end{align}
which yields the desired estimate. 

We now consider the outer solution $f_O$. We define the good derivative $D_\lambda$:
\begin{subequations}
   \label{good:d}
    \begin{align}
 D_\lambda &:= a(t,y) \p_y + \p_\lambda,\label{thed} \\
    a(t,y) &:= \frac{1 - \chi_{\frac{\ell(t)}{2}}(y) }{b'(y)}.\label{modd}
\end{align}
\end{subequations}
\begin{remark}
    The reason for choosing $\frac{\ell(t)}{2}$ in the definition of $a(t,y)$ is for handling  the ``Good Derivative commutator"   \eqref{m7}. This definition ensures that the commutator term and $f_O$ have a weak interaction. 
 \end{remark}
From Theorem \ref{rep}, we can represent $f_O(t,y)$ with the formula
\begin{align}
\label{m1}
    f_O(t,y) = \frac{e^{ -\sigma_0 \varepsilon^{\frac{N+1}{N+3}}  t}}{2 \pi}\int_{\R} e^{-i \lambda t  } f_O(y, \lambda) d \lambda,
\end{align}
where $f_O(y,\lambda)$ satisfies
\begin{align}
\label{m2}
    -\varepsilon\p_y^2 f_O(y, \lambda)   -\sigma_0 \varepsilon^{\frac{N+1}{N+3}}f_O(y, \lambda) + i(b(y) - \lambda)f_O(y, \lambda) = f_O^{\rm in}(y).
\end{align}
Integrating by parts with respect to $\lambda$ and using the definition of the good derivative \eqref{good:d} yields
\begin{align}
\label{m3}
  &\int_{\R} e^{-i\lambda t} f_{O}(y, \lambda) d \lambda =  
   \frac{1}{it} \int_\R e^{-i\lambda t}  \p_\lambda f_{O}(y, \lambda) d \lambda  \\
  &= \frac{1}{it} \int_\R e^{-i\lambda t} \left( D_\lambda f_{O}(y, \lambda)     -a(t,y ) \p_y f_{O}(y, \lambda) \right) d \lambda.  \notag
\end{align}
Applying $\D$ to \eqref{m2} yields
\begin{align}
\label{m4}
    &-\varepsilon\p_y^2 \D f_O -  \sigma_0\varepsilon^{\frac{N+1}{N+3}} \D f_O + i(b(y) -\lambda) \D f_O \\
    &=  a \p_y f_O^{\rm in}  + \chi_{\frac{\ell(t)}{2}}(y)f_O  -  \varepsilon a''(t,y) \p_y f_O - 2\varepsilon a'(t,y) \p_y^2 f_O.  \notag
\end{align}
Combining \eqref{m3}, \eqref{m4}, and recalling that 
$f_O(y, \lambda) = A_\varepsilon^{-1} (f_O^{\rm in })    $
we have 
\begin{align}
\label{m5}
    &it\int_\R e^{-i \lambda t} f_O(y, \lambda) d\lambda \\ &= \int_\R e^{-i \lambda t}  \left[ A_\varepsilon^{-1} (a \p_y f_O^{\rm in})  + A_\varepsilon^{-1} ( \chi_{\frac{\ell(t)}{2}} f_O) - A_\varepsilon^{-1}( \varepsilon a''(t, \cdot ) \p_y f_O ) - A_\varepsilon^{-1}(2\varepsilon a'(t,\cdot) \p_y^2 f_O)\right] d \lambda \notag \\
    &- \int_\R e^{-i \lambda t}  a(t,y)\p_yA_\varepsilon^{-1}(f_{O}^{\rm in}) d \lambda.  \notag
\end{align}
We will show that each term in \eqref{m5} is bounded up to an implicit constant by $ \ell^{-N}(t) \norm{f^{\rm in}}_{H^1} $. This combined with the factor $e^{-\sigma_0  \varepsilon^{\frac{N+1}{N+3}   }t} $ in \eqref{m1} will imply Theorem \ref{thm:1}. We categorize the terms as follows:
\begin{enumerate}
    \item Initial data 
    \begin{align}
    \label{m6}
      A_\varepsilon^{-1} (a \p_y f_O^{\rm in}) , \quad   a(t,y)\p_y A_\varepsilon^{-1} ( f_O^{\rm in}),
\end{align}
    \item Good Derivative commutator
    \begin{align}
    \label{m7}
        \ai(\chi_{\frac{\ell(t)}{2}} f_O(\cdot, \lambda)),
    \end{align}
    \item Viscous errors
    \begin{align}
    \label{m8}
        \ai(\varepsilon  a''(t, \cdot)  \p_y f_O(\cdot, \lambda) ), \quad \ai(\varepsilon a' (t, \cdot ) \p_y^2 f_O(\cdot, \lambda)).
    \end{align}
\end{enumerate}
We begin with the initial data terms \eqref{m6}. Recall that  
\begin{align}
   \frac{e^{-\sigma_0 \varepsilon^{\frac{N+1}{N+3}}}}{2\pi}\int_\R e^{-i \lambda t}  A_\varepsilon^{-1} (a \p_y f_O^{\rm in}) d\lambda = e^{t L_\varepsilon}(a \p_y f_O^{\rm in}). 
\end{align}
Using Theorem \ref{edlinf} implies
\begin{align}
    \lb e^{tL_\varepsilon} (a \p_y f_O^{\rm in}), \phi \rb   = \int_{\T} a \p_y f_O^{\rm in} \, e^{tL_\varepsilon} \bar{\phi}\,  dy \leq \norm{a \p_y f_O^{\rm in }}_{L^1} \norm{ e^{tL_\varepsilon} \bar{\phi}}_{L^\infty}  \lesssim \norm{a \p_y f_O^{\rm in}}_{L^1} e^{-\sigma_0  \varepsilon^{\frac{N+1}{N+3}   }t} . 
\end{align}
The definition of $f_O^{\rm in} $ \eqref{out} implies
\begin{align}
    \norm{a \p_y f_O^{\rm in}}_{L^1} &\lesssim \norm{a (1-\chi_{\ell(t)} )\p_y f_j^{\rm in} }_{L^1} + \norm{a \chi_{ \ell(t)}'  f_j^{\rm in }}_{L^1} \\
    &\lesssim \norm{a}_{L^2}\norm{\p_y f_j^{\rm in }}_{L^2} + \norm{a}_{L^1} \norm{\chi_{\ell(t)}'}_{L^{\infty}} \norm{f_j^{\rm in}}_{L^\infty}\notag\\
    &\lesssim \ell^{-N}(t)\norm{f^{\rm in}}_{H^1}.
\end{align}
Next, we observe that 
\begin{align}
    \frac{e^{-\sigma_0 \varepsilon^{\frac{N+1}{N+3}}}}{2\pi}\int_\R e^{-i \lambda t}  a(t,y) \p_yA_\varepsilon^{-1} ( f_O^{\rm in}) d\lambda = a(t,y) \p_ye^{t L_\varepsilon}(f_O^{\rm in}).
    \end{align}
Therefore, for $\phi$ satisfying $\norm{\phi}_{H^1} = 1$,
\begin{align}
    \lb \phi,  a(t,y) \p_y e^{t L_\varepsilon} (f_O^{\rm in}) \rb 
    &= -\lb \p_y(a (t,y) \phi ), e^{t L_\varepsilon} f_O^{\rm in} \rb  \\
    &\lesssim  e^{-\sigma_0  \varepsilon^{\frac{N+1}{N+3}   }t} (\norm{a'}_{L^1} \norm{\phi}_{L^\infty} \norm{f_{O}^{\rm in}}_{L^\infty }       + \norm{a}_{L^2} \norm{\p_y \phi}_{L^2} \norm{ f_O^{\rm in } }_{L^\infty}) \\
    &\lesssim \ell^{-N}(t) e^{-\sigma_0  \varepsilon^{\frac{N+1}{N+3}   }t}  \norm{  f^{\rm in} }_{H^1}.
\end{align}
This completes the proof for the initial data terms. 

\begin{remark}
\label{rem:local}
Before estimating \eqref{m7} and \eqref{m8}, we highlight some useful simplifications.  We may assume that $\gamma_j =0$. This amounts to shifting the domain of integration. By the embedding $L^1(\T) \to H^{-1}(\T)$ it suffices to estimate \eqref{m7} and \eqref{m8} in $L_{y,\lambda}^1$. The terms in  \eqref{m7} and \eqref{m8} are of the form
\begin{align}
    \ai( g \p_y^{\alpha} f_O(\cdot, \lambda) ), \quad \alpha \in \{ 0, 1, 2 \}, 
\end{align}
where $\norm{g}_{L^\infty (\T)} \lesssim \epsilon^{-2}$. We claim that  
\begin{align}
    \int_{ \T \setminus [ - 3 \sigma_\sharp , 3 \sigma_\sharp]} \int_{\R} |\ai( g \p_y^{\alpha} f_O )|   d \lambda  dy\lesssim \ell^{-N}(t). 
\end{align}
To see this, recall from the definition of $f_O^{\rm in}$~\eqref{out} and the definition of $f_j^{\rm in}$~\eqref{finj} that the support of $f_O^{\rm in}$ is contained in $[- 2 \sigma_\sharp,   2\sigma_\sharp ]$. Therefore, from equation \eqref{m2} and the proof of Lemma~\ref{lemE1}, we can prove, for some $\delta_1 >0$ and $\beta_1 \in (2\sigma_{\sharp} ,2.5\sigma_\sharp)$, both independent of $\varepsilon$, that 
\begin{align}
    |f_O(y, \lambda)| \lesssim  e^{-\delta_1 \frac{d(y, I_{\beta_1}(0))}{\varepsilon^{1/2}} }, \quad I_{\beta}(0) := [- \beta,  \beta ]
\end{align}
where $d(y, I)$ is the distance to interval $I$. Moreover, for $\delta_2 > 0$ satisfying $\delta_2 \leq \delta_1$ and $ \beta_2 \in (\beta_1,  2.5 \sigma_\sharp)$, both independent of $\varepsilon$, we can use the fundamental solution bounds from Proposition \ref{k:bdds} to deduce that 
\begin{align}
    |\ai(f_O)(y, \lambda)| \lesssim e^{-\delta_2 
    \frac{ d(y, I_{\beta_2}(0) )  }   {\varepsilon^{1/2}}       }. 
\end{align}
This argument can be iterated an arbitrary number of times depending on how many times the operator $\ai$ is applied. Using this observation, in all of the terms appearing in the estimates of \eqref{m7} and \eqref{m8} (see \eqref{gdcom}, \eqref{v11}, \eqref{v22}, \eqref{v33}, \eqref{v44}, \eqref{v55}, and \eqref{v66}), for $y \in [ - 3 \sigma_\sharp, 3 \sigma_\sharp]^c $ we can obtain bounds of the form
\begin{align}
     |\ai( g \p_y^{\alpha} f_O(\cdot, \lambda) )(y)| \lesssim \frac{ \epsilon^{-3}    e^{-\frac{\delta}{\epsilon^{1/2}}  }    }   {(|b(y) - \lambda| + \epsilon^{1/3}|B'(y)|^{2/3})^{ \gamma} }, \quad  \gamma \in (1, 3].
\end{align}
for some $\delta > 0$. Therefore,
\begin{align}
    \int_{\R} |\ai( g \p_y^{\alpha} f_O(\cdot, \lambda) )|  
    d \lambda \lesssim \epsilon^{-5} e^{-\frac{\delta}{\epsilon^{1/2}}} \lesssim 1
\end{align}
which yields the desired result since $\ell(t) \leq 1$. Consequently, Theorem \ref{thm:1} will follow if we can prove 
\begin{align}
    \int_{- 3 \sigma_\sharp}^{ 3 \sigma_\sharp} \int_{\R} |\ai( g \p_y^{\alpha} f_O )|   d \lambda  dy\lesssim \ell^{-N}(t). 
\end{align}
 \end{remark}
We now estimate the commutator term \eqref{m7}. In the following, we will frequently make use of the fact that 
\begin{align}
\label{ellcomp}
    \ell(t) \gtrsim \frac{\varepsilon^{1/3}}{|B'(\ell)|^{1/3}}
\end{align}
which saturates when $0$ is a critical point of order $N$ and $ \ell(t) \approx \varepsilon^{\frac{1}{N+3}}$.  
The proofs of Proposition \ref{ai:bdds} , Lemma \ref{lemE1},  and Lemma \ref{lemE2}  yield
\begin{align}
\label{gdcom}
    |\ai(\chi_{\frac{\ell(t)}{2}} f_O(\cdot, \lambda)) (y, \lambda)| \lesssim     \frac{\norm{f_j^{\rm in}}_{L^\infty}  E_{O, \ell(t)}(y) \tilde{E}_{I ,\ell(t)}(y)  }{(|b(y) -\lambda|   + \varepsilon^{1/3} |B'(y)|^{2/3})^2   },
\end{align}
where we have used $\text{supp}(1 - \chi_{\ell(t)}) \subset [-\ell(t), \ell(t)]^c  $ and $\text{supp} (\chi_{\frac{\ell(t)}{2}}) \subset [-\ell(t), \ell(t)] $. Therefore, by Remark \ref{rem:local}, we consider
\begin{align}
\label{term1}
    &\int_{3 \sigma_\sharp}^{3 \sigma_\sharp}    \tilde{E}_{I, \ell(t)}(y) E_{O, \ell(t)}(y)    \int_\R \frac{1}{(|b(y) -\lambda|   + \varepsilon^{1/3} |B'(y)|^{2/3})^2} d \lambda dy\\ 
    \lesssim &\int_{3 \sigma_\sharp}^{3 \sigma_\sharp} \frac{\tilde{E}_{I, \ell(t)}(y) E_{O, \ell(t)}(y) }{\varepsilon^{1/3} |B'(y)|^{2/3}  } dy.
\end{align}
Using the evenness of $\tilde{E}_{I, \ell(t)}$ and $ E_{O, \ell(t)}(y)$ it suffices to bound the integral on $[0, 3 \sigma_\sharp]$. 
Strictly speaking, $|B'(y)|$ is not even. However, due to the localization around $\pm \ell(t)$ that we gain from $\tilde{E}_{I, \ell(t)}$ and $E_{O, \ell(t)}$, only the behavior near the critical point is important. Since $|B'(y)| \approx |y|^m + \varepsilon^{\frac{m}{m+3}}$, assuming that $|B'(y)|$ is even does not significantly affect the bounds. Therefore, there is no loss of generality by only considering $y > 0$.  
Decomposing the integral as
\begin{align}
    \int_{0}^{3 \sigma_\sharp}  = \int_{0}^{\ell(t)} + \int_{\ell(t)}^{3 \sigma_\sharp}, 
\end{align}
we can bound the integral over $[0, \ell(t)]$  as 
\begin{align}
    \int_0^{\ell(t)} \frac{\tilde{E}_{I, \ell(t)}(y) }{\varepsilon^{1/3} |B'(y)|^{2/3}  } dy &\approx \int_0^{\ell(t)} \frac{  e^{  -c_{**}  \frac{|B'(\ell)|^{1/3}}{\varepsilon^{1/3}}  |y - \ell(t)|    }}{|B'(y)|^{2/3} |B'(\ell)|^{1/3} } \frac{|B'(\ell)|^{1/3}}{\varepsilon^{1/3}} dy \\
    &\lesssim  \frac{1}{|B'(\ell(t))|} 
    \lesssim  \ell^{-N}(t).
\end{align}
For the integral on the complement we obtain 
\begin{subequations}
\label{outerest}
    \begin{align}
   \int_{\ell(t)}^{3 \sigma_\sharp} \frac{E_{O, \ell(t)}(y)}{\varepsilon^{1/3} |B'(y)|^{2/3} } dy &\approx \int_{\ell(t)}^{3 \sigma_\sharp } \frac{e^{-c_* \left(\frac{|B'(\ell)|^{1/3}}
   {\varepsilon^{1/3}} \right) |y -\ell(t)|}}{\varepsilon^{1/3}|B'(y)|^{2/3} } dy \\
   &\lesssim \frac{1}{ |B'(\ell(t))| } \\
   &\lesssim \ell^{-N}(t).
\end{align}
\end{subequations}
Therefore, 
\begin{align}
   \norm{ \frac{e^{-\sigma_0 \varepsilon^{\frac{N+1}{N+3} } t}}{t} \int_\R \ai(\chi_{\frac{\ell(t)}{2}} f_O ) d \lambda }_{H^{ -1}} \lesssim \ell^{-N}(t)e^{-\sigma_0 \varepsilon^{\frac{N+1}{N+3} } t} \norm{f^{\rm in}}_{H^1},
\end{align}
as desired. 

Finally, we consider the viscous error terms \eqref{m8}. As for the commutator term, it suffices to bound the terms in \eqref{m8} in $L_{y,\lambda }^1$. Differentiating \eqref{m2} with respect to $y$ and then letting the second derivative land on the Airy kernel leads to estimating 
\begin{subequations}
    \begin{align}
    |\p_y f_O(y,\lambda)| &\lesssim  |\ai \left[b' f_O(\cdot, \lambda) \right]|  + |\ai\left[\p_yf_j^{\rm in} (1  - \chi_{j, \ell(t)} )\right]|  + |\ai \left[ f_j^{\rm in} \chi_{j, \ell(t)}' \right]|,\label{ai0.5}\\
    |\p_y^2 f_O(y,\lambda)| &\lesssim  |\p_y\ai \left[b' f_O(\cdot, \lambda) \right]|  + |\p_y \ai\left[\p_yf_j^{\rm in} (1  - \chi_{j, \ell(t)} )\right]|  + |\p_y\ai \left[ f_j^{\rm in} \chi_{j, \ell(t)}' \right]|. \label{ai0.6}
\end{align}
\end{subequations}
Applying the proof of Proposition \ref{ai:bdds} to \eqref{m2} yields
\begin{align}
\label{ai1}
    |\ai\left[\p_yf_j^{\rm in} (1  - \chi_{j, \ell(t)} )\right]| \lesssim \frac{\varepsilon^{-1/4} \norm{\p_y f_j^{\rm in}}_{L^2} }{ (|b(y) - \lambda| + \varepsilon^{1/3} |B'(y)|^{2/3})^{3/4}  },
\end{align}
\begin{align}
\label{ai1.5}
    |\p_y\ai\left[\p_yf_j^{\rm in} (1  - \chi_{j, \ell(t)} )\right]| \lesssim \frac{\varepsilon^{-3/4} \norm{\p_y f_j^{\rm in}}_{L^2} }{ (|b(y) - \lambda| + \varepsilon^{1/3} |B'(y)|^{2/3})^{1/4}  },
\end{align}
\begin{align}
    \label{ai2}
    |\ai \left[ f_j^{\rm in} \chi_{j, \ell(t)}' \right]| \lesssim  \frac{ \ell^{-1}(t)  E_{I, \ell(t) }(y) E_{O, 2 \ell(t) }(y) \norm{f_j^{\rm in}  }_{L^\infty}  }{|b(y) -\lambda| + \varepsilon^{1/3}|B'(y)|^{2/3} },
\end{align}
\begin{align}
    \label{ai2.5}
    |\p_y\ai \left[ f_j^{\rm in} \chi_{j, \ell(t)}' \right]| \lesssim  \frac{ \varepsilon^{-1/2} \ell^{-1}(t)E_{I, \ell(t)}(y) E_{O, 2 \ell(t)}(y) \norm{f_j^{\rm in}   }_{L^\infty}}{(|b(y) -\lambda| + \varepsilon^{1/3}|B'(y)|^{2/3} )^{1/2}},
\end{align}
\begin{align}
\label{ai3}
    |\ai \left[b' f_O(\cdot, \lambda) \right]| \lesssim \frac{ |B'(y)| \norm{f_j^{\rm in}}_{L^\infty}  }{(|b(y) - \lambda| + \varepsilon^{1/3} |B'(y)|^{2/3}  )^2},
\end{align}
\begin{align}
\label{ai3.5}
    |\p_y\ai \left[b' f_O(\cdot, \lambda) \right]| \lesssim \frac{ \varepsilon^{-1/2}|B'(y)| \norm{f_j^{\rm in}}_{L^\infty}  }{(|b(y) - \lambda| + \varepsilon^{1/3} |B'(y)|^{2/3}  )^{3/2}}.
\end{align}
We introduce 
\begin{align}
\label{rho}
    \varrho(y) := \min_{j} (|y-\gamma_j| + \ell(t)).
\end{align}
Using \eqref{rho} we can bound $a$ from \eqref{modd} and its derivatives as 
\begin{align}
    |\p_y^{\alpha }a(t,y)| \lesssim \frac{1}{\varrho^{N+ \alpha}(y)} \lesssim \ell^{-(N + \alpha)}(t), \quad \alpha \in \mathbb{N}.
\end{align}

We start with estimating $\ai( a''(t, \cdot) \p_y f_O (\cdot, \lambda) )$. From  \eqref{ai0.5}, we need to bound
\begin{align}
\label{v11}
    \ai( a''  \ai (b' f_O(\cdot, \lambda) )),  \quad  \ai( a''  \ai ( \p_y f_j^{\rm in}  (1 - \chi_{j, \ell(t)})    )), \quad \ai( a''  \ai(f_j^{\rm in} \chi_{j, \ell(t)}'   )  ). 
\end{align}
We have 
\begin{align}
\label{v22}
    |\ai \left[ a'' \ai(b' f_O) \right]| \lesssim  \frac{\ell^{-2}(t)\norm{f_j^{\rm in}}_{L^\infty}E_{I, \frac{\ell(t)}{2}  }(y) }{(|b(y) - \lambda |   + \varepsilon^{1/3}|B'(y)|^{2/3})^3 }  =: v_1 ,
\end{align}
\begin{align}
\label{v33}
     |\ai( a''  \ai ( \p_y f_j^{\rm in}  (1 - \chi_{j, \ell(t)})    ))| \lesssim \frac{\varepsilon^{-1/4}  \norm{\p_y f_j^{\rm in }}_{L^2} E_{I, \frac{\ell(t)}{2}}(y)  }{\varrho^{N+2}(y) (|b(y) -\ \lambda|  + \varepsilon^{1/3}|B'(y)|^{2/3}  )^{7/4}  } =: v_2, 
\end{align}
and 
\begin{align}
    |\ai( a''  \ai(f_j^{\rm in} \chi_{j, \ell(t)}'   )  )| \lesssim  \frac{ \ell^{-N-3}(t) \norm{f_j^{\rm in}}_{L^\infty} \tilde{E}_{I, \ell(t)}(y) \tilde{E}_{O, 2\ell(t)}(y)    }{(|b(y) - \lambda| + \varepsilon^{1/3}|B'(y)|^{2/3} 
 )^2}  =:v_3.
\end{align}
We now bound $\ai(\varepsilon a' (t, \cdot ) \p_y^2 f_O(\cdot, \lambda)) $. From \eqref{ai0.6}, we need to control
\begin{align}
    \ai( a'  \p_y\ai (b' f_O(\cdot, \lambda) )),  \quad  \ai( a'  \p_y\ai ( \p_y f_j^{\rm in}  (1 - \chi_{j, \ell(t)})    )), \quad \ai( a'  \p_y\ai(f_j^{\rm in} \chi_{j, \ell(t)}'   )  ) . 
\end{align}
We have 
\begin{align}
\label{v44}
    |\ai( a'  \p_y\ai (b' f_O(\cdot, \lambda) ))| \lesssim  \frac{ \varepsilon^{-1/2}\norm{f_j^{\rm in }}_{L^\infty}        E_{I, \frac{\ell(t)}{2}}(y)} { \varrho(y) (|b(y) -\lambda| + \varepsilon^{1/3} |B'(y)|^{2/3}  )^{5/2}  } =: v_4,
\end{align}
\begin{align}
\label{v55}
    |\ai( a'  \p_y\ai ( \p_y f_j^{\rm in}  (1 - \chi_{j, \ell(t)})    ))| \lesssim  \frac{\varepsilon^{-3/4} \norm{\p_y f_j^{\rm in}}_{L^2}  \tilde{E}_{I, \ell(t)}(y)}{ \varrho^{N + 1}(y) (|b(y) - \lambda| + \varepsilon^{1/3} |B'(y)|^{2/3})^{5/4}  } =:v_5 , 
\end{align}
and 
\begin{align}
\label{v66}
    |\ai( a'  \p_y\ai(f_j^{\rm in} \chi_{j, \ell(t)}'   )  )| \lesssim  \frac{ \varepsilon^{-1/2} \ell^{-1}(t)\tilde{E}_{I, \ell(t)}(y) \tilde{E}_{O, 2 \ell(t)}(y) \norm{f_j^{\rm in}   }_{L^\infty}}{ \varrho^{N+1}(y)  (|b(y) -\lambda| + \varepsilon^{1/3}|B'(y)|^{2/3} )^{3/2}} =:v_6.
\end{align}

Obtaining the bounds for $v_j, j \in \{1, ..., 6 \}$ is essentially the same, so we only show  $v_4$ and $v_6$. 
By scaling we can take $\norm{f_j^{\rm in}}_{L^\infty} = 1$. By Remark \ref{rem:local}, it suffices to only check the contribution for  $y \in [-3 \sigma_\sharp, 3 \sigma_\sharp]$ (recall that we have assumed that $\gamma_j = 0$). Starting with $v_4$, we have 
\begin{align}
    \int_{- 3\sigma_\sharp} ^{ 3\sigma_\sharp} \int_\R  \varepsilon v_4\,  d\lambda \, d y &=  \int_{- 3\sigma_\sharp} ^{ 3\sigma_\sharp}  \frac{E_{I, \frac{\ell(t)}{2}}(y) }{\varrho(y)}\int_\R \frac{ \varepsilon^{1/2}         }{  (|b(y) -\lambda| + \varepsilon^{1/3} |B'(y)|^{2/3}  )^{5/2}  } \, d\lambda \, d y \\
   &\lesssim \int_{- 3\sigma_\sharp} ^{ 3\sigma_\sharp} \frac{E_{I, \frac{\ell(t)}{2}}(y)}{ \varrho(y)|B'(y)|} dy \notag \\
   &\lesssim \ell^{-N}(t).  \notag
\end{align}
For $v_6$, 
\begin{align}
    \int_{- 3\sigma_\sharp} ^{ 3\sigma_\sharp}  \int_\R  \varepsilon v_6\,  d\lambda \, d y &= \ell^{-1}(t)\int_{- 3\sigma_\sharp} ^{ 3\sigma_\sharp}  \frac{ \tilde{E}_{I, \ell(t)}(y) \tilde{E}_{O, 2 \ell(t)}(y)}{\varrho^{N+1}(y)}\int_\R \frac{ \varepsilon^{1/2}  }{   (|b(y) -\lambda| + \varepsilon^{1/3}|B'(y)|^{2/3} )^{3/2}}\, dy\, d \lambda  \notag \\
    &\lesssim \ell^{-1}(t)\int_{- 3\sigma_\sharp} ^{ 3\sigma_\sharp}  \frac{ \varepsilon^{1/3}\tilde{E}_{I, \ell(t)}(y) \tilde{E}_{O ,2 \ell(t)}(y)}{|B'(y)|^{1/3} \varrho^{N+1}(y)} dy  \\
    &\lesssim \ell^{-N}(t) \notag
\end{align}
as desired. This completes the proof for $j \geq 1$.


\subsection{Changes needed for $j = 0$}
\label{j=0}
For $j = 0$, we proceed as for the outer solution $f_O(t,y)$ 
with 
\begin{align}
\label{m0.25}
    f_0(t,y) = \frac{e^{ -\sigma_0 \varepsilon^{\frac{N+1}{N+3}}  t}}{2 \pi}\int_{\R} e^{-i \lambda t  } f_0(y, \lambda) \, d \lambda,
\end{align}
where $f_0(y,\lambda)$ satisfies
\begin{align}
\label{m0.5}
    -\varepsilon\p_y^2 f_0(y, \lambda)   -\sigma_0 \varepsilon^{\frac{N+1}{N+3}}f_0(y, \lambda) + i(b(y) - \lambda)f_0(y, \lambda) = f_0^{\rm in}(y).
\end{align}
The most important difference is that we can modify the definition of $a$ in the good derivative \eqref{modd} so that
\begin{align}
    a_0(t,y) := \frac{1 - \chi_{\frac{\ell(0)}{2}}(y) } {b'(y)}
\end{align}
 This makes the bounds for $f_0(t,y)$ as good or better than $f_O(t,y)$. Since the bounds for $f_O(t,y)$ are consistent with Theorem \ref{thm:1}, the proof follows.

\subsection{Comparison with \cite{J23, BCJ24}}
\label{sec:compare}
We highlight some key technical differences between the proof of Theorem \ref{thm:1}  and the results in \cite{J23, BCJ24}. Since we treat any finite order of critical point, the pointwise fundamental solution bounds in Proposition~\ref{k:bdds} are necessarily more general. However, by utilizing the ``effective potential"~\eqref{effective}, we give a unified treatment for any order of critical point. Another point of departure is in the design of the good derivative. We crucially make use of the fact that~\eqref{thed} is \emph{independent} of the spectral parameter $\lambda$ appearing in the representation formula~\eqref{rep}. This allows us to fully exploit the desirable properties of the advection-diffusion equation \eqref{adv:diff} such as the $L^\infty$ enhanced dissipation from Theorem~\ref{edlinf}, which are hard to obtain directly from the representation formula with optimal regularity.

%% file: efunctiongluing.tex
\section{Slowest eigenmodes and asymptotic length scale}
\label{sec:efngluing}

In this section, we prove Theorem~\ref{thm:eigenfunctions}. First, we construct solutions to the eigenvalue problem
\begin{equation}
    \label{eq:eigenvalueproblem}
    [\lambda + ib(y)] f = \varepsilon \p_y^2 f \, .
\end{equation}
In Section~\ref{sec:uniquenesssec}, we asymptotically characterize the spectrum in the spectral `windows' $\{ \Re \lambda \geq - q \varepsilon^{1/2} \}$, $q \geq 1$, and the corresponding eigenfunctions, as exactly those we already constructed. In Section~\ref{sec:asymptoticdescription}, we apply the spectral information to describe the time-asymptotics of solutions to the original problem~\eqref{adv:diff} and thus prove Corollary~\ref{cor:spectralcorollary}. Finally, we prove higher-order in $\varepsilon$ asymptotics for the spectrum in Section~\ref{sec:higherorderasymptotics}.

\subsection{Gluing framework}

We assume the hypotheses of Theorem~\ref{thm:eigenfunctions}. In particular, $b : \T \to \R$ is a shear profile satisfying the standing assumptions with non-degenerate critical points. Let $(\gamma_j)_{j=1}^{J_0}$ be the critical points with $\gamma_1 < \cdots < \gamma_{J_0}$.

We will construct eigenvalues and eigenfunctions associated to a critical point $\gamma_j$, and, without loss of generality, we may assume that $j=1$. At this point, it will be convenient to assume that\footnote{Otherwise, we conjugate the equation: $(\bar{\lambda} - i b(y)) \bar{f} = \varepsilon \p_y^2 \bar{f}$, construct $\bar{\lambda}$ and $\bar{f}$, and conjugate back to obtain $\lambda$ and $f$.}
\begin{equation}
    b''(\gamma_1) > 0 \, .
\end{equation}

Let $\ell > 0$ be an $\varepsilon$-independent cut-off scale satisfying $\ell \leq \min_{1 \leq j \leq J_0} |\gamma_j - \gamma_{j+1}|/16$. We will choose $\ell$ sufficiently small during the course of the proof. Let $0 \leq \chi_I \in C^\infty_0(\R)$ with $\chi_I \equiv 1$ when $|y - \gamma_1| \leq 3\ell/2$ and $\chi_I \equiv 0$ when $|y-\gamma_1| \geq 2\ell$. Let $0 \leq \chi_O \in C^\infty_0(\R)$ with $\chi_O \equiv 1$ when $|y - \gamma_1| \geq \ell$ and $\chi_O \equiv 0$ when $|y - \gamma_1| \leq \ell/2$. Crucially, the cut-offs have non-trivial overlap in the sense that
\begin{equation}
    \label{eq:supportproperty}
\begin{aligned}
{\rm supp} \, \p_y \chi_I \subset \{ \chi_O \equiv 1\} \, , \quad {\rm supp} \, \p_y \chi_O \subset \{ \chi_I \equiv 1 \} \, ,
\end{aligned}
\end{equation}
see Figure~\ref{fig:cutoffs}.

\begin{figure}
    \centering
    \includegraphics[width=0.65\linewidth]{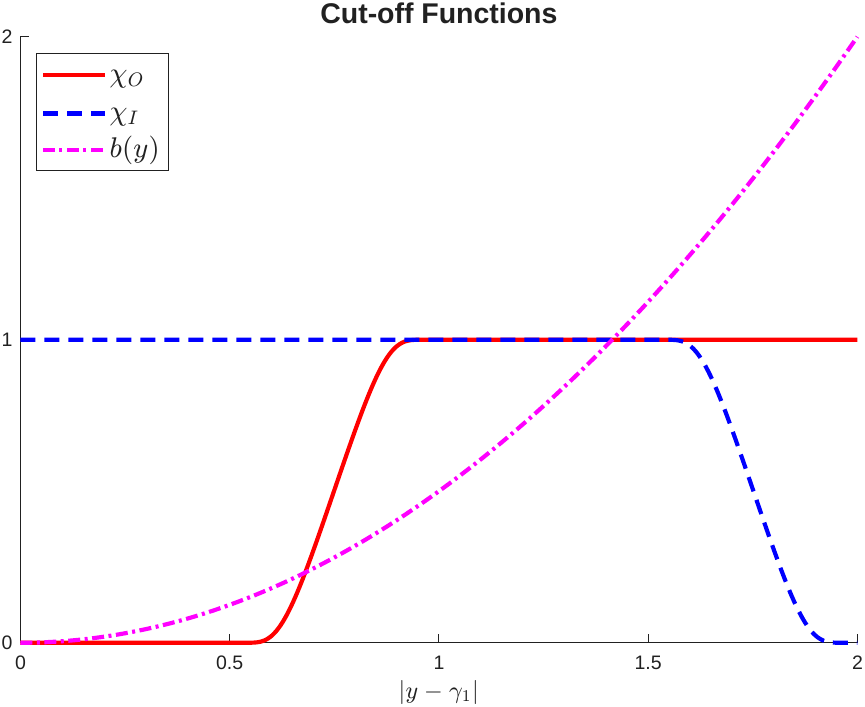}
    \caption{Cut-off functions $\chi_O$, $\chi_I$ with $\ell=1$ demonstrating overlap and shear profile $b(y) = (y-\gamma_1)^2/2$.}
    \label{fig:cutoffs}
\end{figure}

We make the ansatz
\begin{equation}
\label{gluesatz}
    f = f_I \chi_I + g \chi_O \, ,
\end{equation}
where $f_I : \R \to \R$ and $g : \R \to \R$.\footnote{Due to the overlapping of $\chi_I$ and $\chi_O$, the above ansatz is not built from a unique decomposition. This is characteristic of gluing constructions.} We plug this ansatz into~\eqref{eq:eigenvalueproblem}:
\begin{equation}
    \label{eq:originalequalitytoimply}
    [\lambda + ib(y)] (f_I \chi_I + g \chi_O) = \varepsilon \p_y^2 (f_I \chi_I + g \chi_O) \, .
\end{equation}
As we commute $\p_y^2$ through $\chi_I$ and $\chi_O$, we generate additional terms involving derivatives of $\chi_I$ and $\chi_O$. We have the freedom to organize all the terms into a system of two equations, which, if solved, will imply~\eqref{eq:originalequalitytoimply}:
\begin{align}
    \label{equationlabel1}
    (\lambda + ib(y)) f_I &= \varepsilon \p_y^2 f_I + \varepsilon (2 \p_y g \p_y \chi_O + g \p_y^2 \chi_O) \text{ on } {\rm supp} \, \chi_I \\
\label{equationlabel2}
    (\lambda + ib(y)) g &= \varepsilon \p_y^2 g + \varepsilon (2 \p_y f \p_y \chi_I + f \p_y^2 \chi_I) \text{ on } {\rm supp} \, \chi_O \, .
\end{align}
In obtaining~\eqref{equationlabel1}-\eqref{equationlabel2} from~\eqref{eq:originalequalitytoimply}, we have ``divided by" $\chi_I$ and $\chi_O$, respectively. This is possible because, by the support property~\eqref{eq:supportproperty} of $\chi_I$ and $\chi_O$, we have $\chi_I \equiv 1$ on ${\rm supp} \, \p_y \chi_O$ and vice versa. The terms are organized so that $g$ forces $f_I$ and vice versa.

We now extend the equations to the whole space. Let $0 \leq \tilde{\chi}_I \in C^\infty_0(\R)$ be an enlarged cut-off function with $\tilde{\chi}_I \equiv 1$ when $|y - \gamma_1| \leq 3\ell$ and $\tilde{\chi}_I \equiv 0$ when $|y - \gamma_1| \geq 4\ell$. To extend $b(y)|_{{\rm supp} \, \chi_I}$, let
 \begin{equation}
    \label{eq:hereismuhhat}
     \hat{b}(y) := b(\gamma_1) + \frac{b''(\gamma_1)}{2} (y-\gamma_1)^2 + \underbrace{\tilde{\chi}_I \left[ b(y) - b(\gamma_1) - \frac{b''(\gamma_1)}{2} (y-\gamma_1)^2 \right]}_{=: r_1(y)}
 \end{equation}
 be the extended inner potential. 
To extend $b(y)|_{{\rm supp} \, \chi_O}$, we view $b(y)|_{{\rm supp} \, \chi_O}$ as defined on the fundamental domain $[\gamma_1,\gamma_{J_0+1}]$, where $\gamma_{J_0+1} = \gamma_1 + 2\pi$. Let $\tilde{b}(y)$ be a smooth extension of $b(y)|_{{\rm supp} \, \chi_O}$ to $\R$ which is monotone on the extended regions, namely, satisfying $|\tilde{b}'(y)| \gtrsim \ell$ on $\{ y \leq \gamma_1 + \ell \}$ and $\{ y \geq \gamma_{J_0+1} - \ell \}$, as well as $\tilde{b}''(y) = 0$ on $\{ y \leq \gamma_1 \}$ and $\{ y \geq \gamma_{J_0+1} \}$. The extension can be constructed to satisfy $|\tilde{b}'''(y)| \lesssim \ell^{-1}$ on $\{ |y-\gamma_1| \leq \ell\}$ and $\{ |y-\gamma_{J_0+1}| \leq \ell\}$, with implicit constant depending on $b''(\gamma_1)$.

 The extended equations are
 \begin{align}
    \label{eq:gluingoneeqn}
    (\lambda + i\hat{b}(y)) f_I &= \varepsilon \p_y^2 f_I + \varepsilon (2 \p_y g \p_y \chi_O + g \p_y^2 \chi_O) \text{ on } \R \\
\label{eq:gluingtwoeqn}
    (\lambda + i\tilde{b}(y)) g &= \varepsilon \p_y^2 g + \varepsilon (2 \p_y f_I \p_y \chi_O + f_I \p_y^2 \chi_O) \text{ on } \R \, .
\end{align}

In conclusion, \emph{if $(f_I,g)$ satisfies~\eqref{eq:gluingoneeqn}-\eqref{eq:gluingtwoeqn}, then $f$ satisfies~\eqref{eq:eigenvalueproblem}}.

We now introduce a change of variables for equation~\eqref{eq:gluingoneeqn}. First, we write
\begin{equation}
\lambda + i\hat{b}(y) = [\lambda + ib(\gamma_1)] + i [\hat{b}(y) - b(\gamma_1)] \, .
\end{equation}
Second, we define\footnote{It would be more appropriate to write $Y_1$ instead of $Y$, but this would be too notationally heavy.}
\begin{equation}
    Y := \frac{y-\gamma_1}{\ell_1}
\end{equation}
\begin{equation}
    F_I(Y) := f_I(y)
\end{equation}
\begin{equation}
    \Lambda := \tau_1 [\lambda + ib(\gamma_1)] \, ,
\end{equation}
where
\begin{equation}
    \ell_1 := \varepsilon^{1/4} (b''(\gamma_1)/2)^{1/4} =: \tilde{\ell}_1 \varepsilon^{1/4} \, ,\quad \tau_1 := (\varepsilon b''(\gamma_1)/2)^{-1/2} =: \varepsilon^{-1/2} \tilde{\tau}_1 \, .
\end{equation}
We similarly write $G(Y) := g(y)$ and $X_\ast(Y) := \chi_\ast(y)$, $\ast \in \{ I, O \}$, etc. Finally,
\begin{equation}
    B(Y) := \ell_1^{-2} [\hat{b}(y) - b(\gamma_1)] =: Y^2 + \ell_1 R_1(Y) \, ,
\end{equation}
where $R_1(Y) = \ell_1^{-3} r_1(y)$ is the truncation of the third-order Taylor remainder, which implicitly depends on $\varepsilon$.
Recall that ${\rm supp} \, R_1 \subset B_{4\ell/\ell_1}$, which is of size $O(\varepsilon^{-1/4})$; eventually, we will choose $\ell$ small enough so that $\ell_1 R_1 \ll Y^2$. The inner equation~\eqref{eq:gluingoneeqn} now reads
\begin{equation}
    \label{eq:gluingoneeqnrewritten}
    (\Lambda + iY^2+ i \ell_1 R_1(Y)) F_I = \p_Y^2 F_I + 2 \p_Y G \p_Y X_O + G \p_Y^2 X_O \, .
\end{equation}

Let $\alpha \in \N_0$ and
\begin{align}
    \Lambda^{\rm app}(Y) &:= - e^{ i\pi/4} (2\alpha+1) \\
    \Phi^{\rm app}(Y) &:= e^{i\pi/16} G_\alpha(e^{i\pi/8} Y) \, ,
\end{align}
which solve
\begin{equation}
    (\Lambda^{\rm app} + iY^2) \Phi^{\rm app} = \p_Y^2 \Phi^{\rm app} \, .
\end{equation}
Let
\begin{equation}
\label{decomp}
    F_I := \Phi^{\rm app} + \Phi \, , \quad \Lambda := \Lambda^{\rm app} + M \, ,
\end{equation}
and $\lambda = \lambda^{\rm app} + \mu$. The system for $(\Phi,g)$ and the eigenvalue correction $\mu$ is therefore
\begin{align}
\label{eq:innereqforphi}
    &(\Lambda^{\rm app} + M) \Phi + i Y^2 \Phi + i \ell_1 R_1(Y) \Phi + (M + i \ell_1 R_1(Y)) \Phi^{\rm app} \\
    &\quad = \p_Y^2 \Phi + 2\p_Y X_O \p_Y G + \p_Y^2 X_O G \notag \\
\label{eq:outereqforgwithphi}
    &(\lambda^{\rm app} + \mu) g + i \tilde{b}(y) g = \varepsilon \p_y^2 g + \varepsilon [2 \p_y \chi_I \p_y (\phi^{\rm app} + \phi) + (\phi^{\rm app} + \phi) \p_y^2 \chi_I ] \, .
\end{align}
The main operator in~\eqref{eq:innereqforphi} is
\begin{equation}
\mathcal{L} := \p_Y^2 - \Lambda^{\rm app} - i Y^2  \, ,
\end{equation}
since both $M$ and $\ell_1 R_1(Y)$ will be small perturbations when $|M| \ll 1$ and $\ell \ll 1$. This operator has a one-dimensional kernel, spanned by $\Phi^{\rm app}$, and its range is co-dimension one; it will be invertible onto its range if we restrict the domain to functions satisfying the condition\footnote{This is not the unique such condition, as there is not a unique way to mod out the kernel; for example, $\langle \Phi, \Phi^{\rm app} \rangle = 0$ would also do.}
\begin{equation}
    \label{eq:kernelcond}
    \int \Phi \Phi^{\rm app} = \langle \Phi, \bar{\Phi}^{\rm app} \rangle = 0 \, .
\end{equation}
The main operator in~\eqref{eq:outereqforgwithphi} is
\begin{equation}
     i[\tilde{b}(y) - b(\gamma_1)] - \varepsilon \p_y^2 g \, ,
\end{equation}
which will be shown to be invertible by a combination of (i) monotone enhanced dissipation away from the critical points $\gamma_j$, $j \geq 2$, and (ii) different wave speeds $b(\gamma_j) \neq b(\gamma_1)$, $j \geq 2$, at those critical points. (This is where the convenient condition that the wave speeds $b(\gamma_j)$ are distinct is used.) These facts will be established below. \\

To solve the eigenvalue problem, we will apply the \emph{Lyapunov-Schmidt reduction}. The main idea is the following: The inner equation~\eqref{eq:innereqforphi}, considered schematically as $\mathcal{L} \Phi = {\rm RHS}$, is only solvable when the right-hand side belongs to the range of $\mathcal{L}$. This equation can therefore be considered as two equations:
\begin{equation}
    \mathcal{L} \Phi = Q {\rm RHS} \, , \quad (I-Q) {\rm RHS} = 0 \, ,
\end{equation}
where $Q$ is the $L^2$-orthogonal projection onto the range of $\mathcal{L}$,\footnote{equal to the orthogonal complement of the kernel of the adjoint, which is spanned by $\bar \Phi^{\rm app}$. This particular choice of projection $Q$ is convenient, due to the above characterization of the range, but it is not the unique choice. See~\cite[Chapter 1]{kielhöfer2011bifurcation} for an introduction to the Lyapunov-Schmidt reduction.} namely,
\begin{equation}
    \label{eq:defofQ}
    Q F = F - \langle F, \bar{\Phi}^{\rm app} \rangle \bar{\Phi}^{\rm app} / \| \bar{\Phi}^{\rm app} \|_{L^2}^2 \, .
\end{equation}
(By our choice of constraint~\eqref{eq:kernelcond}, we also have $Q\Phi = \Phi$.) The infinite-dimensional equation $\mathcal{L} \Phi = Q {\rm RHS}$, together with the outer equation, will be solvable for \emph{any} small $M$ and $\varepsilon$, to determine a pair $(\Phi,g)(\mu,\varepsilon)$. The remaining equation $(I-Q) {\rm RHS}(\mu,\varepsilon) = 0$ is one-dimensional, living in the cokernel of $\mathcal{L}$, and responsible for determining $\mu = \mu(\varepsilon)$. This reduction from infinite to finite dimensions is precisely the Lyapunov-Schmidt reduction. \\

In conclusion, we can formulate the problem for $(\Phi,g)$ as the system
\begin{subequations}
    \label{eq:totalreduced}
\begin{align}
    &(\Lambda^{\rm app} + M) \Phi + i Y^2 \Phi + i \ell_1  Q (R_3(Y) \Phi) + Q (M + i \ell_1 R_1(Y)) \Phi^{\rm app} \label{eq:innerreduced} \\
    &\qquad = \p_Y^2 \Phi + Q ( 2\p_Y X_O \p_Y G + \p_Y^2 X_O G ) \notag \\
    &\langle \Phi, \bar{\Phi}^{\rm app} \rangle = 0 \label{eq:innerconstraint} \\
    &(\lambda^{\rm app} + \mu) g + i \tilde{b}(y) g = \varepsilon \p_y^2 g + \varepsilon [2 \p_y \chi_I \p_y (\phi^{\rm app} + \phi) + (\phi^{\rm app} + \phi) \p_y^2 \chi_I ] \, , \label{eq:outerreduced}
\end{align}
\end{subequations}
which we expect to be uniquely solvable for $(\Phi,g)(\mu,\varepsilon)$ given any $|\mu| \ll \varepsilon^{1/2}$, together with a one-dimensional reduced equation
\begin{equation}
    \label{eq:onedimreducedeq}
    \begin{aligned}
        &(I-Q) [M + i \ell_1 R_1(Y)] \Phi \\
        &\quad = - (I-Q) [M + i \ell_1 R_1(Y)] \Phi^{\rm app} + (I-Q) ( 2\p_Y X_O \p_Y G + \p_Y^2 X_O G )
    \end{aligned}
\end{equation}
for determining $\mu$.

The Lyapunov-Schmidt method is to solve~\eqref{eq:innerreduced}-\eqref{eq:outerreduced} \emph{first} for $(\Phi,g)(\mu,\varepsilon)$ with arbitrary $\mu$, $\varepsilon$, and subsequently solve~\eqref{eq:onedimreducedeq} for $\mu$. The main term in the reduced equation is $M (I-Q) \Phi^{\rm app}$, which is simply equal to $M$. With this in hand, the reduced equation is equivalent to
\begin{equation}
    \label{eq:onedimreduceinintegralform}
    M + i \ell_1  \langle R_3(Y) \Phi^{\rm app},\bar{\Phi}^{\rm app} \rangle  = - \langle i \ell_1 R_1(Y) \Phi, \bar{\Phi}^{\rm app} \rangle + \langle 2\p_Y X_O \p_Y G + \p_Y^2 X_O G, \bar{\Phi}^{\rm app} \rangle \, ,
\end{equation}
which we solve by the implicit function theorem.

\subsection{The inner equation}
\label{sec:innerequation}

This section is adapted from~\cite[Section 2.1]{AlbrittonOzanskiColumns}; see also~\cite{DaviesComplexHarmonic}.

We consider the operators
\begin{equation}
    L_\zeta := \p_Y^2 - e^{i2\zeta} Y^2 : D(L_0) \subset L^2 \to L^2 \, , \quad \zeta \in (-\pi/2,\pi/2) \, ,
\end{equation}
with dense domain $D(L_0) = \{ F \in H^2 : Y^2 F \in L^2 \}$. The operator $L_0$ is the quantum harmonic oscillator, and we are particularly interested in $L_{\pi/4}$. The above operators are closed with compact resolvent and therefore have discrete spectrum. The spectrum of the $L_0$ consists of algebraically simple eigenvalues
\begin{equation}
    \mu_\alpha = -(2\alpha+1) \, , \quad \alpha \in \N_0 \, ,
\end{equation}
with corresponding eigenfunctions the $L^2$-normalized \emph{Hermite functions}
\begin{equation}
    \label{eq:Hermitefnshey}
    G_\alpha(Y) := \frac{1}{(2^\alpha \alpha! \sqrt{\pi})^{1/2}} e^{-\frac{Y}{2}} H_\alpha(Y) \, , \quad H_\alpha := e^{Y^2} \frac{d^\alpha}{dY^\alpha} e^{-Y^2} \, ,\quad \alpha \in \N_0 \, , \
\end{equation}
Let $Z = e^{i\zeta/2} Y$ (generalized Wick rotation). Formally,
\begin{equation}
    (L_\zeta)_Y = e^{i\zeta} (L_0)_Z \, .
\end{equation}
Therefore, the spectrum of $L_\zeta$ contains eigenvalues
\begin{equation}
    \Lambda_{\alpha,\zeta} = - e^{i\zeta} (2\alpha+1) \, , \quad \alpha \in \N_0 \, ,
\end{equation}
with corresponding eigenfunctions
\begin{equation}
    \label{eq:Phialphazetadef}
    \Phi_{\alpha,\zeta}(Y) = e^{i\zeta/4} G_\alpha(e^{i\zeta/2} Y) \, , \quad \alpha \in \N_0 \, .
\end{equation}
It is known and verified in~\cite{AlbrittonOzanskiColumns} that the above eigenvalues are algebraically simple and exhaust the spectrum. By the definition of~\eqref{eq:Phialphazetadef} and a shift of the integration contour, we have
\begin{equation}
\int_{\R} \Phi_\alpha^2(Y) \, dY = 1 \, .
\end{equation}


For $q_0 \in [0,1)$, define the weight
\begin{equation}
w(q_0,\zeta)(Y) := e^{q_0 Y^2 (\Re e^{i \zeta})/2}
\end{equation}
and weighted norm $\| F \|_{L^2_w} := \| w(q_0,\zeta) F \|_{L^2}$.

The following was essentially verified in~\cite[Lemma 2.2]{AlbrittonOzanskiColumns}.\footnote{There it is only explicitly stated for $\zeta = e^{i\pi/4}$, and the constraint $\int F \bar{\Phi}_{\alpha,\zeta} \, dY$ was used instead of~\eqref{eq:constraintinlemma}.}

\begin{lemma}[Inner problem solvability]
Let $\zeta \in (-\pi/2,\pi/2)$.

For $\Lambda \not\in \sigma(L_{\zeta})$ and $\Phi \in L^2$, the problem
\begin{equation}
    \label{eq:LalphazetaLambda}
(\Lambda - L_{\zeta}) F = \Phi
\end{equation}
has a unique solution $F \in D(L_0)$, and
\begin{equation}
    \label{eq:Fboundholds}
    \| |F| + |\p_Y F| + |\p_Y^2 F| + |Y|^2 F \|_{L^2} \lesssim \| \Phi \|_{L^2} \, .
\end{equation}
The implicit constant can be estimated locally uniformly in $\lambda \in \rho(L_\zeta)$.

For $\Lambda = \Lambda_{\zeta,\alpha} \in \sigma(L_{\zeta})$ and $\Phi \in {\rm range} \, (\Lambda_{\zeta,\alpha} - L_{\zeta})$, the problem~\eqref{eq:LalphazetaLambda} augmented with the condition
\begin{equation}
    \label{eq:constraintinlemma}
    \int F \Phi_{\alpha,\zeta} \, dY = 0
\end{equation}
has a unique solution $F \in D(L_0)$, and~\eqref{eq:Fboundholds} holds.

Let $q_0 \in [0,1)$. In either case, we have the weighted estimates
\begin{equation}
    \label{eq:innerweightedestimates}
    \| |F| + |\p_Y F| + |\p_Y^2 F| + |Y|^2 F \|_{L^2_w} \lesssim \| \Phi \|_{L^2_w} \, ,
\end{equation}
provided that the right-hand side is finite.
\end{lemma}



\subsection{Outer problem solvability}

Recall the outer equation
\begin{equation}
    \label{eq:outerproblem}
    (\lambda^{\rm app} + \mu) g + i \tilde{b}(y) g = \varepsilon \p_y^2 g + h \, ,
\end{equation}
where $\tilde{b} : \R \to \R$ is the smooth extension of the shear profile, as discussed below~\eqref{eq:hereismuhhat}, which has been modified to be monotone at $O(1)$ scales around the critical point $\gamma_1$. The estimate below will depend on this modification and $\ell$. (Ultimately, the smallness conditions on $\ell$ will be imposed before the smallness conditions on~$\varepsilon$.)

\begin{proposition}[Solvability of outer system]
    \label{prop:solvabilityofoutersystem}
For $|\mu| \ll 1$ with $|\Re \mu| \ll \varepsilon^{1/3}$ and $h \in L^2(\R)$, there exists a unique solution $g \in H^1(\R)$ to~\eqref{eq:outerproblem}, and it satisfies
\begin{equation}
    \varepsilon^{1/3} \| g \| + \varepsilon^{2/3} \| \p_y g \| \lesssim \| h \| \, . 
\end{equation}
\end{proposition}
Here and below, $\| \cdot \|$ is shorthand for $\| \cdot \|_{L^2(\R)}$.
\begin{proof}
Let $\mu \in \C$. By ellipticity, the solvability of the equation will follow from the \emph{a priori} estimates and the Fredholm theory, so we focus on the estimate.
Suppose that $g \in H^1(\R)$ is a solution~\eqref{eq:outerproblem} with $h \in L^2(\R)$. The strategy is to apply the monotone enhanced dissipation estimate in Appendix~\ref{sec:monotoneresolventestimate} away from the critical points $\gamma_j$, $j \geq 2$, and otherwise to utilize that $\tilde{b}(\gamma_j) \neq b(\gamma_1)$, $j \geq 2$.

Let $\chi_j$ be cut-off functions localized to scale $\ell$ around the critical points $\gamma_j$, $2 \leq j \leq J_0$. Define
\begin{equation}
f_j := \chi_j g \, , \quad 2 \leq j \leq J_0 \, ,
\end{equation}
which satisfies
\begin{equation}
    [\lambda^{\rm app} + \mu + i \tilde{b}(y)] f_j - \varepsilon \p_y^2 f_j = \varepsilon (2 \p_y \chi_j \p_y g + g \p_y^2 \chi_j) \, .
\end{equation}
We moreover want to obtain functions in the monotone regions, extended to the whole space. For $1 \leq j \leq J_0$, let
\begin{equation}
    g_j :=  (1-\sum_{k=2}^{J_0} \chi_k) g \, , \quad \gamma_j \leq y \leq \gamma_{j+1} \, .
\end{equation}
When $2 \leq j \leq J_0-1$, extend $g_j$ by zero outside $[\gamma_j,\gamma_j+1]$. For $j=1$, $g_1 = g$ for $y \leq \gamma_1$ and zero for $y \geq \gamma_2$. For $j=J_0$, $g_{J_0} = g$ for $y \geq \gamma_{J_0+1}$ and zero for $y \leq \gamma_{J_0}$. Then $g_j$ satisfies
\begin{equation}
    [\lambda^{\rm app} + \mu + i \tilde{b}_j(y)] f_j - \varepsilon \p_y^2 f_j = \varepsilon (2 \p_y \chi_j \p_y g + g \p_y^2 \chi_j) \, ,
\end{equation}
where $\tilde{b}_j$ is a smooth, monotone extension of $\tilde{b}|_{{\rm supp} \, g_j}$ to $\R$. (We could arrange for overlapping cut-offs, but it is not necessary here.)

For $j \geq 2$, we perform the basic energy estimates on $f_j$ (multiply the equation by $\overline{f}_j$ and extract real and imaginary parts):
\begin{equation}
    \label{eq:gjestimate}
    \inf_{{\rm supp} \, \chi_j} |\tilde{b}(y) - b(\gamma_1)| \, \| f_j \| + \varepsilon^{1/2} \| \p_y f_j \| \lesssim (|\lambda^{\rm app} + i b(\gamma_1)| + |\mu|) \| f_j \| + \varepsilon \| |\p_y g| + |g| \| + \| h \| \, .
\end{equation}
The contributions from the eigenvalue can be absorbed into the left-hand side when $|\mu|, \varepsilon \ll 1$.

For $j \geq 1$, we apply the monotone enhanced dissipation estimate~\eqref{eq:enhanceddissmonotone} to $g_j$:
\begin{equation}
    \label{eq:g0estimate}
\begin{aligned}
    \varepsilon^{1/3} \| g_j \| + \varepsilon^{2/3} \| \p_y g_j \| &\lesssim |\Re (\lambda^{\rm app} + \mu)|\, \| g_j \| + \varepsilon \| |\p_y g| + |g| \| + \| h \| \, .
\end{aligned}
\end{equation}
The contributions from the eigenvalue can be absorbed into the left-hand side when $|\Re \mu| \ll \varepsilon^{1/3}$, $\varepsilon \ll 1$.

We sum~\eqref{eq:gjestimate} and~\eqref{eq:g0estimate} to obtain
\begin{equation}
    \varepsilon^{1/3} \| g \| + \varepsilon^{2/3} \| \p_y g \| \lesssim \| h \| \, .
\end{equation}
\end{proof}







\subsection{Projected system}

Recall from~\eqref{eq:innereqforphi}-\eqref{eq:outereqforgwithphi} that
the \emph{projected system} for the remainder $(\Phi,g)$ is
\begin{equation}
    \label{eq:projectedsystem}
\left\lbrace
\begin{aligned}
    &(\Lambda^{\rm app} + MQ) \Phi + i Y^2 \Phi + i \ell_1  Q(R_3(Y) \Phi) = \p_Y^2 \Phi \\
    &\quad- Q(M\Phi^{\rm app} + i\ell_1 R_1(Y) \Phi^{\rm app})
     + Q(2\p_Y X_O \p_Y G + \p_Y^2 X_O G) \\
    &\int \Phi \Phi^{\rm app} \, dY = 0 \\
    &(\lambda^{\rm app} + \mu) g + i \tilde{b}(y) g = \varepsilon \p_y^2 g + \varepsilon [2 \p_y \chi_I \p_y (\phi^{\rm app} + \phi) + \p_y^2 \chi_I \p_y (\phi^{\rm app} + \phi)] \, .
    \end{aligned}
    \right.
\end{equation}

\begin{remark}
    One could also work with the condition $\int \Phi \overline\Phi^{\rm app} \, dY = 0$, as was done in~\cite{AlbrittonOzanskiColumns}. However, this would not be as convenient for the reduced equation~\eqref{eq:onedimreducedeq}.
\end{remark}

\begin{proposition}[Solvability of projected system]
    \label{pro:solvabilityprojectedsystem}
Let $\ell \ll 1$. Let $q_0 \in [0,1)$. Let $|M|, \varepsilon \ll 1$. Then there exists a unique solution
\begin{equation}
(\Phi,g)(\mu,\varepsilon) \in H^2_w(\R) \times H^1(\R)
\end{equation}
to the projected system~\eqref{eq:projectedsystem}. The solution satisfies
\begin{equation}
\label{errorbdds}
    \| \Phi \|_{H^2_w} \lesssim |M| + |\varepsilon|^{1/4} \, , \quad \| g \|_{H^1} \lesssim e^{-C^{-1}q_0 (\ell/\varepsilon^{1/4})^2} \, .
\end{equation}
\end{proposition}



\begin{proof}
The cut-off scale $\ell$ is chosen to satisfy $\ell_1 |R_1| \ll Y^2$, and eigenvalue upper bound is chosen small enough, to ensure that the operator
\begin{equation}
    L_I := (\Lambda^{\rm app} + M) + i Y^2 + i \ell_1  Q \circ R_1 - \p_Y^2 : D(L_0) \cap \{ (I-Q)\Phi = 0 \} \to {\rm range} \, Q \subset L^2
\end{equation}
is invertible. $L_I$ consists of a small perturbation $M + i \ell_1  Q \circ R_1$ to the invertible operator $\Lambda^{\rm app} + i Y^2 - \p_Y^2$, so $L_I$ is invertible and, for each fixed $\varepsilon$, depends analytically on $M$. Let $L_O^{-1}$ be the solution operator of the outer equation~\eqref{eq:outerproblem}, which also depends analytically on $\mu$ for each fixed $\varepsilon$. The projected equation can therefore be rephrased as
\begin{equation}
    \label{eq:projectedequationmatrix}
    \left( I + \begin{bmatrix}
        & L_I^{-1} A_{12} \\
        L_O^{-1} A_{21} &
    \end{bmatrix} \right) \begin{bmatrix}
        \Phi \\
        g
    \end{bmatrix}= \begin{bmatrix}
        L_I^{-1} B_1 \\
        L_O^{-1} B_2
    \end{bmatrix} \, ,
\end{equation}
where
\begin{equation}
    A_{12}g = -Q(2\p_Y X_O \p_Y G + \p_Y^2 X_O G)
\end{equation}
\begin{equation}
    A_{21} \Phi = - \varepsilon (\p_y^2 \chi_I \phi + 2 \p_y \chi_I \p_y \phi)
\end{equation}
\begin{equation}
    B_1 = - Q(M\Phi^{\rm app} + i\ell_1 R_1(Y) \Phi^{\rm app})
\end{equation}
\begin{equation}
    B_2 = \varepsilon (\p_y^2 \chi_I \phi^{\rm app} + 2 \p_y \chi_I \p_y \phi^{\rm app}) \, .
\end{equation}
The equation~\eqref{eq:projectedequationmatrix} will be solvable provided that
\begin{equation}
    \label{eq:similarbuteasiermyass}
\| L_I^{-1} A_{12} \|_{H^1(\R) \to H^2_w} \| L_O^{-1} A_{21} \|_{H^2_w \to H^1(\R)} \ll 1 \, .  
\end{equation}
Below, we prove the desired \emph{a priori} estimate on $(\Phi,g)$, with the understanding that a similar computation would verify~\eqref{eq:similarbuteasiermyass}. First,
    \begin{equation}
        \label{eq:innerestimateforprojected}
    \begin{aligned}
        \| \Phi \|_{H^2_w} \lesssim 
        \underbrace{\| (|M| + \ell_1 |R_1|) \Phi^{\rm app} \|_{L^2_w}}_{\lesssim |M| + \varepsilon^{1/4}} + \| w \|_{L^\infty({\rm supp}\, \p_Y X_O) } \varepsilon^{-1/8}  (\varepsilon^{1/4} \ell^{-1} \| \p_y g \| + \varepsilon^{1/2} \ell^{-2} \| g \|)
    \end{aligned}
    \end{equation}
    \begin{equation}
        \label{eq:outerestimateforprojected}
    \begin{aligned}
        \varepsilon^{1/3} \| g \| + \varepsilon^{2/3} \| \p_y g \| &\lesssim \varepsilon \| \ell^{-1} (|\p_y \phi^{\rm app}| + |\p_y \phi|) + \ell^{-2} (|\phi^{\rm app}| + |\phi|) \|_{L^2({\rm supp} \, \p_y \chi_I )} \\
        &\lesssim \varepsilon \| \frac{1}{w} \|_{L^\infty({\rm supp}\, \p_Y X_I) } \ell^{-2} \varepsilon^{1/8} (O(\varepsilon^{\infty}) + \| \Phi \|_{H^2_w})
    \end{aligned}
    \end{equation}
We substitute the estimate in~\eqref{eq:outerestimateforprojected} into~\eqref{eq:innerestimateforprojected} to obtain
\begin{equation}
\begin{aligned}
    &\| \Phi \|_{H^2_w} \lesssim |M| + \varepsilon^{1/4} + \| w \|_{L^\infty({\rm supp}\, \p_Y X_O) } \| \frac{1}{w} \|_{L^\infty({\rm supp}\, \p_Y X_1) } \times \\
    &\quad \times \varepsilon^{-1/8} (\varepsilon^{1/4} \varepsilon^{-2/3} \ell^{-1} + \varepsilon^{1/2} \varepsilon^{-1/3} \ell^{-2}) \varepsilon^{9/8} (O(\varepsilon^{\infty}) + \| \Phi \|_{H^2_w}) \, .
\end{aligned}
\end{equation}
Since\footnote{When $q_0 > 0$ in the Gaussian weight $w$, this leads to a big gain of $e^{-C^{-1}q_0 \varepsilon^{-1/2}}$, which is not necessary to close the estimates. Moreover, in the uniqueness proof, the roles of $\chi_I$ and $\chi_O$ will be switched, and it will be more convenient to work with the $q=0$ space.}
\begin{equation}
     \| w \|_{L^\infty({\rm supp}\, \p_Y X_O) } \| \frac{1}{w} \|_{L^\infty({\rm supp}\, \p_Y X_1) } \leq 1
\end{equation}
and the power of $\varepsilon$ is positive, we obtain
\begin{equation}
    \| \Phi \|_{H^2_w} \lesssim |M| + \varepsilon^{1/4} \, .
\end{equation}
Substituting this into~\eqref{eq:outerestimateforprojected} and using the gain from the weight when $q_0 > 0$, we arrive at
\begin{equation}
    \| g \|_{H^1} \lesssim e^{-C^{-1} q_0 (\ell/\varepsilon^{1/4})^2} \, ,
\end{equation}
as desired. \end{proof}

\subsection{Reduced system}
For $|M|, \varepsilon \ll 1$, we consider the solution $(\Phi,g)(\mu,\varepsilon)$ to the projected system. We aim to solve the one-dimensional reduced equation
~\eqref{eq:onedimreduceinintegralform}
\begin{equation}
    M + i \ell_1  \langle R_1(Y) \Phi^{\rm app},\bar{\Phi}^{\rm app} \rangle  = - i \ell_1 \langle R_1(Y) \Phi, \bar{\Phi}^{\rm app} \rangle + \langle 2\p_Y X_O \p_Y G + \p_Y^2 X_O G, \bar{\Phi}^{\rm app} \rangle \, ,
\end{equation}
for the eigenvalue correction $M(\varepsilon)$. We denote the right-hand side by $Z(M,\varepsilon)$; we already established that its dependence on $M$ is analytic for each fixed $\varepsilon$. We have
\begin{equation}
    \label{eq:solvemebyIFT}
    M + Z(M,\varepsilon) = O(\varepsilon^{1/4}) \, ,
\end{equation}
where $|Z(M,\varepsilon)| \lesssim \varepsilon^{1/4}$, and the new right-hand is independent of $M$. By the implicit function theorem,~\eqref{eq:solvemebyIFT} is uniquely solvable for small $M$ when $\varepsilon \ll 1$, and the solution is of size $O(\varepsilon^{1/4})$.\footnote{More specifically, since $Z$ is analytic in $M$, the bound $|Z(M,\varepsilon)| \lesssim \varepsilon^{1/4}$ on a ball $B_{2c}$ grants a bound $|\p_M Z(M,\varepsilon)| \lesssim \varepsilon^{1/4}$ on a ball $B_{c}$, so that the region of applicability of the inverse function theorem a lower bound which is uniform in small $|M|, \varepsilon$.}

In conclusion, we have proven the following theorem:

\begin{theorem}
    Let $\ell \ll 1$. There exists $c_0 > 0$ such that, for all $\varepsilon \ll 1$ and $q_0 \in [0,1)$, there exists a unique solution
    \begin{equation}
        (\Phi,g,M)(\varepsilon) \in H^2_w(\R) \times H^1(\R) \times B_{c_0} \, .
    \end{equation}
    to the gluing system~\eqref{eq:totalreduced}. The solution satisfies
    \begin{equation}
        \label{eq:sizeestimatesforgluing}
        \| \Phi \|_{H^2_w} \lesssim \varepsilon^{1/4} \, , \quad \| g \|_{H^1} \lesssim e^{-C^{-1}q_0 (\ell/\varepsilon^{1/4})^2} \, , \quad |M| \lesssim \varepsilon^{1/4} \, .
    \end{equation}
\end{theorem}
Upon rescaling,~\eqref{eq:sizeestimatesforgluing} implies the size assertions~\eqref{eq:evalremainderass} and~\eqref{evalexp} in Theorem~\ref{thm:eigenfunctions}. Also, the smallness requirement on $\varepsilon$ does not depend on $q_0$; the construction can be done with a single $q_0 \in [0,1)$ and the remaining weighted estimates obtained \emph{a posteriori}.






\subsection{Uniqueness}
\label{sec:uniquenesssec}
We now prove the uniqueness assertion in Theorem~\ref{thm:eigenfunctions}. Let $q \geq 1$ be fixed, $\varepsilon_k \to 0^+$, and $\lambda_k$ be in the spectrum of $L_{\varepsilon_k}$ with $\Re \lambda_k \in [-q \varepsilon_k^{1/2},0]$. By elementary energy estimates, we have that $\lambda_k$ must be in the range of $-i b$, and we therefore consider any subsequence (without relabeling) such that $\Im \lambda_k$ is converging.

First, we demonstrate
\begin{quote}
    (I) There exists $j$ such that $\Im \lambda_k \to - b(\gamma_j)$.
\end{quote}
Without loss of generality, we may suppose $j=1$.
\begin{proof}[Proof of (I)]
For the sake of contradiction, suppose instead that $\Im \lambda_k \to - b(\gamma)$ for some $\gamma \in \T$ and $b'(\gamma) \neq 0$. One may argue similarly to the proof of Proposition~\ref{prop:solvabilityofoutersystem}: Cut the function into the near-critical point regions (where the inverse of the operator $\lambda + ib(y) - \varepsilon \p_y^2$ is estimated because the imaginary part of the eigenvalue deviates from each $b(\gamma_j)$, $j \geq 2$, by $O(1)$) and the remaining regions (where the inverse of the operator $i[b(y) + \Im \lambda_k] - \varepsilon \p_y^2$ is estimated by monotone enhanced dissipation, and $\Re \lambda_k$ is considered as a perturbation).
\end{proof}

Second, we demonstrate
\begin{quote}
    (II) $\Lambda_k := \tilde{\tau}_1 \varepsilon_k^{-1/2} (\lambda_k + ib(\gamma_1))$ remains bounded.
\end{quote}

    For this and the remainder of the proof, we require a decomposition of the solution. Define
\begin{equation}
    \label{eq:reversegluingdecomposition}
    f_I := \chi_I f \, , \quad g := \chi_O f \, , 
\end{equation}
where we omit the dependence on $k$. Then $(f_I,g)$ satisfy
\begin{equation}
    \label{eq:reversegluingsystem}
\begin{aligned}
    \Lambda_k F_I + i Y^2 F_I + i \ell_1 R_1(Y) F_I &= \p_Y^2 F_I - 2\p_Y X_I \p_Y G - \p_Y^2 X_I G \\
\lambda_k g + i \tilde{b}(y) g &= \varepsilon \p_y^2 g - \varepsilon (2 \p_y \chi_O \p_y f_I + \p_y^2 \chi_O f_I) \, .
\end{aligned}
\end{equation}
This system is akin to the previous gluing system~\eqref{eq:gluingoneeqnrewritten} and~\eqref{eq:gluingtwoeqn} except that the cut-off functions $\chi_I$ and $\chi_O$ have switched roles on the right-hand side and the signs of the error terms are reversed. To obtain~\eqref{eq:reversegluingsystem}, we use the overlap property that $f = g$ on ${\rm supp} \, \p_y \chi_I$ and $f = f_I$ on ${\rm supp} \, \p_y \chi_O$.

\begin{proof}[Proof of (II)]
    For the sake of contradiction, suppose instead that there exists a subsequence (without relabeling) such that $|\Im \lambda_k + b(\gamma_1)|/\varepsilon_k^{1/2} \to +\infty$.
    
\begin{lemma}
    \label{lem:auxedlem}
For $A \in \R$ and $R : \R \to \R$ satisfying
\begin{equation}
    \label{eq:derivativeboundonR}
    |R(Y)| \ll \langle Y \rangle^2 \, , \quad |\p_Y R(Y)| \lesssim \langle Y \rangle \, ,
\end{equation}
the solution $F \in H^2$ to the equation
\begin{equation}
    [i (Y^2 + R(Y) - A) - \p_Y^2] F = H \in L^2
\end{equation}
satisfies
\begin{equation}
    \| F \|_{L^2} \lesssim \langle A \rangle^{-1/3} \| H \|_{L^2} \, .
\end{equation}
\end{lemma}

\begin{proof}[Proof of Lemma~\ref{lem:auxedlem}]
    First, we integrate against $\bar{F}$ and extract the real part:
    \begin{equation}
        \label{eq:firstineqforlemma}
        \| \p_Y F \|^2 \leq \| F \| \| H \| \, .
    \end{equation}
    Second, we integrate against $(Y^2 + R(Y) - A) \bar{F}$ and extract the imaginary part:
    \begin{equation}
        \label{eq:secondequalityformelem}
    \begin{aligned}
        &\int (Y^2 + R(Y) - A)^2 |F|^2 + \underbrace{\Im \int (Y^2 +R(Y) - A) \p_Y F \p_Y \bar{F}}_{=0} \\
        &\quad + \Im \int (2 Y + R'(Y)) \p_Y F \bar{F} = \Im \int H (Y^2 +R(Y) - A) \bar{F} \, .
    \end{aligned}
    \end{equation}
    The equation~\eqref{eq:secondequalityformelem}, the bound~\eqref{eq:derivativeboundonR} on $R'(Y)$, and Cauchy's inequality yield
    \begin{equation}
        \label{eq:mygoodestforlem}
        \int (Y^2 + R(Y) - A)^2 |F|^2  \lesssim \| \langle Y \rangle F \| \| \p_Y F \| + \| H \|^2 \, .
    \end{equation}
    When $A \leq 0$, its contribution is advantageous, so we focus on the case when $A = L^2 \geq 0$, where $L \geq 0$. We also assume that $L \geq 2$ and revisit $L \in [0,2]$ afterward.
    
    The term $\int (Y^2 + R(Y) - L^2)^2 |F|^2$ in~\eqref{eq:mygoodestforlem} may lose coercivity at $Y \approx \pm L$, but this loss is ameliorated by the $\int |\p_Y F|^2$ term. Let $B(Y) = Y^2 + R(Y)$. Notably, by the assumption~\eqref{eq:derivativeboundonR}, we may assume that $|B(0)| \leq \delta$ and $(2-\delta) Y - \delta \leq B'(Y) \leq (2+\delta)Y + \delta$ when $Y \geq 0$. For each $L \geq 2$, there exists a unique $Y_+ \geq 0$ at which $B(Y_+) = L^2$. We additionally have $Y_+ \in [L/2,3L/2]$, and, for a length scale $D \in (0,1]$ to be optimized,
    \begin{equation}
        B(Y) - L^2 \gtrsim D (\langle Y \rangle + L) \text{ in } \{ |Y-Y_+| \geq D \} \cap \{ Y \geq 0 \} \, ,
    \end{equation}
    which we obtain from integrating $B'(Y)$ starting from $Y_+$ (and our preexisting knowledge of $B$ around $Y=0$).
    Therefore,
    \begin{equation}
        \label{eq:ypluscoerciviy}
        \int_{\{ |Y-Y_+| \geq 1 \} \cap \{ Y \geq 0 \}} D^2 (\langle Y \rangle^2 + L^2) |F|^2 \, dY \lesssim \| \langle Y \rangle F \| \| \p_Y F \| + \| H \|^2 \, .
    \end{equation}
    In the region $\{ |Y-Y_+| \leq 1 \} \cap \{ Y \geq 0 \}$, we employ the following functional (Poincar{\'e}) inequality:
    \begin{equation}
        \label{eq:poincareinequalitything}
    \int_{Y_+-D}^{Y_++D} |F|^2 \lesssim \int_{Y_++D}^{Y_++2D} |F|^2 + D^2 \int_{Y_+-D}^{Y_++2D} |\p_Y F|^2 \, ,
    \end{equation}
    which can be proved by the fundamental theorem of calculus, translation, and rescaling. Together,~\eqref{eq:poincareinequalitything},~\eqref{eq:ypluscoerciviy}, and~\eqref{eq:firstineqforlemma} yield
    \begin{equation}
        L^2 D^2 \int_{Y_+-D}^{Y_++D} |F|^2 \lesssim \| \langle Y \rangle F \| \| \p_Y F \| + \| H \|^2 + L^2 D^4 \| F \| \| H \| \, .
    \end{equation}
    A similar inequality holds in $\{ Y \leq 0 \}$. In total,
    \begin{equation}
        D^2 \int (\langle Y \rangle^2  + L^2) |F|^2 \lesssim \| \langle Y \rangle F \| \| \p_Y F \| + \| H \|^2 + L^2 D^4\| F \| \| H \| \, ,
    \end{equation}
    which yields
    \begin{equation}
        \int (\langle Y \rangle^2  + L^2) |F|^2 \lesssim (L^2 D^2 + D^{-4}) \| F \| \| H \| + D^{-2} \| H \|^2 \, .
    \end{equation}
    To optimize this inequality, we choose $D = L^{-1/3}$. This yields
    \begin{equation}
        \int (\langle Y \rangle^2  + L^2) |F|^2 \lesssim L \| F \| L^{1/3} \| H \| + L^{2/3} \| H \|^2 \, ,
    \end{equation}
    and, after we apply Cauchy's inequality,
    \begin{equation}
        \int (\langle Y \rangle^2  + L^2) |F|^2 \lesssim L^{2/3} \| H \|^2 \, .
    \end{equation}
    
    For the case $L \leq 2$, $B(Y) - L^2$ may fail to be coercive on the region $\{ |Y| \leq 10 \}$, and we apply a Poincar{\'e} inequality like~\eqref{eq:poincareinequalitything} to control this region.
\end{proof}

We apply the above lemma with $R(Y) = \ell_1 R_1(Y)$ and $A = - \Im \Lambda_k$, assuming $\ell \ll 1$, to the $F_I$ equation in~\eqref{eq:reversegluingsystem} to obtain\footnote{Recall that $\ell_1 R_1(Y)$ is the rescaled third-order Taylor remainder multiplied by a cut-off at scale $\varepsilon^{-1/2} \ell$.}
\begin{equation}
\begin{aligned}
    \| |F_I| + |\p_Y F_I| \|_{L^2} \lesssim o_{k \to +\infty}(1) ( |\Re \Lambda_k | \| F_I \|_{L^2} + \| \varepsilon^{1/2} |G| + \varepsilon^{1/4} |\p_Y G| \|_{L^2}) \, .
     \end{aligned}
     \end{equation}
     Since $|\Re \Lambda_k|$ is assumed $O(1)$, we can absorb its contribution into the left-hand side due to the $o_{k \to +\infty}(1)$ coefficient. This gives
     \begin{equation}
        \| |F_I| + |\p_Y F_I| \|_{L^2} \lesssim o_{k \to +\infty}(1) \varepsilon^{3/8} \| |g| + |\p_y g| \|_{L^2} \, .
\end{equation}
By the outer estimates in Proposition~\ref{prop:solvabilityofoutersystem}, we have
\begin{equation}
    \varepsilon^{1/3} \| g \|_{L^2} + \varepsilon^{2/3} \| \p_y g \|_{L^2} \lesssim \varepsilon (\varepsilon^{-1/8} \| F_I \|_{L^2} + \varepsilon^{-3/8} \| \p_Y F_I \|_{L^2}) \, .
\end{equation}
Combining these estimates, we have
\begin{equation}
    \| |F_I| + |\p_Y F_I| \|_{L^2} \lesssim o_{k \to +\infty}(1) \varepsilon^{3/8} \varepsilon^{-2/3} \varepsilon \varepsilon^{-3/8} \| |F_I| + |\p_Y F_I| \|_{L^2} \, 
\end{equation}
which yields $F_I \equiv 0$ and $g \equiv 0$ when $k$ is sufficiently large.
\end{proof}

We may therefore pass (without relabeling) to any convergent subsequence  $\Lambda_k \to \Lambda_\infty$. (Recall the definition of $\Lambda_k$ in the statement of (II).)

Third, we demonstrate
\begin{quote}
    (III) $\Lambda_\infty \in \sigma(L_{\pi/4})$.
\end{quote}
\begin{proof}[Proof of (III)]
For the sake of contradiction, suppose instead that $\Lambda_\infty \notin \sigma(L_{\pi/4})$. The proof is similar to (II). We rewrite the inner system as
\begin{equation}
        \Lambda_\infty + i Y^2 F_I - \p_Y^2 F_I = - [(\Lambda_k - \Lambda_\infty) + i \ell_1 R_1(Y)] F_I - 2\p_Y X_I \p_Y G - \p_Y^2 X_I G \, .
\end{equation}
The operator on the left-hand side is invertible, and by shrinking $\ell$ and taking $k$ large, the $F_I$ terms on the right-hand side may be considered a small perturbation. The proof is then completed as in (II).
\end{proof}

Finally, we demonstrate
\begin{quote}
    (IV) For sufficiently large $k$, $(\lambda_k,f_k)$ is equal (up to constant multiple of $f_k$) to an eigenvalue-eigenfunction pair $(\lambda,\phi)(\varepsilon_k)$ constructed in the previous section.
\end{quote}

\begin{proof}[Proof of (IV)]
Let $\Phi^{\rm app}$ be the normalized eigenfunction for $L_{\pi/4}$ corresponding to eigenvalue $\Lambda_\infty$. We multiply $f_k$ by a non-zero constant such that, without loss of generality, we have
\begin{equation}
    \int F_I \Phi^{\rm app} =: \iota = 1 \text{ or } 0 \, .
\end{equation}
(We will demonstrate that $\iota = 1$.)
Then we decompose
\begin{equation}
F_I = \iota \Phi^{\rm app} + \Phi \, , \quad \lambda = \lambda^{\rm app} + \mu
\end{equation}
with $\int \Phi \Phi^{\rm app} = 0$ and $\varepsilon^{-1/2} \mu \to 0$ as $k \to +\infty$. The remainder $(\Phi,g)$ solves the projected system
\begin{equation}
    \label{eq:projectedsystemreverse}
\left\lbrace
\begin{aligned}
    &(\Lambda^{\rm app} + MQ) \Phi + i Y^2 \Phi + i \ell_1  Q(R_1(Y,\varepsilon) \Phi) = \p_Y^2 \Phi \\
    &\quad- Q(M\Phi^{\rm app} + i\ell_1 R_1(Y,\varepsilon) \Phi^{\rm app})
     - Q(2\p_Y X_I \p_Y G + \p_Y^2 X_I G) \\
    &\int \Phi \Phi^{\rm app} \, dY = 0 \\
    &(\lambda^{\rm app} + \mu) g + i \tilde{b}(y) g = \varepsilon \p_y^2 g -  \varepsilon (2 \p_y \chi_O \p_y f_I + \p_y^2 \chi_O \p_y f_I) \, ,
    \end{aligned}
    \right.
\end{equation}
which is to be compared with~\eqref{eq:projectedsystem}. The difference is that the roles of $\chi_I$ and $\chi_O$ are reversed on the right-hand sides of the equations and the errors terms come with a minus sign. By following the proof of Proposition~\ref{pro:solvabilityprojectedsystem} with the weight $w=1$ and $\chi_O$ and $\chi_I$ switched, we obtain that the projected system has a unique solution $(\Phi,g)$ when $M$ and $\varepsilon$ are sufficiently small. Additionally,
\begin{equation}
    \| \Phi \|_{H^2} \lesssim \iota (|M| + \varepsilon^{1/4}) \, ,
\end{equation}
\begin{equation}
    \varepsilon^{1/3} \| g \|_{L^2} + \varepsilon^{2/3} \| \p_y g \|_{L^2} \lesssim \varepsilon^{7/8} \| \Phi \|_{H^2} + O(\varepsilon^{\infty}) \, .
\end{equation}
This demonstrates that $\iota = 1$, since otherwise $f = 0$. We also record
\begin{equation}
    \varepsilon^{-1/3} \| G \|_{L^2} + \varepsilon^{-1/4} \| \p_Y G \|_{L^2} \lesssim \varepsilon^{1/12} (|M| + \varepsilon^{1/4}) \, .
\end{equation}
We now plug this information into the one-dimensional reduced equation
\begin{equation}
    \label{eq:reducedequationuniqueness}
    \begin{aligned}
     &(I-Q)(M\Phi^{\rm app} + i\ell_1 R_1(Y) \Phi^{\rm app}) \\
     &\quad = - (I-Q)(M\Phi + i\ell_1 R_1(Y) \Phi) - (I-Q)(2\p_Y X_I \p_Y G + \p_Y^2 X_I G) \, ,   
    \end{aligned}
\end{equation}
which we rewrite as
\begin{equation}
        M + i \ell_1  \langle R_1(Y) \Phi^{\rm app},\bar{\Phi}^{\rm app} \rangle  = - i \ell_1 \langle R_1(Y) \Phi, \bar{\Phi}^{\rm app} \rangle - \langle 2\p_Y X_O \p_Y G + \p_Y^2 X_O G, \bar{\Phi}^{\rm app} \rangle \, .
\end{equation}
to be compared with~\eqref{eq:onedimreduceinintegralform}. Again, this is of the form
\begin{equation}
    M + Z(M,\varepsilon) = O(\varepsilon^{1/4}) \, ,
\end{equation}
where $|Z(M,\varepsilon)| \lesssim \varepsilon^{1/4}$, the right-hand is independent of $M$, and $Z$ depends analytically on $M$ for each fixed $\varepsilon$. Therefore, there is a unique solution $(\Phi,g,\lambda)(\varepsilon)$ satisfying $|M| \ll 1$. If we apply the decomposition~\eqref{eq:reversegluingdecomposition} to the eigenfunction we constructed, then this furnishes a solution to~\eqref{eq:projectedsystemreverse}. Therefore, solutions must be equal to the one constructed.
\end{proof}


\subsection{Asymptotic behavior}
\label{sec:asymptoticdescription}

The goal of this section is to explain the consequences of the spectral picture in Theorem~\ref{thm:eigenfunctions}, concerning $L_\varepsilon = \varepsilon \p_y^2 - ib(y)$, for the advection-diffusion equation~\eqref{adv:diff}.

Let $\lambda_0$ be any eigenvalue of $L_\varepsilon$ and $\phi(y)$ a corresponding eigenfunction. Then
\begin{equation}
     f^{\rm lin}(x,y,t) := e^{ix+\lambda t} \phi(y) + e^{-ix+\bar{\lambda} t} \bar{\phi}(y) = 2 \Re e^{ix+\lambda t} \phi(y)
\end{equation}
is a solution to the equation $\p_t f + b(y) \p_x f = \varepsilon \p_y^2 f$. Writing $\lambda_0 = a-ic$ ($c = - \Im \lambda_0)$, we have
\begin{equation}
    \label{eq:travelingwaves}
    f^{\rm lin}(x,y,t) = 2 e^{at} \Re e^{i(x-ct)} \phi(y) = e^{at} f^{\rm lin}(x-ct,y,0) \, .
\end{equation}
In particular, the corresponding dynamical solution $f^{\rm lin}$ has the form of a decaying traveling wave with amplitude $\sim e^{at}$ moving to the right with speed $c$.

Next, we analyze the asymptotics of a generic solution to~\eqref{adv:diff} with vanishing $x$-averages. To begin, we will isolate the slowest parts of the spectrum of the full evolution, which will belong to the $|k|=1$ modes. Indeed, the operator acting in the $k$th Fourier mode, $k \geq 1$, is precisely $-\kappa k^2 + k L_{\kappa/k}$, so that the parameter $\varepsilon = \kappa/k$ is \emph{smaller} the larger $k$ is, and therefore the spectral asymptotics are even \emph{better} for $k \geq 2$. Let
\begin{equation}
    \tau_0^{-1} := \min_j \sqrt{|b''(\gamma_j)|/2} \, .
\end{equation}
For $k \geq 1$ and $\kappa \ll 1$, we have
\begin{equation}
    \sup \Re \sigma(-\kappa k^2 + k L_{\kappa/k}) = -\kappa k^2 - k [ \tau_0^{-1} \sqrt{\frac{\kappa}{2k}} + O\Big( \frac{\kappa}{k} \Big)^{3/4} ] \, .
\end{equation}
(For $k \leq -1$, the analogous spectral gap is obtained by adjointness.) We therefore focus on the $k=1$ mode. By the above characterization of the spectrum, we have that, whenever $\kappa \ll 1$, there is at least one `slowest eigenvalue' of $L_\kappa$, that is, $\lambda$ with
\begin{equation}
\Re \lambda = s(L_\kappa) := \sup \Re \sigma(L_\kappa) \leq - c_0 \kappa^{1/2} \, ,
\end{equation}
$c_0 \ll 1$. Such an eigenvalue will be associated to a critical point $\gamma_j$ having $|b''(\gamma_j)| = \min_j b''(\gamma_j)$ and $\alpha=0$ in the above indexing convention, i.e., $\alpha$ is the eigenfunction index in~\eqref{eq:Hermitefnshey}. Let $K_0 \leq J_0$ be the number of critical points attaining the above minimum. Then, for a given $\kappa \ll 1$, there exist $K \in [1,K_0]$ `slowest' eigenvalues of $L_\kappa$. We re-index the critical points in a $\kappa$-dependent way such that the slowest eigenvalues are associated to the critical points $\gamma_j$, $1 \leq j \leq K$. Let $P_{\rm slow}$ be the spectral projection onto these eigenvalues. In particular, we have\footnote{The `little oh' for the $k=1$ mode can be deduced from the spectral mapping theorem in~\cite[Chapter IV, Corollary 3.11]{engel2000one} applied to $L_\kappa|_{\ker P_{\rm slow}}$.}
\begin{equation}
    \label{eq:decompositionoffintoslow}
    f(x,y,t) = \underbrace{e^{-t} e^{ix} [e^{tL_\kappa} P_{\rm slow} f^{\rm in}_1](y) + \text{ complex conjugate}}_{=: f_{\rm slow}(x,y,t) = e^{ix} f_{1,{\rm slow}} + \text{ cc}} + o_{t \to +\infty} (e^{s(L_\kappa) t}) \, ,
\end{equation}
and the `little oh' is measured in $L^2$. By the analysis around~\eqref{eq:travelingwaves}, we have demonstrated the assertion about generic decaying traveling waves in Corollary~\ref{cor:spectralcorollary}. The real $2K$-dimensional subspace $E_\kappa$ is that for which $P_{\rm slow} f^{\rm in}_1$ vanishes.

We now examine the asymptotics of the length scales
\begin{equation}
    \ell(t) := \frac{\| f(t) \|_{L^2}}{\| f(t) \|_{\dot H^1}} \, , \quad \bar{\ell}(t) := \frac{\| f(t) \|_{\dot H^{-1}}}{\| f(t) \|_{L^2}} \, ,
\end{equation}
for generic solutions. It is not difficult to demonstrate using~\eqref{eq:decompositionoffintoslow} that
\begin{equation}
\ell(t) \sim_{t \to +\infty} \frac{\| f_{\rm slow}(t) \|_{L^2}}{\| f_{\rm slow}(t) \|_{\dot H^1}} \, , \quad
    \bar{\ell}(t) \sim_{t \to +\infty} \frac{\| f_{\rm slow}(t) \|_{H^{-1}}}{\| f_{\rm slow}(t) \|_{L^2}} \, .
\end{equation}
(Technically, one must use the smoothing property of the equation to demonstrate that the `little oh' in~\eqref{eq:decompositionoffintoslow} can also be measured in $H^1$.)
We additionally have
\begin{equation}
    \| f_{\rm slow}(t) \|_{L^2_{x,y}}^2 = 2 \pi e^{-2(\kappa + s(L_\kappa))t} \| \widehat{P_{\rm slow} f_1^{\rm in}} \|^2_{L^2_y}
\end{equation}
\begin{equation}
    \| f_{\rm slow}(t) \|_{\dot H^{1}_{x,y}}^2 = 2 \pi e^{-2(\kappa + s(L_\kappa))t} \| \widehat{P_{\rm slow} f_1^{\rm in}} \|_{ H^1_y}^2
\end{equation}
\begin{equation}
    \| f_{\rm slow}(t) \|_{\dot H^{-1}_{x,y}}^2 = 2 \pi e^{-2(\kappa + s(L_\kappa))t} \| \widehat{P_{\rm slow} f_1^{\rm in}} \|_{H^{-1}_y}^2 \, ,
\end{equation}
where
\begin{equation}
\| f \|_{H^{-1}_y}^2 := \sup_{\varphi \in H^1_y} \frac{|\langle f, \varphi \rangle|^2}{\| \varphi \|_{L^2}^2 + \| \p_y \varphi \|_{L^2}^2} \, .
\end{equation}
This is a consequence of the Fourier multiplier characterization of $\| \cdot \|_{\dot H^{-1}_{x,y}}$,
\begin{equation}
\| f \|_{\dot H^{-1}_{x,y}}^2 = \| (|k|^2+|\eta|^2)^{-1/2} \hat{f}(k,\eta)\|_{L^2_{k,\eta}}^2 \, ,
\end{equation}
and the 2D Fourier transform carrying an extra multiple of $(2\pi)^{-1/2}$.



Let $g = P_{\rm slow} f^{\rm in}_1$. By construction,
\begin{equation}
    g = \sum_{k=1}^K c_k \phi_k \, , \quad \phi_k = \phi_{k,I} \chi_{k,I} + \phi_{k,O} \chi_{k,O} \, ,
\end{equation}
where the $\phi_k$ are the constructed eigenfunctions associated to critical points $\gamma_k$. Notably, $\| \phi_{k,I} \chi_{k,I} \|_{L^2} \approx \varepsilon^{1/8}$, in the sense of upper and lower bounds, whereas $\| \phi_{k,O} \|_{L^2} = O(\varepsilon^{\infty})$.
Then
\begin{equation}
    \| g \|_{L^2}^2 = \sum_{k=1}^K |c_k|^2 ( \| \phi_{k,I} \chi_{k,I} \|_{L^2}^2 + O(\varepsilon^{\infty})) \approx \sum_{k=1}^K |c_k|^2 \varepsilon^{1/4} \, .
\end{equation}
Next, we coarsely upper bound $\| g \|_{\dot H^1_y}$ and $\| g \|_{H^{-1}_y}$:
\begin{equation}
    \| g \|_{\dot H^1_y} \lesssim \sum |c_k| \| \phi_k \|_{\dot H^1_y} \lesssim \sum |c_k| \varepsilon^{3/8}
\end{equation}
\begin{equation}
    \| g \|_{H^{-1}_y} \lesssim \sum |c_k| \| \phi_k \|_{H^{-1}_y} \lesssim \sum |c_k| \| \phi_k \|_{L^1_y} \lesssim \sum |c_k| \varepsilon^{1/4} \, .
\end{equation}
For the corresponding lower bound on $\| g \|_{\dot H^1_y}$, it is enough to observe that $\| \p_\xi \Phi_{k,I} \|_{L^2(\{ |\xi| \leq 1 \})} \gtrsim 1$. For the lower bound on $\| g \|_{H^{-1}_y}$, we may look at $\varphi$ supported away from $\gamma_j$ ($j \neq k$) and satisfying
\begin{equation}
    \varphi \equiv 1 \text{ on } D_k := \{ \chi_{k,I} \equiv 1 \} \, ,
\end{equation}
in which case $\| \varphi \|_{H^1} \lesssim 1$ and
\begin{equation}
\begin{aligned}
    \langle g, \varphi \rangle &= c_k \langle \phi_{k,I}  , \varphi \rangle + c_k \langle \phi_{k,I} (\chi_{k,I} - 1), \varphi \rangle + O(\varepsilon^{\infty}) \max_{j=1,\hdots,K} |c_j| \\
    &= c_k \int_{D_k} \phi_{k,I}(y) \, dy + O(\varepsilon^{\infty}) \max_{j=1,\hdots,K} |c_j| \\
    &= c_k \tilde{\ell}_k \varepsilon^{1/4} \int_{\R} \Phi_{k,I}(Y) \, dY + O(\varepsilon^{\infty}) \max_{j=1,\hdots,K} |c_j| \\
    &= c_k \tilde{\ell}_k \varepsilon^{1/4} \int G_1(e^{i \pi/8} Y) \, dY + O(\varepsilon^{1/2}) \max_{j=1,\hdots,K} |c_j| \\
    &= c_k \tilde{\ell}_k \varepsilon^{1/4} + O(\varepsilon^{1/2}) \max_{j=1,\hdots,K} |c_j| \, .
\end{aligned}
\end{equation}
Since this was valid for every $k=1,\hdots,K$, and since the $\ell^1$ and $\ell^\infty$ norms on $\{ 1, \hdots, K \}$ are equivalent, we obtain the desired upper and lower bounds on the $\limsup$ and $\liminf$ of $\ell(t)$ and $\bar{\ell}(t)$. This completes the proof of Corollary~\ref{cor:spectralcorollary}.

\begin{remark}[Comments on projective quantities]
    \label{rmk:projectivized}
We see that there are difference between the projectivized (i.e., normalized) quantities $\ell$ and $\bar{\ell}$. On functions belonging entirely to the $k$th Fourier mode, both $\sim k$ as $k \to +\infty$. However, differences arise when the energy is split evenly between two modes $k_2 \gg k_1 = O(1)$, namely, $\ell \sim k_2$ whereas $\bar{\ell} \sim k_1$ as $k_2/k_1 \to +\infty$. What is more relevant here is the difference between $L^2$-projectivized quantities based on $\| \cdot \|_{\dot H^s}$ with $s > -1/2$ and $s \leq -1/2$: Schwartz functions belonging to $\dot H^{-1/2}$ necessarily have vanishing mean. This is ultimately why, in the above arguments, the quantity $\bar{\ell}(t)$ presents the same scaling in $\varepsilon$ as, e.g., $\| f(t) \|_{L^1}/\| f(t) \|_{L^2}$.
\end{remark}


\begin{remark}
    The spectral projectors onto the slowest eigenvalues are given by $\langle \cdot, \bar{\phi} \rangle/\| \bar{\phi} \|_{L^2}^2$, i.e., integration against the associated eigenfunction, suitable normalized. A consequence of this is that continuous initial data which vanish at the critical points will have small projections onto the slowest eigenmodes.
\end{remark}

\subsection{High-order spectral asymptotics}
\label{sec:higherorderasymptotics}

We now construct higher-order asymptotic expansions of the eigenvalues and eigenfunctions. Let $m \geq 0$ be the order of the expansion. Let
\begin{equation}
    b(y) - b(\gamma_1) = \sum_{k=0}^{m} \underbrace{\frac{1}{(k+2)!} b^{(k+2)}(\gamma_1) (y-\gamma_1)^{k+2}}_{=: b_k(y)} + \tilde{r}_{m+1}(y)
\end{equation}
be the Taylor expansion of $b(y) - b(\gamma_1)$ around $y=\gamma_1$ with remainder $\tilde{r}_{m+1}(y)$ of order $m+3$. Recall $\ell_1 = \tilde{\ell}_1 \varepsilon^{1/4}$ and consider also the change of variables 
\begin{equation}
    Y = \frac{y}{\ell_1} \, , \quad B_k(Y) = \ell_1^{-(k+2)} b_k(y)\, ,
\end{equation}
and $\tilde{R}_{m+1}(Y) = \ell_1^{-(m+3)} \tilde{r}_{m+1}(y)$. We additionally expand
\begin{equation}
    \Lambda = \sum_{k=0}^m \ell_1^k \Lambda_k + M \, , \quad F = \sum_{k=0}^m \ell_1^k \Phi_k + \Phi \, .
\end{equation}
The eigenvalue equation can be written
\begin{equation}
\begin{aligned}
    &\left( \sum_{k=0}^m \ell_1^k \Lambda_k + M \right) \left( \sum_{k=0}^m \ell_1^k \Phi_k + \Phi \right) + \left( \sum_{k=0}^{m} \ell_1^k B_k(Y) + \ell_1^{m+1} \tilde{R}_{m+1}(Y) \right) \left( \sum_{k=0}^m \ell_1^k \Phi_k + \Phi \right) \\
    &\quad = \p_Y^2 \left( \sum_{k=0}^m \ell_1^k \Phi_k + \Phi \right) \, .
\end{aligned}
\end{equation}
The equation at zeroth order is
\begin{equation}
    (\Lambda_0 + iY^2 - \p_Y^2) \Phi_0 = 0 \, .
\end{equation}
We therefore fix $(\Phi_0,\Lambda_0)$ a solution of this equation: $\Lambda = \Lambda_\alpha$ and $\Phi_0 = \Phi_\alpha$ for $\alpha \in \N_0$, as in Section~\ref{sec:innerequation}, with the normalization condition $\int \Phi_0^2= 1$.

The equation at $k$th order is
\begin{equation}
    (\Lambda_0 - L_{\pi/4}) \Phi_k = -  \sum_{j=1}^k \Lambda_j \Phi_{k-j} - i \sum_{j=1}^{k} B_j(Y) \Phi_{k-j} \, ,
\end{equation}
which we supplement with the normalization condition $\int \Phi_k \Phi_0 = 0$. This equation is solvable provided that the right-hand side is in the range of the operator $\Lambda_0 - L_{\pi/4}$:
\begin{equation}
    \sum_{j=1}^k \Lambda_j \int \Phi_{k-j} \Phi_0 = - i \sum_{j=1}^k \int B_j(Y) \Phi_{k-j} \Phi_0 \, .
\end{equation}
On the left-hand side, the $j<k$ terms vanish. Therefore,
\begin{equation}
    \label{eq:eigenvalueconditionlambdak}
    \Lambda_k := - i \sum_{j=1}^k \int B_j(Y) \Phi_{k-j} \Phi_0 \, .
\end{equation}
In this way, we may inductively solve for $(\Phi_k,\Lambda_k)$.

Define
\begin{equation}
    \Phi^{\rm app} = \sum_{k=0}^m \ell_1^k \Phi_k \, , \quad \Lambda^{\rm app} = \sum_{k=0}^m \ell_1^k \lambda_k
\end{equation}
\begin{equation}
    B^{\rm app}(Y) = \sum_{k=0}^m \ell_1^k B_k(Y) \, .
\end{equation}

We write the inner-outer gluing system as before except with the ansatz
\begin{equation}
    F_I = \Phi^{\rm app} + \Phi \, , \quad \Lambda = \Lambda^{\rm app} + M \, .
\end{equation}
Moreover, in the inner equation, we use the following extension of the shear flow:
\begin{equation}
    B(Y) = B^{\rm app}(Y) + \ell_1^{m+1} \underbrace{\tilde{R}_{m+1}(Y) \tilde{X}_I}_{=: R_{m+1}(Y)} \, .
\end{equation}
The resulting system is
\begin{equation}
\left\lbrace
\begin{aligned}
   &(\Lambda^{\rm app} + M) \Phi + i B^{\rm app} \Phi + i \ell_1^{m+1}  R_{m+1}(Y) ( \Phi^{\rm app} + \Phi) +  M \Phi^{\rm app} \\
   &\quad\quad = \p_Y^2 \Phi +  ( 2\p_Y X_O \p_Y G + \p_Y^2 X_O G ) \\
    & \langle \Phi, \bar{\Phi}_0 \rangle = 0\\ 
    &(\lambda^{\rm app} + \mu) g + i \tilde{b}(y) g = \varepsilon \p_y^2 g + \varepsilon [2 \p_y \chi_I \p_y (\phi^{\rm app} + \phi) + (\phi^{\rm app} + \phi) \p_y^2 \chi_I ] \, . 
\end{aligned}
\right.
\end{equation}
This system can be solved using the Lyapunov-Schmidt reduction as before. (However, it may be necessary to shrink $\ell$ and $\varepsilon$ depending on $m$.)

It is possible to solve for $\Lambda_k$ and $\Phi_k$ more explicitly in the following way. First, it is convenient to undo the generalized Wick rotation by writing $X = e^{i\zeta_0/2} Y$, $\zeta_0 = \pi/4$. The problem for the $k$th approximate solution is
\begin{equation}
    \label{eq:ibecomereal}
    (\hat{\Lambda}_0 - L_0) \hat{\Phi}_k = - \sum_{j=1}^k \hat{\Lambda}_j \hat{\Phi}_{k-j} - \sum_{j=1}^k B_j(X) \hat{\Phi}_{k-j} \, ,
\end{equation}
where
\begin{equation}
    \hat{\Phi}_k(X) = e^{ik \frac{\zeta_0}{2}} e^{-i\frac{\zeta_0}{4}} \Phi_k(Y) \, , \quad
    \hat{\Lambda}_k(X) = e^{i(k-2) \frac{\zeta_0}{2}} \Lambda_k \, .
\end{equation}
The above problem~\eqref{eq:ibecomereal} is entirely real-valued. The self-adjoint operator $L_0$ can be diagonalized in terms of Hermite functions. Notably, $\hat{\Phi}_0 = G_\alpha$, the $\alpha$th Hermite function, and each solve adds degree at most three, so $\hat{\Phi}_k$, $k \geq 1$, is expressible as a linear combination of $G_\beta$, $0 \leq \beta \leq \alpha+3k$, $\beta \neq \alpha$. The condition~\eqref{eq:eigenvalueconditionlambdak} ensures the solvability, i.e., that the projection of the right-hand side onto $G_\alpha$ is zero:
\begin{equation}
    \hat{\Phi}_k = - \sum_{0 \leq \beta \leq \alpha+3k, \beta \neq \alpha}  \frac{G_\beta(X)}{2(\beta-\alpha)} \sum_{j=1}^k  \hat{\Lambda}_j \langle \hat{\Phi}_{k-j}, G_\beta \rangle + \langle B_j \hat{\Phi}_{k-j}, G_\beta \rangle \, ,
\end{equation}
\begin{equation}
    \hat{\Lambda}_k := - \sum_{j=1}^k \int B_j \hat{\Phi}_{k-j} G_\alpha \, .
\end{equation}




%% file: Conclusion.tex
\section{Future questions}
    \label{sec:futurequestions}


We conclude by highlighting some questions regarding uniform mixing for passive scalars. 


One theme awaiting development is \emph{localized mixing and enhanced dissipation}.
Currently, Theorem~\ref{thm:1} captures the ``worst case scenario": We only prove that the solution mixes with a rate corresponding to the highest order critical point. 
For the inviscid problem, it is possible to prove
\begin{align}
\label{local}
    \norm{ \lb t \tilde{B}'(y)\rb f }_{H^{-1}} \lesssim \norm{f^{\rm in }}_{H^1} \, ,
\end{align}
where $\tilde{B}'$ is defined as $B$ from~\eqref{B} with $t^{-\frac{1}{m+1}}$ replacing $\epsilon^{\frac{m}{m+3}}$. In particular, this captures that away from critical points, where $|B'|$ is bounded away from zero, the solution is mixing at the monotone rate. Hence, the following question seems natural:
\begin{enumerate}
    \item[(I)] Does an analogue of the localized mixing estimate~\eqref{local} hold when $\kappa > 0$?
\end{enumerate}
Unlike the inviscid setting, the ``leakage" caused by the diffusion seems like it could prevent \eqref{local} from holding globally in time. Trying to adopt a similar approach to the one used to prove Theorem~\ref{thm:1} suggests that localized enhanced dissipation estimates of the form 
\begin{align}
\label{localed}
    \norm{e^{c \kappa^{1/3}|B'(y)|^{2/3} t } f(t)}_{L^\infty} \lesssim \norm{f^{\rm in}}_{L^\infty}
\end{align}
would be useful to have; however,~\eqref{localed} is \emph{false} globally in time: It fails for pure eigenstates, although it remains true beyond the $\varepsilon^{-1/2}$ enhanced dissipation time scale, due to the localization from the eigenfunction itself. See \cite[Theorem 2]{GLM24} for a statement of localized enhanced dissipation.  
It seems likely that a suitable refinement of the vector field method or a stochastic approach to shear mixing would be useful for further investigations in this direction. 

A related point which we encountered in our attempts to prove Theorem~\ref{thm:1} is that the known uniform mixing estimates are not obviously stable with respect to cutting off the initial data: $\| e^{tL} (f^{\rm in} \chi) \|_{H^{-1}} \lesssim C(t) \| f^{\rm in} \chi \|_{H^1}$, which may be large, depending on $\chi'$.\footnote{In the monotone case, it \emph{is} possible to prove estimates of the type $\| e^{tL_\varepsilon} (f^{\rm in} \chi_\ell) \|_{H^{-1}} \lesssim \langle t \rangle^{-1} \| f^{\rm in} \|_{H^1}$ for certain $\chi_\ell$ satisfying $\chi_\ell \equiv 1$ on $B_1$ and $\chi_\ell \equiv 0$ outside $B_{1+2\ell}$, where $\ell \in (0,1]$. The proof we have in mind requires running the vector field method on $f^{\rm in} \chi_\ell$.} However, in the monotone inviscid case, non-stationary phase behaves perfectly well with respect to such cut-offs. One way to view this is as a consequence of an estimate $\| e^{-itb(y)} f^{\rm in} \|_{W^{-1,\infty}} \lesssim \langle t \rangle^{-1} \| f^{\rm in} \|_{W^{1,1}}$, which is stronger than the $H^1 \to H^{-1}$ estimate, at least for compactly supported $f^{\rm in}$, and moreover stable with respect to cut-offs. We therefore pose the problem:
 \begin{enumerate}
     \item[(II)] Prove uniform $W^{1,1} \to W^{-1,\infty}$ mixing.
 \end{enumerate}
    
Mixing with non-degenerate critical points arises naturally in the Doi-Saintillan-Shelley model, a kinetic model of an active suspension~\cite{AlbrittonOhmStabilizing,CotiZelati2023,DCZGV2024}:
\begin{equation}
    \label{eq:DSS}
\begin{aligned}
    \p_t f + (u+p) \cdot \nabla_x f + \div_p^{\mathbb{S}^{d-1}} [(I-p\otimes p) \nabla_x u p f] &= \kappa \Delta_x f + \nu \Delta_p^{\mathbb{S}^{d-1}} f \\
    -\Delta u + \nabla \pi = \pm \int_{\mathbb{S}^{d-1}} f p \otimes p \, dp \, , \quad \div u &= 0 \, ,
\end{aligned}
\end{equation}
where $f=f(x,p,t) : \T^d \times \mathbb{S}^{d-1} \to [0,+\infty)$ is a number density of microswimmers. This model experiences a Landau damping phenomenon near the homogeneous, isotropic equilibrium $f \equiv \text{const.}$, which was justified for the linearization of~\eqref{eq:DSS} with $\nu=0$ in~\cite{AlbrittonOhmStabilizing,CotiZelati2023}. In~\cite{CotiZelati2023}, the authors moreover extended this to $\nu > 0$ with certain $\log t$ losses in the mixing rate. With the methods developed in the present work, the following should be possible:
\begin{enumerate}
    \item[(III)] Prove the sharp uniform-in-$\nu$ Landau damping estimates for the linearized Doi-Saintillan-Shelley model.
\end{enumerate}
The above works obtain better-than-expected decay for $\nabla_x u$ by applying better mixing estimates specific to initial data vanishing at the critical points, akin to eliminating the main term in the stationary phase asymptotics. We do not attempt such refinements here, but it would be interesting to have them.

Shear mixing arises generically in 2D steady flows, but thus far it has been most carefully studied in the simple setting of parallel shear flows.
\begin{enumerate}
    \item[(IV)] Prove uniform mixing and enhanced dissipation for passive scalars in vortices and Hamiltonian flows $u = \nabla^\perp \psi$.
\end{enumerate}
For Hamiltonian flows, the mean on streamlines may no longer be conserved, so the problem must be formulated properly. These investigations are comparatively still in their mathematical infancy. We recommend~\cite{DolceHamiltonian,BrueCotiZelatiHamiltonian2024} and references therein.



Finally, the spectral asymptotics in Theorem~\ref{thm:eigenfunctions} only describe the behavior near non-degenerate critical points.
\begin{enumerate}
    \item[(V)] Extend the eigenfunction asymptotics in Theorem~\ref{thm:eigenfunctions} to higher-order critical points $N \geq 2$ and vortices.
\end{enumerate}
One may also consider eigenfunctions asymptotically localized at boundaries (e.g., Neumann boundary conditions are considered in~\cite{Camassa2010}) and Morse critical points on manifolds.


%% file: funsolbds.tex
\appendix
\section{Fundamental solution bounds for Airy}
\label{sec:airy}
In this section we give a proof of Proposition \ref{k:bdds}. We break the proof up into several parts following \cite{J23, BCJ24}. We consider 
\begin{align}
\label{AiryKass}
    &-\varepsilon\p_y^2 K(y,z, \lambda) -\sigma_0 \varepsilon^{\frac{N+1}{N + 3}}K(y,z, \lambda)  + \alpha K(y,z, \lambda)+  i(b(y) - \lambda)K(y,z, \lambda) = \delta(y-z), \\
    &\alpha \geq 0, \sigma_0 > 0, \notag
\end{align}
where $\sigma_0$ will be determined in Lemma \ref{spec:gap}.
The effect of $-\sigma_0\varepsilon^{\frac{N+1}{N+3}}$ and $\alpha$ is simply to translate the spectrum left and right, respectively, in the complex plane. These terms play a relatively passive role in the subsequent estimates. For conciseness, we will abuse notation and consider  
\begin{align}
    \label{AiryK}
    -\varepsilon\p_y^2 K(y,z,\lambda) +  i(b(y) - \lambda)K(y,z, \lambda) = \delta(y-z),
\end{align}
and remark how to include them in the final estimates.

Recall $\chi_j$ from \eqref{chij}. Define $K_j(y,z,\lambda) := \chi_j(y)K(y,z, \lambda) $. We have that  $K_j(y,z, \lambda)$ satisfies
\begin{align}
    \label{locK}
    &-\varepsilon\p_y^2 K_j(y,z, \lambda) +  i(b(y) - \lambda)K_j(y,z,\lambda) \\
    &=  \chi_j(z)\delta(y-z) -2\varepsilon \p_y K(y,z) \chi_j'(y) -  \varepsilon K(y,z, \lambda)\chi_j''(y) 
\end{align}
For simplicity of notation, we will drop the $\lambda$ dependence on $K$ in the proofs but not in the statements of results. In the following, we will frequently use the notation 
\begin{align}
    I_L(z) := [z - L, z + L].
\end{align}

We begin with a functional inequality independent of the PDE which is sometimes referred to in the literature as a ``spectral gap" inequality. 
\begin{lemma}
\label{spec:gap}
 Let $f: \T \to \C$ and recall the role of $\sigma_0$ in the definition of $|B'(y)|$ from \eqref{B}. There exists  $\sigma_0 \in (0,1)$ such that, for all $\lambda \in \R$ and $f \in H^1(\T)$, we have
    \begin{align}
    \label{sg1}
         \norm{\varepsilon^{1/6}|B'(y)|^{1/3} f}_{L^2(\T)} \leq  \norm{|b(y) - \lambda|^{1/2} f}_{L^2(\T)}  +  \norm{\varepsilon^{1/2}\p_y f}_{L^2(\T)}.
    \end{align}
    
 \end{lemma}

\begin{proof}
Define 
\begin{align}
    f_j(y) := \chi_j(y) f(y)
\end{align}
where $\chi_j$ is defined in \eqref{chij}. We first demonstrate that the proof will follow (with a smaller $\sigma_0$) if we can prove it for $f_j$:
\begin{align}
    \norm{ \varepsilon^{1/6} |B'(y)|^{1/3}  f}_{L^2(\T)} &\leq \sum_{j \geq 0} \norm{ \varepsilon^{1/6} |B'(y)|^{1/3}  f_j}_{L^2(\T)} \\
    &\leq \sum_{j \geq 0 } \left(\norm{|b(y) - \lambda|^{1/2} f_j}_{L^2(\T)}  +   \norm{\varepsilon^{1/2}\p_y f_j}_{L^2(\T)}\right)\\
    & \lesssim \norm{|b(y) - \lambda|^{1/2} f}_{L^2(\T)}  +  \norm{\varepsilon^{1/2}\p_y f}_{L^2(\T)} + \sum_{j\geq 0}\norm{ \varepsilon^{1/2}f \p_y \chi_j}_{L^2(\T)}. 
\end{align}
Since we have 
\begin{align}
    \norm{ \varepsilon^{1/2}f \p_y \chi_j}_{L^2(\T)} \lesssim \varepsilon^{1/2}\norm{f}_{L^2(\T)} \ll \norm{ \varepsilon^{1/6} |B'(y)|^{1/3}  f}_{L^2(\T)}
\end{align}
for $\varepsilon$ sufficiently small depending on $\sigma_0$, it follows that we can absorb this term in the left-hand side. See Remark~\ref{rmk:notationconvention} for interpreting the $\lesssim, \ll$ notation.  With this simplification, we consider two cases: $j = 0$ and $j > 0$.  

Assume $j = 0$: For $\lambda$ satisfying $ \inf_{y \in \T} |\lambda-b(y)| \gtrsim 1 $, \eqref{sg1} follows immediately. We thus focus on $\lambda$ inside the range of $b$ (for values close to but not in the range of $b$ the proof is simpler since $|b(y) - \lambda|$ does not vanish). Fix $\lambda \in \text{range}\, b$, and let $\{ y_{n} \}$ be an enumeration of the points in $b^{-1}(\lambda) \cap \text{supp} f_0 $.  
We decompose the domain into regions around $\{ y_n \}$ as 
\begin{align}
\label{sg3}
    \int_\T |f_0|^2 = \sum_n  \int_{I_{L_0}(y_n)} |f_0|^2    + \int_{ (\cup_{n} I_{L_0}(y_n))^c }|f_0|^2.
\end{align}
For $y \in (\cup_{n} I_{L_0}(y_n))^c  $, $|b(y) - \lambda|$ does not vanish and we have that 
\begin{align}
\label{sg4}
     \int_{ (\cup_{n} I_{L_0}(y_n))^c }|f_0|^2 &= \int_{ (\cup_{n} I_{L_0})(y_n))^c } \frac{1}{|b(y) - \lambda|} |b(y) - \lambda||f_0|^2 \\
     &\lesssim \frac{1}{L_0}  \int_{ (\cup_{n} I_{L_0})(y_n))^c } |b(y) - \lambda||f_0|^2 \notag 
\end{align}
where we have used that, on the support of $f_0$, $|b'(y)|$ is bounded below and 
\begin{align}
\label{sg5}
    |b(y) - \lambda| \approx |y- y_n|. 
\end{align}
Now we estimate the integrals over $I_{L_0}(y_n)$. 
Recall that for any $g \in C^\infty(\T)$ which vanishes and for any interval $I \subset \T $, we have the inequality 
\begin{align}
\label{sg7}
    \int_{I} |g|^2  \lesssim  |I| \norm{g}_{L^2(\T)}\norm{\p_yg}_{L^2(\T)}.
\end{align}
Therefore,
\begin{align}
\label{sg8}
    \int_{I_{L_0}(y_n)} |  f_0|^2 \lesssim |L_0| \norm{  f_0}_{L^2(\T)}  \norm{ \p_yf_0}_{L^2(\T)} .
\end{align}
Combining \eqref{sg4} and \eqref{sg8} yields
\begin{align}
\label{sg9}
    \int_{\T} |f_0|^2 \lesssim  \norm{\p_y f_0}_{L^2(\T) }^2  L_0^2  +  \frac{1}{L_0} \int_{ \T  }  |b(y) - \lambda||f_0|^2. 
\end{align}
 Choosing 
\begin{align}
    L_0 = \varepsilon^{1/3}
\end{align}
yields
\begin{align}
\label{sg10}
    \varepsilon^{1/3}\norm{f_0}_{L^2(\T)}^2 \lesssim \varepsilon \norm{\p_y f_0}_{L^2(\T)}^2 + \int_{\T} |b(y) - \lambda||f_0|^2,
\end{align}
which completes the case $j = 0$.

Now, consider the case $j > 0$. Fix a critical point $\gamma_j$ which we take to be of order $m$. Let $\{y_n\}$ 
be an enumeration of the points in $b^{-1}(\lambda) \cap \text{supp } f_j$ (this has cardinality bounded by two due to the support assumptions of $f_j$). We will break the proof into cases depending on how close $\lambda$ is to $b(\gamma_j)$.  We  will refer to the set of spectral values $\lambda$ satisfying  $|b(\gamma_j) - \lambda| \ll \varepsilon^{\frac{m+1}{m+3}}$ as being in the ``viscous regime."   Assume that $\lambda$ is in the viscous regime.
Here, we let 
\begin{align}
    L_j = \varepsilon^{1/(m + 3)},
\end{align}
and we decompose 
\begin{align}
    \int_\T |f_j|^2 = \int_{I_{L_j}(\gamma_j)} |f_j|^2 + \int_{(I_{L_j}(\gamma_j))^c}|f_j|^2.
\end{align}
The integral over $I_{L_j}(\gamma_j)$ is estimated as in \eqref{sg8} to obtain 
\begin{align}
\label{sg16}
   \int_{I_{L_j}(\gamma_j)} |f_j|^2 \lesssim |L_j| \norm{  f_j}_{L^2(\T)}  \norm{ \p_yf_j}_{L^2(\T)}. 
\end{align}
We now estimate the integral over the complement as 
\begin{align}
    \int_{(I_{L_j}(\gamma_j))^c}|f_j|^2 &= \int_{(I_{L_j}(\gamma_j))^c}  \frac{|b(y) -b(\gamma_j)| }{|b(y) -b(\gamma_j)|}   |f_j|^2 \lesssim  \frac{1}{L_j^{m+1}}\int_{\T}  |b(y) -b(\gamma_j)| |f_j|^2  \\
    &\lesssim  \frac{1}{L_j^{m+1}}\int_{\T}  |b(y) - \lambda| |f_j|^2 + \frac{|b(\gamma_j) - \lambda|}{L_j^{m+1}} \int_{\T}   |f_j|^2. \notag
\end{align}
Combined with the estimate~\eqref{sg16} and Young's inequality, this implies\footnote{The implicit constant appearing in the definition of the viscous regime is fixed by the requirement that the term $\frac{|b(\gamma_j) - \lambda |}{L_j^{m + 1}}\norm{f_j}_{L^2(\T)}^2 $ can be made small enough so that \eqref{final} holds. }
\begin{align}
\label{final}
    \varepsilon^{\frac{m+1}{m+3}} \norm{f_j}_{L^2(\T)}^2 \lesssim \varepsilon\norm{\p_yf_j}_{L^2(\T)}^2 + \int_{\T} |b(y) - \lambda| |f_j|^2. 
\end{align}
Now we assume that $|b(\gamma_j) -\lambda| \gtrsim \varepsilon^{\frac{m+1}{m+3}}$. Proceeding as in \eqref{sg3}, we have 
\begin{align}
\label{sg11}
    \int_\T |f_j|^2 = \sum_n  \int_{I_{L_j}(y_n)} |f_j|^2    + \int_{ (\cup_{n} I_{L_j}(y_n))^c }|f_j|^2.
\end{align}
For $y \in (\cup_{n} I_{L_j}(y_n))^c  $,  we have
\begin{align}
\label{sg12}
     \int_{ (\cup_{n} I_{L_j}(y_n))^c }|f_j|^2 &= \int_{ (\cup_{n} I_{L_j}(y_n))^c } \frac{1}{|b(y) - \lambda|} |b(y) - \lambda||f_j|^2 \\
     &\lesssim \frac{1}{L_j  \min_n |b'(y_n)|}  \int_{ (\cup_{n} I_{L_j}(y_n))^c } |b(y) - \lambda||f_j|^2 \notag 
\end{align}
where we have used that $\min_n |b'(y_n)| > 0$ since we are assuming that we are outside of the viscous regime and 
\begin{align}
\label{sg13}
    |b(y) - \lambda| \gtrsim |y- y_n| \min_n |b'(y_n)|, \quad y \in (\cup_n I_{L_j}(y_n))^c \cap \text{supp} f_j.  
\end{align}
In the case where the cardinality of $\{ y_n \}$ is two, we note that  $|b'(y_1)| \approx |b'(y_2)|$, as can be seen from the Taylor expansion of $b$ around $\gamma_j$.

The integral over $I_{L_j}(y_n)$ can be estimated as 
\begin{align}
\label{sg14}
   \int_{I_{L_j}(y_n)} |f_j|^2 \lesssim |L_j| \norm{  f_j}_{L^2(\T)}  \norm{ \p_yf_j}_{L^2(\T)}.
\end{align}
Combining \eqref{sg12} and \eqref{sg14} yields
\begin{align}
\label{sg15}
    \int_{\T} |f_j|^2 \lesssim  \norm{\p_y f_j}_{L^2(\T) }^2  L_j^2  +  \frac{1}{L_j  \min_n |b'(y_n)| } \int_{ \T  }  |b(y) - \lambda||f_j|^2.
\end{align}
Choosing 
\begin{align}
    \label{eq:Ljexample}
    L_j \ll \frac{\varepsilon^{1/3}}{\min_n |b'(y_n)|^{1/3}},
\end{align}
yields\footnote{The implicit constant appearing in $L_j$ depends on the choice of constant needed to define the viscous regime. }
\begin{align}
   \norm{\varepsilon^{1/3} \min_n  |b'(y_n)|^{2/3}  f_j      }_{L^2(\T)}^2 \lesssim \varepsilon \norm{\p_y f_j}_{L^2(\T)}^2 + \int |b(y) - \lambda||f_j|^2.
\end{align}
On the support of $f_j$, we have  
\begin{align}
    \varepsilon^{1/3}|b'(y)|^{2/3}  \lesssim \varepsilon^{1/3} \min_n |b'(y_n)|^{2/3} + |b(y) - \lambda|,
\end{align}
where we have used that $\min_n |b'(y_n)| \approx \max_n |b'(y_n)| $. This completes the proof outside the viscous regime. Therefore, by taking $\delta_0$ sufficiently small, this completes the proof of the lemma. \end{proof}

\begin{remark}[$\lesssim, \ll$]
\label{rmk:notationconvention}
    We wish to clarify a notational convention used above. Consider non-negative numbers $A,B$.
    \begin{itemize}
        \item The notation $A \lesssim B$ as a conclusion means that there exists a constant $C > 0$ such that $A \leq CB$; the parameters on which the implicit constant depends may not be stated explicitly when they are clear enough.
        \item The notation $A \ll B$ as a hypothesis means that there exists a constant $c_0 > 0$ such that if $A \leq c_0 B$, then the conclusion holds; similarly, the dependence of $c_0$ may not be made explicit.
        \item The notation $A \ll B$ as a conclusion means that, for any $\varepsilon_0 > 0$, the parameters of the problem can be chosen\footnote{e.g., by shrinking $\varepsilon$ in~\eqref{eq:Ljexample}} such that $A \leq \varepsilon_0 B$.
        \item The notation $A \lesssim B$ as a hypothesis means that, for any $c > 0$, the conclusion holds under the assumption that $A \leq c B$; the constants in the proceeding analysis may depend on $c$.
    \end{itemize}

    The above analysis was divided into the ``viscous regime" $|b(\gamma_j) - \lambda| \ll \varepsilon^{\frac{m+1}{m+3}}$ and the complementary regime $|b(\gamma_j) - \lambda| \gtrsim \varepsilon^{\frac{m+1}{m+3}}$, where $\ll$ and $\gtrsim$ are interpreted as hypotheses. The prefactor defining the ``viscous regime" is shrunk to ensure some portion of the analysis requiring smallness goes through, while the analysis in the complementary region is not particularly sensitive to its defining prefactor.
\end{remark}

We now prove an important lemma which establishes control on quadratic in $K(y,z,\lambda)$ norms in terms of linear control at the pole $z$. 
\begin{lemma}
\label{gbl:bdds}
Let $K(y,z, \lambda)$ solve \eqref{AiryK}. Then
\begin{align}
    \norm{\varepsilon^{1/6} |B'(y)|^{1/3} K  }_{L^2(\T)}^2 + \norm{\varepsilon^{1/2} \p_y K}_{L^2(\T)}^2 + \norm{|b(y) -\lambda| K}_{L^2(\T)}^2 \lesssim |K(z,z,\lambda)|. 
    \end{align}
\end{lemma}
\begin{remark}
 As will be clear from the proof, for $K$ solving \eqref{AiryKass}, we have 
 \begin{align}
    \norm{\varepsilon^{1/6} |B'(y)|^{1/3} K  }_{L^2(\T)}^2  + \alpha \norm{K}_{L^2(\T)}^2 +  \norm{\varepsilon^{1/2} \p_y K}_{L^2(\T)}^2 + \norm{|b(y) -\lambda| K}_{L^2(\T)}^2 \lesssim |K(z,z,\lambda)|. 
    \end{align}
\end{remark}

\begin{proof}
Testing \eqref{AiryK} with $\Bar{K}(y,z)$ and taking the real part yields 
\begin{align}
\label{gb1}
    \varepsilon \int_\T |\p_y K|^2 \leq |K(z,z)|.
\end{align}
If $\lambda \notin \text{range}\, b$, then testing \eqref{AiryK} with $\Bar{K}$ and taking the imaginary part yields 
\begin{align}
\label{gb2}
   \int_\T |b(y) - \lambda| |K|^2 \leq |K(z,z)|,  
\end{align}
which implies the desired result using \eqref{sg1}. 
Now we handle the case with $\lambda \in \text{range } b$.
 Let $\{ y_n \}$ be an enumeration of the points in $b^{-1}(\lambda) \cap \text{supp } K_j$ where $K_j$ satisfies \eqref{locK}.
For $j >0$, let $\gamma_j$ be a critical point of order $m$.
Let $\phi_{j}$ be an approximation of the sign of $b(y) - \lambda$ satisfying the following properties\footnote{For the construction of $\phi_j$, when $|y_n - \gamma_j| \gg \varepsilon^{\frac{m}{m+3}}$,  then $\phi_j'$ is supported around the individual critical layers. For $|y_n - \gamma_j| \lesssim  \varepsilon^{\frac{m}{m+3}}$ the critical layers ``overlap" and instead the center of the support of $\phi_j'$ is at the critical point $\gamma_j$.}:
\begin{enumerate}
    \item 
    \begin{align}
    \label{cut0}
      |\phi_{j}| \leq 1
    \end{align}
    \item 
    \begin{align}
    \label{cut1}
        |\phi_{j}'| \lesssim \frac{(|b'(y_n)| + \varepsilon^{ \frac{m}{m+ 3}}   )^{1/3}}{   \varepsilon^{1/3}    }
    \end{align}
    \item 
    \begin{align}
    \label{cut2}
        |(b(y) - b(y_n))(\phi_j - \text{sgn}(b(y) - b(y_n))| \lesssim (|b'(y_n)| + \varepsilon^{\frac{m}{m + 3}})^{2/3} \varepsilon^{1/3}.
    \end{align}
\end{enumerate}  


Testing \eqref{locK} with $\phi_j  \Bar{K}_j$ and taking the imaginary part yields:
\begin{align}
   &\int_\T |b(y) - \lambda| |K_j|^2  \lesssim \int_\T (b(y) - \lambda)(\phi_j - \text{sgn}(b(y) - \lambda ))|K_j|^2  + \varepsilon\int_\T |\phi_j'| |\p_y K_j||K_j|\\
   &  +    \varepsilon \int_\T |\p_y K| |K| +  |K(z,z)|  \\
     &\lesssim \int_\T (b(y) - \lambda)(\phi_j - \text{sgn}(b(y) - \lambda ))|K_j|^2 + \tau \varepsilon  \int_\T |\phi_j'|^2 |K_j|^2    + C_\tau|K(z,z)| 
\end{align}
where $\tau \in (0,1)$ and we used \eqref{gb1}. 
By  modifying $\phi_j$ so it changes sign over a smaller region, for any $\theta \in (0,1)$, we have 
\begin{align}
   & \int_\T (b(y) - b(y_n))(\phi_j - \text{sgn}(b(y) - b(y_n) ))|K_j|^2  \leq  \theta   \varepsilon^{1/3} (|b'(y_n)| + \varepsilon^{\frac{m}{m+3}}   )^{2/3}   \int_{\text{supp} (\phi_j') } |K_j|^2  \\
   & \varepsilon  \int_\T |\phi_j'|^2 |K_j|^2 \leq C_\theta\varepsilon^{1/3} (|b'(y_n)| + \varepsilon^{\frac{m}{m+3}}   )^{2/3} \int_{ \text{supp}(\phi_j') } |K_j|^2.
\end{align}
For any $\tau' >0$, taking $\theta, \tau$ sufficiently small   yields 
\begin{align}
    \int_\T |b(y) - \lambda| |K_j|^2  \lesssim  \tau' \int_\T\varepsilon^{1/3} (|b'(y)| + \varepsilon^{\frac{m}{m+3}}   )^{2/3} |K_j|^2  + |K(Z,Z)|.
\end{align}
Using a similar proof for $K_0$, we can obtain 
\begin{align}
    \int_\T |b(y) - \lambda| |K_0|^2  \lesssim  \tau' \int_\T\varepsilon^{1/3} |K_0|^2  + |K(Z,Z)|.
\end{align}
Summing over $j$ yields
\begin{align}
    \int_\T |b(y) - \lambda| |K|^2  \lesssim \sigma_0^{-1}\tau' \int_\T \varepsilon^{1/3}  |B'(y)|^{2/3}|K|^2 + |K(Z,Z)|.  
\end{align}
 Using Lemma \ref{spec:gap} and choosing $\tau'$ and  $\varepsilon$ sufficiently small (recalling that  $\varepsilon \ll \varepsilon^{1/3} |B'(y)|^{2/3}$ for $\varepsilon \ll 1$) yields
\begin{align}
\label{gb3}
    \norm{\varepsilon^{1/6}|B'(y)|^{1/3} K}_{L^2(\T)}^2 \lesssim |K(z,z)|,
\end{align}
as desired. 
\end{proof}

Using Lemma \ref{gbl:bdds} we can now deduce pointwise ``amplitude" bounds on the solution of \eqref{AiryK} at the pole $z$.
\begin{lemma}
\label{lem:diag}
The solution of \eqref{AiryK} satisfies
   \begin{align}
    |K(z, z, \lambda)| \lesssim \frac{\varepsilon^{-1/2}}{ (  \varepsilon^{1/3} |B'(z)|^{2/3}  + |b(z) - \lambda|)^{1/2}  }.
\end{align}
\end{lemma}
\begin{remark}
    The solution of \eqref{AiryKass} satisfies 
    \begin{align}
         |K(z, z, \lambda)| \lesssim \frac{\varepsilon^{-1/2}}{ ( \alpha +   \varepsilon^{1/3} |B'(z)|^{2/3}  + |b(z) - \lambda|)^{1/2}  }.
    \end{align}
\end{remark}
\begin{proof}
For any open interval $I \subset \T$ and all $f \in H^1(I)$  we have the inequality 
    \begin{align}
    \label{d1}
        \norm{f}_{L^\infty(I)}^2 \lesssim \norm{f}_{L^2(I)}^2 |I|^{-1} + \norm{\p_yf}_{L^2(I)}^2 |I|. 
    \end{align}
Let $L$ be the length, to be determined later in the proof, of the interval $I_{L(z)}$. We have from Lemma \ref{gbl:bdds} that 
\begin{subequations}
    \begin{align}
    &\int_{I_{L}(z)} \varepsilon |\p_yK|^2 dy \lesssim |K(z,z)|, \label{d2}\\
   & \int_{I_{L}(z)} (\varepsilon^{1/3} |B'(y)|^{2/3} + |b(y) - \lambda|   )   |K|^2  \lesssim |K(z,z)|. \label{d2.5}
\end{align}
\end{subequations}
First assume that $z \in \text{supp } K_0$ defined in \eqref{locK}. Then $B'(y) \approx 1$. If $L$ satisfies 
\begin{align}
\label{d3}
    L \ll \varepsilon^{1/3}
\end{align}
then \eqref{d2.5} implies 
\begin{align}
\label{d4}
    \int_{I_L(z)} |K|^2 \lesssim \frac{|K(z,z)|}{\varepsilon^{1/3} + |b(z) - \lambda|  }. 
\end{align}
We also trivially have
    \begin{align}
    \label{d5}
        \int_{I_{L}(z)}  |\p_yK|^2 \lesssim \frac{|K(z,z)|}{\varepsilon}.
    \end{align}
Multiplying \eqref{d4} by $L^{-1}$ and \eqref{d5} by $L$ and then optimizing gives 
\begin{align}
\label{d6}
    L \approx \frac{ \varepsilon^{1/2}}{ (\varepsilon^{1/3}  + |b(z) - \lambda|)^{1/2} }
\end{align}
which satisfies \eqref{d3}, saturating at critical layers, by choosing a small enough implicit constant in the definition \eqref{d6} of $L$. Applying \eqref{d1} gives the desired result. 

Now consider the case where $z \in \text{supp }K_j $. Suppose that $\gamma_j$ is a critical point of order $m$ so that $B'(y) \approx |b'(y)| + \varepsilon^{\frac{m}{m+3}}$. If $L$ satisfies 
\begin{align}
    \label{d7}
    L \ll \frac{\varepsilon^{1/3}}{  |B'(z)|^{1/3}}, 
\end{align}
then \eqref{d2.5} implies 
\begin{align}
    \int_{I_L(z)} |K|^2 \lesssim \frac{|K(z,z)|}{\varepsilon^{1/3}|B'(z)|^{2/3} + |b(z) - \lambda|  }.
\end{align}
Proceeding as in \eqref{d6} yields
\begin{align}
    L \approx \frac{\varepsilon^{1/2}}{ (\varepsilon^{1/3} |B'(z)|^{2/3} + |b(z) - \lambda|)^{1/2}   }
\end{align}
which, again, is consistent with \eqref{d7} so long as we choose a sufficiently small implicit constant in \eqref{d7}. The result in this case then follows by \eqref{d1}. 
\end{proof}

In order to get bounds on the kernel away from the diagonal, we will use the so-called entanglement inequality \cite{J23}.

\begin{lemma}
\label{lem:ent}
Let $K$ satisfy \eqref{AiryK}. There exists $c_0\in (0,1)$ such that, for any  piecewise $C^1$ function $\varphi$ whose support does not contain $z$, we have
\begin{align}
\label{ent0}
    \int \left[|\varphi'(y)|^2 - c_0\varepsilon^{-1}  (\varepsilon^{1/3}|B'(y)|^{2/3}  +  |b(y) - \lambda| ) |\varphi(y)|^2\right] |K|^2 dy \geq 0. 
\end{align}
\end{lemma}
\begin{remark}
The bound for $K$ solving \eqref{AiryKass} is
    \begin{align}
\label{ent0.5}
    \int \left[|\varphi'(y)|^2 - c_0\varepsilon^{-1}  ( \alpha + \varepsilon^{1/3}|B'(y)|^{2/3}  +  |b(y) - \lambda| ) |\varphi(y)|^2\right] |K|^2 dy \geq 0.
\end{align}
\end{remark}
\begin{proof} 
    Testing equation \eqref{AiryK} with $\varphi^2 \Bar{K}$ and taking the  real part  yields 
    \begin{align}
    \label{ent1}
        \varepsilon \int \varphi^2 |\p_y K|^2 dy \lesssim \varepsilon \int |\p_yK| |\varphi| |\varphi'| |K|  dy  \lesssim \varepsilon \tau \int \varphi^2 |\p_y K|^2 + C_\tau \varepsilon \int |\varphi'||K|^2. 
    \end{align}
   We chose $\tau \in (0,1)$ sufficiently small to absorb the $\varepsilon \tau$ term on the left hands side and obtain
    \begin{align}
    \label{ent2}
        \int \varphi^2 |\p_y K|^2 dy \lesssim \int |K|^2 |\varphi'|^2.
    \end{align}
We now want to gain control on the potential $b(y) - \lambda$. The  proof proceeds similarly to Lemma \ref{gbl:bdds} except that we test \eqref{AiryK} with $\phi_j \varphi \bar{K}$. Proceeding as in Lemma \ref{gbl:bdds} from \eqref{gb2}, we can deduce that 
 \begin{align}
 \label{ent10}
        \int (\varepsilon^{1/3} |B'(y)|^{2/3}  +  |b(y) - \lambda|) |K|^2 \varphi^2 \lesssim \varepsilon \int |\varphi'|^2 |K|^2, 
    \end{align} 
    from which the desired result follows by choosing $c_0$ sufficiently small.  
\end{proof}
Before giving the proof of the fundamental solution bounds we need one more essential technical lemma. One can proceed directly to the proof of the fundamental solution bounds and refer back to the lemma when needed.   
\begin{lemma}
\label{lemlow}
    We have for all $y, z \in \T$ and $\lambda \in \R$ that 
    \begin{align}
    \label{ass}
        \int_z^y|b(x) - \lambda|^{1/2} dx \gtrsim |y-z| (|b(y) - \lambda|^{1/2} + |b(z) - \lambda|^{1/2}) + |y-z| \mathbbm{1}_{|y-z| \geq \sigma_\sharp}.
    \end{align}
\end{lemma}


\begin{proof}
    We make some preliminary reductions to simplify the proof.  We assume that $\lambda$ is in the range of $b$ as the case for $\lambda$ outside the range follows by similar but simpler arguments. For $\lambda$ away from one of the critical values, the proof follows from the monotone case \cite{JiaUniform2023}. Therefore, it suffices to consider $|\lambda| \ll \sigma_\sharp$ where we have assumed that the critical value is $0$.\footnote{The general case can be obtained by shifting $b$.} For $y,z$  an order 1 distance away from the set $b^{-1}(\lambda)$, we have that 
    \begin{align}
        \int_z^y|b(x) - \lambda|^{1/2} dx \gtrsim |y-z|.
    \end{align}
Therefore, it suffices to consider the case where $y,z$ are close to  one of the critical layers, and that this critical layer is close to a critical point (which recall we have taken to be 0). By shifting the argument of $b$, we may take the critical point to be 0. We only provide the proof for $N$ odd since the proof for $N$ even is similar. In the case where $N$ is odd, the leading order term in the Taylor expansion of $b$ around the critical point 0 is even. Therefore, we may assume that $b(x)\geq 0$ in a neighborhood of 0.\footnote{Replacing $b$ with $-b$ gives the case where $b$ is negative.} With these reductions, we have that there is a unique point $y_c \in [0, \sigma_\sharp]$ satisfying $b(y_c) = \lambda$. Finally, we assume that $y\geq |z|$. (The case where $-y\geq |z|$ follows by similar arguments). The proof will proceed via casework with the first case  being the most delicate. We consider the following cases: 
\begin{enumerate}
    \item $y_c \approx y \approx z$ 
    \item $y_c \approx y$, $z \approx -y_c$
    \item $ y \approx  y_c $ and $|z| \ll y$  
    \item  $y \ll y_c$
     \item  $y \gg y_c$
 \end{enumerate}
For Case 1, we write 
\begin{align}
    \int_z^y |b(x) - b(y_c)|^{1/2} dx &\approx |b'(y_c)|^{1/2}\int_z^y |x - y_c|^{1/2} dx \approx |b'(y_c)|^{1/2} |y_c|^{3/2}\\ &\approx |
    y -z| (|b(y) - b(y_c)|^{1/2} + |b(z) - b(y_c)|^{1/2}),
\end{align}
   where we have used that, on account of $y_c \gg |z -y_c| + |y - y_c|$, we have $b(x) - b(y_c) \approx b'(y_c)(x - y_c)$. For Case 2, we can find $\delta \in (0,1)$ such that for $ x \in [-\delta y_c, \delta y_c] $, we have 
   \begin{align}
       |b(x) - b(y_c)|^{1/2} \gtrsim  |b(y_c)|^{1/2}.
   \end{align}
Therefore, for some $\delta > 0$ depending only on $b$ and the implicit constant defining Case~2, we have 
\begin{align}
    \int_z^y|b(x) - b(y_c)|^{1/2} dx > \int_{- \delta y_c}^{\delta y_c } |b(x) - b(y_c)|^{1/2} dx \gtrsim b(y_c).
\end{align}
We also have 
\begin{align}
    |y-z| (|b(y) - \lambda |^{1/2} + |b(z) - \lambda|^{1/2} ) \lesssim b(y_c)
\end{align}
which completes the proof of this case. For Case 3, we have that 
$|b(y) -b(y_c)| + |b(z) - b(y_c)| \approx |b(z) - b(y_c)| \approx b(y_c)$. Also, $|y-z| \approx y \approx y_c$. Finally, for $x \in [z, y_c/2]$, $|b(x) - b(y_c)| \approx b(y_c)$. Therefore,
\begin{align}
    \int_z^y |b(x) - b(y_c)|^{1/2} dx &\gtrsim \int_z^{y_c/2} |b(x) - b(y_c)|^{1/2} dx \gtrsim |b(y_c)|^{1/2} y_c \\
    &\approx |y -z| (|b(y) - b(y_c)|^{1/2} + |b(z) - b(y_c)|^{1/2}). 
\end{align}
For Case 4, we have $x \in [z,y], |b(x) - b(y_c)| \approx b(y_c)$. Thus,
\begin{align}
    \int_z^y |b(x) - b(y_c)|^{1/2} dx &\gtrsim \int_z^{y} | b(y_c)|^{1/2} dx \gtrsim |b(y_c)|^{1/2} |y-z| \\
    &\approx |y -z| (|b(y) - b(y_c)|^{1/2} + |b(z) - b(y_c)|^{1/2}), 
\end{align}
as desired. 
For Case 5, we note  $|b(y) -b(y_c)| + |b(z)-b(y_c)| \approx |b(y)|$ (recall that we are assuming $z \leq y$). Furthermore, if $z \ll y$, then $|y-z| \approx y$ and it suffices to simply integrate over the interval $x \in [y/2,y]$ on which we have $|b(x) - b(y_c)| \approx b(y)$. Hence,
\begin{align}
     \int_z^y |b(x) - b(y_c)|^{1/2} dx &\gtrsim \int_{y/2}^{y} |b(x) - b(y_c)|^{1/2} dx \approx |y||b(y)| \\
     &\approx |y -z| (|b(y) - b(y_c)|^{1/2} + |b(z) - b(y_c)|^{1/2}).  
\end{align}
If $z \approx y$ then for $x \in [z,y],  |b(x) - b(y_c)| \approx b(x)$ and we obtain
\begin{align}
     \int_z^y |b(x) - b(y_c)|^{1/2} dx &\gtrsim \int_{z}^{y} |b(y)|^{1/2} dx \approx |y-z||b(y)| \\
     &\approx |y -z| (|b(y) - b(y_c)|^{1/2} + |b(z) - b(y_c)|^{1/2}).  
\end{align}

\end{proof}

With Lemmas \ref{lem:diag}, \ref{lem:ent}, and \ref{lemlow} we are now in position to prove the claimed fundamental solution bounds from Proposition \ref{k:bdds}.
\begin{proof}[Proof of Proposition \ref{k:bdds}]
As a consequence of Lemma \ref{lemlow}, by potentially taking a smaller value of $c_0$, for a fixed $z$, it suffices to prove the bounds for $y$ satisfying $|y-z|\leq \sigma_\sharp$. Therefore, for fixed $z$, by Assumption \ref{assump:b} there can be at most one critical point in $I_{\sigma_\sharp}(z)$. We introduce the length scale 
\begin{align}
\label{b1}
    \ell(x) :=  \frac{ \varepsilon^{1/2}}{ ( 
     \alpha + \varepsilon^{1/3}|B'(x)|^{2/3}  + |b(x) - \lambda|)^{1/2} }.
\end{align}
 For $y$ satisfying
    \begin{align}
\label{b2}    
        |y - z| \lesssim  \ell(z), 
    \end{align}
 \eqref{fs2} follows from Lemma \ref{lem:diag}. Next, assume that 
 \begin{align}
 \label{b3}
     |y - z| \gg  \ell(z).
 \end{align}
Assume that 
\begin{equation}
\label{b4}
    \int_{z + \ell(z)}^{y + 2 \pi  }\ell^{-1}(x)\ge  \int_{y}^{z - \ell(z)}\ell^{-1}(x).
\end{equation}
Then we can find $y_\ast\in[z+\ell(z),y + 2 \pi ]$ such that 
\begin{equation}
\label{b5}
     m_\ast:= c_0\int_{z + \ell(z)}^{y_*}\ell^{-1}(x)\,dx=  c_0\int_{y}^{z - \ell(z)}\ell^{-1}(x)\,dx. 
\end{equation}
We choose $\varphi$ from the entanglement inequality \eqref{ent0.5} satisfying the following conditions: 
\begin{align}
\label{b7}
    & \varphi'(x) =   c_0 \ell^{-1}(x) \varphi(x) \quad
    \text{ for }  x \in [z + \ell(z), y_*]; \notag\\
    & \varphi'(x) =   -c_0 \ell^{-1}(x) \varphi(x) \quad
    \text{ for }  x \in [y, z - \ell(z)], \quad \varphi(z + \ell(z))=\varphi(z - \ell(z)) = 1; \notag \\
    &\varphi''(x) \equiv 0, \quad\text{ for } x \in (z, z + \ell(z)) \cup (z - \ell(z),z) \cup(y_*, y - \ell(y) + 2\pi),\quad \varphi(z)=0.
\end{align}
Using \eqref{b7} and \eqref{ent0.5} we have 
\begin{align}
\label{b8}
   \int_{z- \ell(z)}^{z + \ell(z) } \left[  |\varphi'(x)|^2 - c_0^2
    \ell^{-2}(x)|\varphi(x)^2   \right] |K|^2 \,dx &\geq c_0^2 \varphi^2(y)\int_{y_*}^{y + 2 \pi } \ell^{-2}(x) |K|^2 \, dx \\
    &\gtrsim    c_0^2\varphi^2(y)\ell^{-2}(y)\int_{y - \ell(y) + 2 \pi }^{y +  2\pi }  |K|^2\,  dx \notag\\
    &= c_0^2\varphi^2(y)\ell^{-2}(y)\int_{y - \ell(y)  }^{y  }  |K|^2\,  dx. \notag
\end{align}
For $x \in [z- \ell(z), z + \ell(z)]$  we have 
\begin{align}
\label{b8.5}
     \int_{z-\ell(z)}^{z + \ell(z)} \left[  |\varphi'(x)|^2 - c_0^2
    \ell^{-2}(x)|\varphi(x)^2   \right] |K|^2 \,dx \lesssim \ell^{-2}(z) \int_{z-\ell(z)}^{z + \ell(z)}|K|^2 dx. 
    \end{align}
We can rearrange \eqref{b8} and use Lemma~\ref{lem:diag} to get 
\begin{align}
    \ell^{-1}(y)\int_{y - \ell(y)}^{y}|K|^2 dx &\lesssim \varphi^{-2}(y) \frac{\ell(y)}{\ell(z)}  \frac{1}{\ell(z)}\int_{z -\ell(z)}^{z + \ell(z)}|K|^2 dx\\
    &\lesssim \varphi^{-2}(y) \frac{\ell(y)}{\ell(z)} \varepsilon^{-2} \ell^2(z). \label{b8.6}
\end{align}
We have 
\begin{align}
\label{b9}
    \varphi(y) = \exp\left(   c_0\int_{y}^{z - \ell(z)} \ell^{-1}(x) \, dx   \right)  \geq  \exp(m_*).
\end{align}
Lemma \ref{lemlow}, the smoothness of $b$, and the size of $\ell(z)$  imply with a potentially smaller value of $c_0$ that
\begin{align}
\label{b10}
    m_* \geq \frac{c_0|y -z|}{\varepsilon^{1/2}} \left( |b(y) - \lambda|  + |b(z) - \lambda| + \varepsilon^{1/3}|B'(y)|^{2/3}  
 + \varepsilon^{1/3}|B'(z)|^{2/3}  + \alpha \right)^{1/2}.
\end{align}
Combining \eqref{b8.6}, \eqref{b10}, and using second derivative bounds on $K$ away from $y = z$, we have
\begin{align}
\label{b11}
    |K(y,z)|^2 &\lesssim \frac{\ell(y)}{\ell(z)} \exp\left( -2c_0\frac{|y-z|}{L(y,z,\lambda)} \right) \varepsilon^{-2}\ell^2(z).
\end{align}
We claim that 
\begin{align}
\label{b12}
    \frac{\ell(y)}{\ell(z)}\exp\left( -c_0\frac{|y-z|}{L(y,z,\lambda)} \right) \lesssim 1.
\end{align}
To see this, observe that if  $\ell(y) \lesssim \ell(z)$ then there is nothing to prove. Therefore, assume $\ell(y) \gg \ell(z)$. Since we have $|y -z| \gg \ell(z)$, we may assume that $|y-z| \gg \ell(y)$.\footnote{Indeed, if it were the case that $|y-z| \lesssim \ell(y)$, then we would have that $\ell(y) \approx \ell(z)$, which would be a contradiction.} Observe that we also have  
\begin{align}
    L(y,z,\lambda) \leq \ell(z).
\end{align}
Therefore,
\begin{align}
   \frac{\ell(y)}{\ell(z)}\exp\left( -c_0\frac{|y-z|}{L(y,z,\lambda)} \right)  \lesssim \frac{\ell(y)}{\ell(z)}\exp\left( -c_0\frac{\ell(y)}{\ell(z)} \right) \lesssim 1
\end{align}
as desired. Combining \eqref{b11} and \eqref{b12} then imply
\begin{align}
    |K(y,z)| &\lesssim \exp\left( -\frac{c_0}{2}\frac{|y-z|}{L(y,z,\lambda)} \right) \varepsilon^{-1}\ell(z).
\end{align}
This concludes the proof by taking $\frac{c_0}{2}$ in place of $c_0$. 
\end{proof}

%% file: monotoneresolventestimate.tex
\section{Monotone resolvent estimate}
\label{sec:monotoneresolventestimate}

In this section, we give a self-contained proof of the monotone resolvent estimate, which is used in Section~\ref{sec:efngluing}.

\begin{proposition}[Resolvent estimate]
    \label{pro:monotoneresolventestimate}
    Let $b : \R \to \R$ be Lipschitz and satisfy
    \begin{equation}
        \label{eq:monotonicity}
    0< m \leq b'(y) \leq M < +\infty \, , \quad \text{ for a.e. } y \in \R \, .
    \end{equation}
    Then the solution $f \in H^1(\R)$ to the resolvent problem
    \begin{equation}
        \label{eq:Cauchyproblemmonotone}
        i (b(y) + i \tau) f - \varepsilon \p_y^2 f = g \in L^2(\R) \, , \quad \tau \in \R
    \end{equation}
    satisfies
    \begin{equation}
        \label{eq:enhanceddissmonotone}
        \varepsilon^{1/3} B'^{2/3}  \| f \|_{L^2} + \varepsilon^{2/3} B'^{1/3} \| \p_y f \|_{L^2} \lesssim\| g \|_{L^2} \, ,
    \end{equation}
    where
    \begin{equation}
        B' := m \big(\frac{M}{m}\big)^{-2} \, .
    \end{equation}
    (Notice that $M \geq m$, so $M/m \geq 1$.)
\end{proposition}

\begin{proof}
We follow the proof of Proposition 5.1~(i) in~\cite{JiaUniform2023}. 
We may suppose that $\tau = 0$ after redefining $b(y)$. Moreover, since $b(y)$ is monotone~\eqref{eq:monotonicity}, we may suppose that $b(y) = 0$ after translating in $y$.

We perform energy estimates by multiplying with $\bar{f}$ and $b(y) \bar{f}$. First, 
\begin{equation}
\varepsilon \int |\p_y f|^2 + i \int b |f|^2 = \int g \bar{f} \, ,
\end{equation}
whose real part yields
\begin{equation}
    \label{eq:firstpartofmonotoneenergy}
\varepsilon^{1/2} \| \p_y f \|_{L^2} \leq \| f \|^{1/2}_{L^2} \| g \|^{1/2}_{L^2} \, .
\end{equation}
Second,
\begin{equation}
\varepsilon \int b |\p_y f|^2 + \varepsilon \int b' \p_y f \bar{f} + i \int b^2 |f|^2 = \int b g \bar{f} \, ,
\end{equation}
whose imaginary part yields
\begin{equation}
\begin{aligned}
    m^2 \int y^2 |f|^2 \les \int b^2 |f|^2 &\les \| g \|_{L^2}^2 + \varepsilon \int |b'| |\p_y f| |f|\\
    &\les \| g\|_{L^2}^2 + M\varepsilon^{1/2} (\varepsilon^{1/2} \| \p_y f \|_{L^2}) \| f \|_{L^2}
\end{aligned}
\end{equation}
and, after we apply~\eqref{eq:firstpartofmonotoneenergy},
\begin{equation}
    \label{eq:secondpartofmonotoneenergy}
\begin{aligned}
\| |y| f \|_{L^2} &\les m^{-1} \| g \|_{L^2} +  m^{-1} M^{1/2} \varepsilon^{1/4} \| g \|^{1/4}_{L^2} \| f \|^{3/4}_{L^2} \\
& \les m^{-1} \| g \|_{L^2} + m^{-1} M^{2/3} \varepsilon^{1/3} \| f \|_{L^2} \, .
\end{aligned}
\end{equation}
As in~\cite{JiaUniform2023}, we utilize the interpolation inequality
\begin{equation}
\begin{aligned}
\| f \|_{L^2} &\lesssim \varepsilon^{-1/4} \| |y| f \|_{L^2}^{1/2} \left( \varepsilon^{1/2} \| \p_y f \|_{L^2} \right)^{1/2} \\
&\overset{\eqref{eq:secondpartofmonotoneenergy}}{\lesssim} \varepsilon^{-1/4} m^{-1/2} ( \| g \|^{1/2}_{L^2} + M^{1/3} \varepsilon^{1/6} \| f \|^{1/2}_{L^2}) \| f \|^{1/4}_{L^2} \| g \|^{1/4}_{L^2} \, .
\end{aligned}
\end{equation}
After applying Young's inequality, we obtain the resolvent estimate
\begin{equation}
    \| f \|_{L^2} \lesssim \varepsilon^{-1/3} m^{-2/3} \big( 1 + \left( \sfrac{M}{m} \right)^{2/3} \big) \, .
\end{equation}
For the derivative estimate, we plug this back into~\eqref{eq:firstpartofmonotoneenergy}.
\end{proof}